%% file: Polygonspaces.tex
\documentclass[11pt]{amsart}

\usepackage[
paper=a4paper,
text={147mm,230mm},centering
]{geometry}

\iftrue
\makeatletter
\def\@settitle{%
  \vspace*{-20pt}
  \begin{flushleft}%
    \baselineskip14\p@\relax
    \normalfont\bfseries\LARGE
    \@title
  \end{flushleft}%
}
\def\@setauthors{
  \begingroup
  \def\thanks{\protect\thanks@warning}%
  \trivlist
  \large \@topsep30\p@\relax
  \advance\@topsep by -\baselineskip
  \item\relax
  \author@andify\authors
  \def\\{\protect\linebreak}%
  \authors
  \ifx\@empty\contribs
  \else
    ,\penalty-3 \space \@setcontribs
    \@closetoccontribs
  \fi
  \normalfont
  \endtrivlist
  \endgroup
}
\def\@setabstracta{%
    \ifvoid\abstractbox
  \else
    \skip@25\p@ \advance\skip@-\lastskip
    \advance\skip@-\baselineskip \vskip\skip@
    \box\abstractbox
    \prevdepth\z@ 
    \vskip-10pt
  \fi
}
\renewenvironment{abstract}{%
  \ifx\maketitle\relax
    \ClassWarning{\@classname}{Abstract should precede
      \protect\maketitle\space in AMS document classes; reported}%
  \fi
  \global\setbox\abstractbox=\vtop \bgroup
    \normalfont\small
    \list{}{\labelwidth\z@
      \leftmargin0pc \rightmargin\leftmargin
      \listparindent\normalparindent \itemindent\z@
      \parsep\z@ \@plus\p@
      
    }%
    \item[\hskip\labelsep\bfseries\abstractname.]%
}{%
  \endlist\egroup
  \ifx\@setabstract\relax \@setabstracta \fi
}
\def\section{\@startsection{section}{1}%
  \z@{-1.2\linespacing\@plus-.5\linespacing}{.8\linespacing}%
  {\normalfont\bfseries\large}}
\def\subsection{\@startsection{subsection}{2}%
  \z@{-.8\linespacing\@plus-.3\linespacing}{.3\linespacing\@plus.2\linespacing}%
  {\normalfont\bfseries}}
\def\subsubsection{\@startsection{subsubsection}{3}%
  \z@{.7\linespacing\@plus.1\linespacing}{-1.5ex}%
  {\normalfont\itshape}}
\def\@secnumfont{\bfseries}
\makeatother
\fi 

\usepackage{amssymb}
\usepackage{mathrsfs}
\usepackage[all]{xy}
\usepackage{graphicx,color,float}
\usepackage[bookmarks]{hyperref}
\usepackage{pinlabel}
\usepackage[textwidth=1.3in,color=yellow]{todonotes}
\usepackage{amsmath}
\usepackage{pb-diagram,pb-xy}
\usetikzlibrary{calc}
\usepackage{tikz}
\usetikzlibrary{arrows,matrix,positioning}
\usepackage{multirow, url}
\usepackage{stmaryrd}
\usepackage{mathtools}

\textwidth=\paperwidth \advance\textwidth by-2\oddsidemargin
\advance\oddsidemargin by-1in
\evensidemargin=\oddsidemargin
\calclayout

\theoremstyle{plain}
\newtheorem{theorem}{Theorem}[section]
 
\newtheorem{thmx}{Theorem}

\newtheorem{proposition}[theorem]{Proposition}
\newtheorem{lemma}[theorem]{Lemma}
\newtheorem{corollary}[theorem]{Corollary}

\theoremstyle{definition}
\newtheorem{question}[theorem]{Question}

\newtheorem{definition}[theorem]{Definition}
\newtheorem{example}[theorem]{Example}

\theoremstyle{remark}
\newtheorem{remark}[theorem]{Remark}

\renewcommand{\tilde}{\widetilde}

\newcommand{\ddge}{\rotatebox[origin=c]{270}{$\leq$}}

\newcommand{\C}{\mathbb{C}}

\newcommand{\Q}{\mathbb{Q}}

\newcommand{\bS}{\mathbb{S}}

\newcommand{\R}{\mathbb{R}}

\newcommand{\Z}{\mathbb{Z}}

\newcommand{\CP}{\mathbb{C}P}

\def\mcal{\mathcal}
\def\frak{\mathfrak}
\def\scr{\mathscr}

\numberwithin{equation}{section} \numberwithin{table}{section}

\def\to{\mathchoice{\longrightarrow}{\rightarrow}{\rightarrow}{\rightarrow}}
\makeatletter
\newcommand{\shortxra}[2][]{\ext@arrow 0359\rightarrowfill@{#1}{#2}}
\def\longrightarrowfill@{\arrowfill@\relbar\relbar\longrightarrow}
\newcommand{\longxra}[2][]{\ext@arrow 0359\longrightarrowfill@{#1}{#2}}

\makeatother
\numberwithin{equation}{section}

\begin{document}                                                                          
\title[Disk potential functions for polygon spaces]{Disk potential functions for polygon spaces}

\author{Yoosik Kim}
\address{Department of Mathematics and Institute of Mathematical Science, Pusan National University}
\email{yoosik@pusan.ac.kr}

\author{Siu-Cheong Lau}
\address{Department of Mathematics and Statistics, Boston University}
\email{lau@math.bu.edu}

\author{Xiao Zheng}
\address{The Institute of Mathematical Sciences, The Chinese University of Hong Kong}
\email{xiaozh259@gmail.com}

\begin{abstract}
We derive a Floer theoretical SYZ mirror for an equilateral and generic polygon space. The disk potential function of the monotone torus fiber of the caterpillar bending system is calculated by computing non-trivial open Gromov--Witten invariants from the structural result of the monotone Fukaya category, the topology of fibers of completely integrable systems, and toric degenerations. Then, combining the result with the work of Nohara--Ueda \cite{NoharaUedaGraclus} and Marsh--Rietsch \cite{MarshRietsch}, we obtain the disk potential functions of bending systems and produce a mirror cluster variety of type $A$ without frozen variables via Lagrangian Floer theory.
\end{abstract}


\maketitle
\setcounter{tocdepth}{1} 
\tableofcontents

\section{Introduction}

For a tuple $\mathbf{r} = (r_1, r_2, \cdots, r_{n+3})$ of positive numbers, a polygon space $\mcal{M}_\mathbf{r}$ is defined to be the moduli space of vectors $(v_1, v_2, \cdots, v_{n+3})$ in $\R^3$ with side length $\mathbf{r}$ and $\sum_{j=1}^{n+3} v_j = 0$ modulo orientation preserving isometries, that is, 
$$
\mcal{M}_\mathbf{r} = \left\{ \mathbf{v} = (v_1, v_2, \dots, v_{n+3}) \in \prod_{j=1}^{n+3} \bS^2(r_j) \colon \sum_{j=1}^{n+3} v_j = 0 \right\} \slash \mathrm{SO}(3).
$$ 
The polygon space can be viewed and constructed in many different ways. It has been broadly investigated from various aspects of geometry and topology. Also, there have been its applications to physics and chemistry as the polygon space provides a model for the moduli spaces of constraint brownian motions and polymers. 
 
The main purpose of this paper is to explore polygon spaces from the aspect of Lagrangian Floer theory and SYZ mirror symmetry. Each polygon space $\mcal{M}_\mathbf{r}$ carries a completely integrable system, called a \emph{bending system}, constructed by Kapovich--Millson  \cite{KapovichMillson}. Each component of the system generates a Hamiltonian bending $\bS^1$-action along the diagonal corresponding to the component. A triangulation of a polygon, a representative of an element of $\mcal{M}_\mathbf{r}$, obtained by adding non-crossing diagonals gives rise to a bending system, a completely integrable system consisting of moment maps of the bending actions. This paper concerns Lagrangian Floer theory of a Lagrangian torus orbit of this system. 

If the polygon space $\mcal{M}_\mathbf{r}$ is a toric manifold and a bending system is a toric integrable system, the study of the bending system boils down to the toric case, which has been extensively studied. Lagrangian Floer theory for toric fibers of toric manifolds was developed in Fukaya--Oh--Ohta--Ono \cite{FOOOToric1, FOOOToric2}, see also Woodward \cite{Woodward}. In particular, when the toric manifold is Fano, Cho--Oh in \cite{ChoOh} computed the disk potential function, which agrees with Givental--Hori--Vafa potential. The holomorphic disks of Maslov index two bounded by a toric fiber are in one-to-one correspondence with the facets of its moment polytope. Such a disk emanating from the inverse image of a facet is called \emph{basic}. The second named author with Chan, Leung, and Tseng in \cite{ChanLauLeungTseng} computed the disk potential function of  semi-Fano toric manifold.

In this paper, we shall study \emph{non-toric} bending systems. Especially, we are concerned with bending systems on the space of equilateral polygons with a fixed length. The given tuple $\mathbf{r}$ is called \emph{equilateral} if the lengths of all edges are equal, that is, $r_1 = \cdots = r_{n+3}$. In the case where $\dim \mcal{M}_\mathbf{r} \geq 4$, every bending system on $\mcal{M}_\mathbf{r}$ is \emph{not} toric because it always admits a fiber having an $\bS^2$-factor. The main part of this paper focuses on the disk potential function of the system, which can be employed to compute deformed Floer cohomology and to provide (a cluster chart of) Landau--Ginzburg mirror of the polygon space. 

A Floer theoretical SYZ mirror for the space of equilateral polygons is a Landau--Ginzburg (LG) model, which is obtained by gluing the charts arising from the moduli space of deformations of a bending torus fiber. To describe the mirror LG model systematically, consider the complex Grassmannian $\mathrm{Gr}(2, \C^{n+3})$
 of two planes in the complex vector space $\C^{n+3}$, regarding it as an embedded subvariety of $\CP(\wedge^2 \C^{n+3})$ via the Pl\"{u}cker embedding. We denote the Pl\"{u}cker coordinates for $\CP(\wedge^2 \C^{n+3})$ by $\{ p_{i,j} \mid 1 \leq i < j \leq n+3 \}$. Regard the variable $p_{i,j}$ in the affine chart $p_{1, 2} \neq 0$ as affine coordinate $p_{i,j}/p_{1,2}$ by setting $p_{i,j} \coloneqq p_{i,j}/p_{1,2}$. 
 
The given tuple $\mathbf{r}$ is called \emph{generic} if $\mcal{M}_\mathbf{r}$ is a smooth algebraic variety. The main theorem of this paper is then stated below.

\begin{thmx}
Suppose that $\mathbf{r} = (r_1, r_2, \cdots, r_{n+3})$ is equilateral and generic (and hence $n$ is {even}). Let 
$$
\check{X} \coloneqq (\mathrm{Gr}(2, \C^{n+3})  \backslash \{ p_{1,2} p_{2,3} \cdots p_{n+2,n+3} p_{1,n+3} = 0\} ) 
\cap \{ p_{1,2} = \cdots = p_{n+2,n+3} = p_{1,n+3} \}
$$
in $\CP(\wedge^2 \C^{n+3})$. Passing $\check{X}$ to a subvariety of the affine chart given by $p_{1, 2} \neq 0$, regard the variable $p_{i,j}$ as affine coordinate $p_{i,j}/p_{1,2}$ by setting $p_{i,j} \coloneqq p_{i,j}/p_{1,2}$. Then a Floer theoretical SYZ mirror of the polygon space $\mcal{M}_\mathbf{r}$ arising from the bending systems is a Landau--Ginzburg model consisting of the mirror space $\check{X}$ together with the superpotential 
\begin{equation}\label{equ_diskpotentialcomp}
W_\mathrm{Bend} (\mathbf{p}) = \sum_{i=1}^{n+3}p_{i, i +2} \colon \check{X} \to \C.
\end{equation}

Moreover, this mirror has a cluster structure of $A$-type without frozen variables. The seeds of the cluster variety are parametrized by the set of triangulations of the regular $(n+3)$-gon by adding non-crossing diagonals.
\end{thmx}

The LG model for the polygon space $\mcal{M}_\mathbf{r}$ in~\eqref{equ_diskpotentialcomp} with an equilateral and generic tuple $\mathbf{r} = (r_1, r_2, \cdots, r_{n+3})$ can be combinatorially written from the regular $(n+3)$-gon as follows. Label the vertices of the $(n+3)$-gon by $\{1, 2, \cdots, n+3\}$ with counter-clockwise orientation. Note that each line segment joining two vertices $i$ and $j$ corresponds to the variables $p_{i,j}$ of the affine chart of $\CP(\wedge^2 \C^{n+3})$. Then the superpotential~\eqref{equ_diskpotentialcomp} is equal to the summation of all variables corresponding to a red line in the first figure of Figure~\ref{fig_bendpotential}. 

\begin{figure}[h]
	\scalebox{0.85}{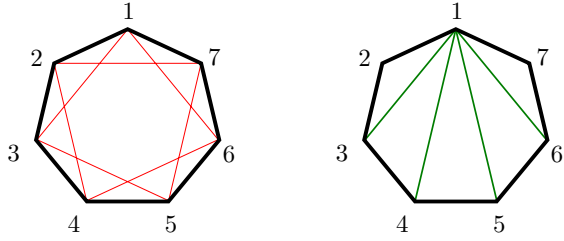}
	\caption{\label{fig_bendpotential} Bending superpotential and caterpillar bending systems}	
\end{figure}

The fiber at the center of each bending system is a monotone Lagrangian torus, see Section~\ref{subsection_diskpotential} for the precise definition of monotone Lagrangian submanifolds. The disk potential function of the monotone bending torus fiber of each bending system can be written from~\eqref{equ_diskpotentialcomp} as follows. Take a triangulation $\mcal{T}$ of the regular $(n+3)$-gon. Then the disk potential function
$$
W_\mathrm{Bend} (\mathbf{p}) |_{\mcal{U}_\mcal{T}} \colon \mcal{U}_\mcal{T} \to \C 
$$
is derived by restricting $W_\mathrm{Bend}$ to the cluster chart 
$$
\mcal{U}_\mcal{T} \coloneqq \mathrm{Spec} \left( \C [\{ p^\pm_{i,j} \colon (i,j) \in \mcal{T} \} ] \right) \subset \check{X}.
$$
Let $\mcal{T}_0$ be the triangulation obtained by adding all diagonals containing the vertex labeled by $1$ as in the second figure of Figure~\ref{fig_bendpotential}. The \emph{caterpillar bending system} is a bending system associated to $\mcal{T}_0$ and its disk potential function can be written as follows.

\begin{thmx}
The disk potential function of the caterpillar bending system of an equilateral and generic polygon space is
\begin{align}\label{equ_diskpotentialcaterpillar}
W_{\mcal{T}_0} = p_{1,3} + \frac{2}{p_{1,3}} + p_{1,n+2} + \frac{2}{p_{1,n+2}} + \sum_{j=3}^{n+1} \left( \frac{p_{1,j+1}}{p_{1,j}} + \frac{p_{1,j}}{p_{1,j+1}} + \frac{1}{p_{1,j}p_{1,j+1}} \right), 
\end{align}
which is defined on $(\C^*)^n$ with the coordinates $p_{13}, p_{14}, \cdots, p_{1, n+2}$ associated to $\mcal{T}_0$.

Moreover, the monotone caterpillar bending torus fiber cannot be displaced by any Hamiltonian diffeomorphisms.
\end{thmx}

The space of equilateral pentagons is isomorphic to a del Pezzo surface $dP_5$ of degree five, $\CP^2$ blown up at generic four points as an algebraic variety. Pascaleff--Tonkonog \cite{PascaleffTonkonog} computed the disk potential function of a Lagrangian torus constructed via Lagrangian surgery for $dP_5$. Their superpotential is equal to the disk potential function of the caterpillar bending system in~\eqref{equ_diskpotentialcaterpillar}. Recently, the scattering diagram and the disk potential of a special Lagrangian fibration in a non-monotone del Pezzo surface were computed by Bardwell-Evans--Cheung--Hong--Lin in \cite{Bardwell-EvansCheungHongLin}, which agrees with the scattering diagram of Gross--Pandharipande--Siebert and Gross--Hacking--Keel \cite{GrossPandharipandeSiebert, GrossHackingKeel}. Their chamber structure agrees with the chambers structure given by various bending systems. Gr{\"{a}}fnitz--Ruddat--Zaslow and You showed that the Carl--Pumperla--Siebert's LG superpotential in \cite{CarlPumperlaSiebert} counting tropical disks agrees with the open mirror map in \cite{GRZ22, You} in the toric case. The disk potential function should be equal to the LG superpotential obtained via tropical coutings.

At the level of completely integrable systems, Nohara--Ueda constructed a toric degeneration of bending systems in \cite{NoharaUedaGrasPoly}. By running their minimal model program, Charest and Woodward provide a method, which can detect torus orbits having non-trivial Floer cohomology in \cite{CharestWoodward}. Also, Woodward and Xu computed Floer cohomology of some non-torus fiber of polygon spaces in \cite{WoodwardXu}. 

Yet, the full disk potential function of non-toric bending systems of higher dimensional polygon spaces remains unknown due to the following difficulties. The major difficulty is to analyze contributions from \emph{non-basic} disks to the disk potential function. Although the bending system has a toric degeneration, we would like to emphasize that it does \emph{not} fit into the situation of Nishinou--Nohara--Ueda \cite{NishinouNoharaUeda}. Indeed, there exist holomorphic disks emanating from the inverse image of dimensional stratum of codimension $\geq 4$. Heuristically, as the bending system admits vanishing cycles, the inverse image can have an extra dimension so that there is a chance for the disk class to have Maslov index two even if it intersects lower dimensional stratum. 

Along the line of the classification of disks, another difficulty lies on a combinatorial description for the image of a bending system. The image is a polytope given by the intersection of half spaces from the triangle inequalities of a chosen triangulation. Indeed, some of half spaces are \emph{not} essential. It means that the superpotential merely written from the triangle inequalities may contain a Laurent monomial which does appear in the disk potential. In the equilateral case, we manage to describe the image via the relation of Grassmannians and polygon spaces obtained by the Gelfand--MacPherson correspondence in \cite{GelfandMacPherson, HausmannKnutson}. 

The strategy to compute the disk potential function is in order. Exploiting the non-redundant combinatorial description for the image and the corresponding toric degeneration, we classify the possible homotopy classes which can be realized as a holomorphic disk contributing the disk potential and compute the counting invariants for basic disks. In other words, we classify lattice points of the Newton polytope of the disk potential which can be a non-zero term of the disk potential. The plan is to compute  unknown counting invariants via relation between closed and open mirror symmetry.

As the dimension of a polygon space increases, the number of undetermined counting invariants that we should compute increases. It makes the computing process hard and sometimes impossible. To get around this difficulty, we implement the following strategy. According to Bouloc \cite{Bouloc}, a fiber of the bending systems can be expressed as a product of $\mathrm{SO}(3)$, $\bS^2$, and $\bS^1$. Using them non-toric Lagrangian submanifold as an anchor, we obtain a local model for $\mcal{M}_\mathbf{r}$. The above classification result ensures that each disk belongs to one of the local models. We then compute the counting invariant by embedding each local model into a simpler global model.

The method we present in the main body for computing the disk potential for equilateral case should not only restricted to the equilateral and generic case. It should work as long as $\mcal{M}_\mathbf{r}$ is smooth and monotone and the combinatorial structure is analyzed. To demonstrate it, we work out non-equilateral examples when $\dim_\C \mcal{M}_\mathbf{r} = 2$.

The authors hope that the result in this paper is paved way for testing many interesting theory by providing concrete examples. The first one is about the relation on the equivariant Floer theory and Floer theory. The polygon space is the symplectic reduction of a Grassmannian. The disk potential function on the quotient should be related to the equivariant disk potential function on the pre-quotient. This process is a subtlety on a choice of a stability condition or a level set. We illustrate this point by examples. Also, we discuss a role of the polygon space for a construction problem for the mirror cluster variety. Namely, the polygon space with bending systems is a symplectic model for the cluster variety of type A without frozen variables.

\subsection*{Acknowledgments}
The authors would like to thank the referee for his/her detailed and helpful comments on the manuscript. The authors would like to thank H. Hong and Y.-S. Lin for helpful comments and explanation on their joint work with S. Bardwell-Evans and M.-W. Cheung. The first named author would like to thank C. Woodward for explanation on his joint work with F. Charest, and thank J.-H. Lee for explanation on his joint with E. Lee. The second named author thanks Nai-Chung Conan Leung for motivating discussions on Gelfand--Zeitlin systems. The first named author was supported by the National Research Foundation of Korea (NRF) grant funded by the Korea government (NRF-2021R1F1A1057739) and (NRF-2020R1A5A1016126).  The second named author was supported by Simons Grant 580648.

\section{Polygon spaces and bending systems}

In this section, we review construction of polygon spaces and bending systems.  Polygon spaces and bending systems were formulated by Kapovich and Millson \cite{KapovichMillson}, which serve as a class of very interesting integrable systems in symplectic geometry. The relation between polygon spaces and Grassmannians of two planes are also reviewed.

\subsection{Polygon spaces and Grassmannians}

Let $n$ be a positive integer.
For an $(n+3)$-tuple $\mathbf{r} = (r_1, r_2, \cdots, r_{n+3})$ of positive real numbers, consider the set of all vectors $v_1, v_2, \dots, v_{n+3} \in \R^3$ satisfying
\begin{enumerate}
\item $\|v_j\|=r_j$ for each $j = 1, 2, \cdots, n+3$ and
\item $\sum_{j=1}^{n+3} v_j= 0$.
\end{enumerate}
For each tuple $(v_1, v_2, \cdots, v_{n+3})$ of vectors in this set, we locate the initial point of the first vector $v_1$ at the origin and the initial point of the vector $v_{j+1}$ at the terminal point of the vector $v_j$ for $j = 1, \dots, n+2$. Arranging the vectors in this manner, the terminal point of $v_{n+3}$ is the origin by the second condition $(2)$. Thus, each tuple of vectors can be regarded as an $(n+3)$-gon in the Euclidean space $\R^3$.

Let the real special orthogonal group $\mathrm{SO}(3)$ act on $\R^3$ linearly. For each $j$, the linear $\mathrm{SO}(3)$-action induces an action on each sphere $\bS^2(r_j)$ with the radius $r_j$ centered at the origin in $\R^3$. We then have the diagonal $\mathrm{SO}(3)$-action on the product $\prod_{j=1}^{n+3} \bS^2(r_j)$ of spheres. The \emph{polygon space} $\mcal{M}_\mathbf{r}$ is defined by the quotient of the space of $(n+3)$-gons in $\R^3$ by the diagonal $\mathrm{SO}(3)$-action, that is,
\begin{equation}\label{equ_polyso3quotient}
\mcal{M}_\mathbf{r} = \left\{ \mathbf{v} = (v_1, v_2, \dots, v_{n+3}) \in \prod_{j=1}^{n+3} \bS^2(r_j) \colon v_1 + v_2 + \cdots + v_{n+3} = 0 \right\} \slash \mathrm{SO}(3).
\end{equation}

If $\mathbf{r}$, the tuple of lengths of vectors, does not meet all necessary \emph{polygonal} inequalities, the polygon space $\mcal{M}_\mathbf{r}$ becomes empty. To avoid this triviality, we require that
\begin{equation}\label{equ_nonemptycondition}
\left( \sum_{j=1}^{n+3} r_j \right) - 2 r_i > 0 \quad \mbox{for all $i = 1, 2, \cdots, n+3$.}
\end{equation} 
The polygon space $\mcal{M}_\mathbf{r}$ in~\eqref{equ_polyso3quotient} can be seen as the GIT quotient $(\CP^1)^{n+3} \sslash \mathrm{PGL}(2, \C)$ and hence $\mcal{M}_\mathbf{r}$ is a projective algebraic variety of complex dimension $n$. 

\begin{definition}
A tuple $\mathbf{r} = (r_1, r_2, \dots, r_{n+3})$ obeying~\eqref{equ_nonemptycondition} is called \emph{generic} if there is no $(\epsilon_1, \epsilon_2, \cdots, \epsilon_{n+3}) \in \{\pm1\}^{n+3}$ such that
\begin{equation}\label{equ_generic}
\sum_{j=1}^{n+3} \epsilon_j r_j = 0.
\end{equation}
\end{definition}

The polygon space $\mcal{M}_\mathbf{r}$ is a smooth variety if and only if $\mathbf{r}$ is generic. Depending on our choice of $\mathbf{r}$, polygon spaces can produce different algebraic varieties. For instance, when $n = 2$, the projective space $\CP^2$, $\CP^1 \times \CP^1$, blow-ups of $\CP^2$ at a point, two points, three points, and four points can be realized as polygon spaces of pentagons with different choices of stability condition $\mathbf{r}$, cf. Appendix~\ref{AppendixA}.

Fix a generic tuple $\mathbf{r}$. To adorn the polygon space $\mcal{M}_\mathbf{r}$ with a K\"{a}hler structure, we regard $\CP^1 \simeq \bS^2$ as the coadjoint orbit $\mcal{O}_{(r_j, -r_j)}$. It is identified with the conjugation $\mathrm{U}(2)$-orbit of the $(2 \times 2)$-diagonal matrix with entries $(r_j, -r_j)$ via the Killing form, that is,
$$
\mcal{O}_{(r_j, -r_j)} = \left\{ Q \ \mathrm{diag} (r_j, - r_j) \, Q^{-1} \mid Q \in \mathrm{U}(2) \right\}.
$$
Each coadjoint orbit $\mcal{O}_{(r_j, -r_j)}$ naturally comes with an invariant symplectic form determined by $r_j$, which is called a Kirillov--Kostant--Souriau (KKS) symplectic form. The linear $\mathrm{SO}(3)$-action on $\R^3$ (and $\bS^2(r_j)$) is induced from the adjoint action of $\mathrm{SU}(2)$ on $\sqrt{-1} \mathfrak{su}(2)$ (and $\mcal{O}_{(r_j, -r_j)}$) via the double cover $\mathrm{SU}(2) \to \mathrm{SO}(3)$ and the identification between $\R^3$ and $\sqrt{-1} \mathfrak{su}(2) \colon$
$$
(x_1, x_2, x_3) \mapsto 
\begin{pmatrix}
x_3 & x_1 -  \sqrt{-1} x_2 \\
x_1 +  \sqrt{-1} x_2  & - x_3
\end{pmatrix}.
$$
The polygon space is the symplectic reduction of the product of $\mcal{O}_{(r_j, - r_j)} \simeq \CP^1$ by the diagonal $\mathrm{SU}(2)$-action, that is, $\prod_{j=1}^{n+3} {\mcal{O}}_{(r_j, -r_j)} \sslash \mathrm{SU}(2)$. It is equipped with the reduced symplectic form $\omega_\mathbf{r}$ from the product KKS forms. Then the polygon space is K\"{a}hler with respect to the reduced symplectic form and the complex structure from the GIT quotient $(\CP^1)^{n+3} \sslash \mathrm{PGL}(2, \C)$. 

The polygon spaces are related with the Grassmannians of two planes.  First, let us recall the description of Grassmannians as a coadjoint orbit.  For each $i = 1, 2, \cdots, n+3$, we set
\begin{equation}
| \mathbf{r} |_{i} \coloneqq \sum_{j=1}^i r_j \, \mbox{ and } \, |\mathbf{r}| \coloneqq |\mathbf{r}|_{n+3}
\end{equation}
and take an $(n+3)$-tuple 
\begin{equation}\label{equ_lambda}
\lambda = (|\mathbf{r}|, |\mathbf{r}|, 0, \dots, 0).
\end{equation}
The choice yields the diagonal matrix $\mathrm{diag}(\lambda) = \mathrm{diag}(|\mathbf{r}|, |\mathbf{r}|, 0, \dots, 0)$. We denote by $\mcal{O}_\lambda$ the conjugation $\mathrm{U}(n+3)$-orbit of the diagonal matrix whose entries are $\lambda$, that is,  
\begin{equation}\label{equ_coadjointorbitfrompoly}
\mcal{O}_\lambda \coloneqq \left\{ Q  \, \mathrm{diag}(|\mathbf{r}|, |\mathbf{r}|, 0, \dots, 0) \, Q^{-1} \mid Q \in \mathrm{U}(n+3) \right\}.
\end{equation}
The adjoint action $\mathrm{U}(n+3)$ on $\mcal{O}_\lambda$ is transitive, and $\mcal{O}_\lambda = \mathrm{U}(n+3) / (\mathrm{U}(2) \times \mathrm{U}(n+1))$, the complex Grassmannian $\mathrm{Gr}(2, \C^{n+3})$ of two-planes in the complex vector space $\C^{n+3}$. 
When $|\mathbf{r}|$ is integral in particular, one can associate $\mathrm{Gr}(2, \C^{n+3})$ to the pull-backed K\"{a}hler form $|\mathbf{r}| \cdot \omega_{\CP^{N}}$ on $\CP^{N}$ where $N = {n+3 \choose 2} - 1$ via the Pl\"{u}cker embedding. This pull-backed form agrees with the KKS form $\omega_{\lambda}$.

Grassmannians can also be constructed via a symplectic reduction. Let us consider the space $\hom_\C (\C^2, \C^{n+3}) \simeq \C^{(n+3) \times 2}$ of complex linear transformations from $\C^2$ to $\C^{n+3}$. The group $\mathrm{U}(n+3)$ (resp. $\mathrm{U}(2)$) acts on $\hom_\C (\C^2, \C^{n+3})$ on the left  (resp. right) matrix multiplication in a Hamiltonian fashion. A moment map of the $\mathrm{U}(2)$-action can be taken as
\begin{equation}\label{equ_momentmapofu2}
\Phi_{\mathrm{U}(2)} \colon \C^{(n+3) \times 2} \to \frak{u}(2)^*, \quad  \Phi_{\mathrm{U}(2)}(A) \coloneqq \frac{A^H A}{2}
\end{equation}
where $A^H$ is the conjugate transpose (or Hermitian transpose) of $A$. Then the Grassmannian $\mathrm{Gr}(2,\C^{n+3}) \simeq \mcal{O}_\lambda$ is isomorphic to the reduction of $\C^{(n+3) \times 2}$ by $\mathrm{U}(2) \colon$
\begin{equation}\label{equ_grquotient}
\mathrm{Gr}(2,\C^{n+3}) \simeq \C^{(n+3) \times 2} \sslash_{|\mathbf{r}| \cdot I_{2}} \mathrm{U}(2).
\end{equation}

We discuss a relation between $\mathrm{Gr}(2,\C^{n+3})$ and the polygon space $\mcal{M}_{\mathbf{r}}$.  Let $T_{\mathrm{U}(n+3)}$ be the maximal torus consisting of the diagonal matrices of $\mathrm{U}(n+3)$. The torus also acts on $\C^{(n+3) \times 2}$ linearly with a moment map 
\begin{equation}\label{equ_momentmaptorus}
\begin{pmatrix}
z_1 & z_2 & \dots &z_{n+3} \\
w_1 & w_2 & \dots & w_{n+3} \\
\end{pmatrix}^T \mapsto 
\left( \frac{|z_1|^2 + |w_1|^2}{2}, \frac{|z_2|^2 + |w_2|^2}{2}, \dots, \frac{|z_{n+3}|^2 + |w_{n+3}|^2}{2} \right).
\end{equation}
The Hamiltonian torus action descends to the quotient~\eqref{equ_grquotient}.  Then the polygon space is isomorphic to the symplectic reduction of $\mcal{O}_\lambda$ by the residual $T$-action, where $T$ is the quotient of $T_{\mathrm{U}(n+3)}$ by the diagonal $\bS^1$-action.

\begin{theorem}[\cite{HausmannKnutson}]\label{Proposition_HK}
For an $(n+3)$-tuple $\mathbf{r}$ of positive real numbers, suppose that $\lambda$ is given by~\eqref{equ_lambda}. Then the polygon space $(\mcal{M}_\mathbf{r}, \omega_\mathbf{r})$ is symplectomorphic to the symplectic reduction ${T} \bbslash_{2\mathbf{r}} \mcal{O}_{\lambda}$ of the coadjoint orbit $(\mcal{O}_\lambda, \omega_\lambda)$.
\end{theorem}

This paper mainly deals with the case where 
\begin{equation}\label{equ_equilateral111}
r_1 = r_2 = \dots = r_{n+3}.
\end{equation} 
In this case, the tuple $\mathbf{r}$ is called \emph{equilateral}.  The corresponding polygon space $\mathcal{M}_\mathbf{r}$ is monotone, namely, the first Chern class $c_1(T\mathcal{M}_\mathbf{r})=\frac{1}{2\pi r_1} [\omega_{\mathbf{r}}] \in H^2(\mathcal{M}_\mathbf{r})$. By normalizing the K\"{a}hler form, we may assume that 
$$
c_1(T\mathcal{M}_\mathbf{r}) = \frac{1}{r_1} [ \omega_{\mathbf{r}}].
$$ 
The coadjoint orbit $\mcal{O}_\lambda$ with the KKS form $\omega_\lambda$ is denoted by $(\mcal{O}_{\lambda}, \omega_{\lambda})$.

\begin{remark}
The polygon space $\mcal{M}_\mathbf{r}$ can be a monotone symplectic manifold even though $\mathbf{r}$ is not equilateral. Such spaces would be also of interest, see some examples in Appendix~\ref{AppendixA}.
\end{remark}

\subsection{Bending systems and Gelfand--Zeitlin systems}\label{sec_bengzsystems}

The coadjoint orbit $\mcal{O}_{\lambda} \simeq \mathrm{Gr}(2,\C^{n+3})$ in~\eqref{equ_coadjointorbitfrompoly} carries a peculiar completely integrable system, which is called a Gelfand--Zeitlin system \cite{GuilleminSternbergGC}. We shall briefly recall this system below. 

Let $A^{(i)}$ be the leading principal $(i \times i)$ submatrix of $A \in \mcal{O}_{\lambda}$. As every element in $\mcal{O}_{\lambda}$ is a Hermitian matrix, every eigenvalue of each submatrix $A^{(i)}$ is real. Our specific choice of the tuple $\lambda$ of the form~\eqref{equ_lambda} implies that each submatrix $A^{(i)}$ has at most two non-zero eigenvalues. We denote by ${\Phi}_{i,1}(A)$ (resp. ${\Phi}_{i-1,2}(A)$) the first (resp. the second) largest eigenvalue of $A^{(i)}$. They define real-valued functions, ${\Phi}_{i,1} \colon \mcal{O}_{\lambda} \to \R$ and ${\Phi}_{i-1,2} \colon \mcal{O}_{\lambda} \to \R$. 

\begin{definition}
The \emph{Gelfand--Zeitlin system} (\emph {GZ system} for short) is defined by 
\begin{equation}\label{equ_GZsystem}
{\Phi}_{\lambda} \colon \mcal{O}_{\lambda} \to \R^{2(n+1)}, \, {\Phi}_{\lambda} \coloneqq ({\Phi}_{1,1}, ({\Phi}_{2,1}, {\Phi}_{1,2}), ({\Phi}_{3,1}, {\Phi}_{2,2}), \dots, ({\Phi}_{n+1,1}, {\Phi}_{n,2}), {\Phi}_{n+1,2}).
\end{equation}
\end{definition}

Each component ${\Phi}_{i,j}$ may \emph{not} be smooth on the whole space $\mcal{O}_{\lambda}$. By the implicit function theorem, it is smooth on an open dense subset of $\mcal{O}_{\lambda}$. 

\begin{theorem}[\cite{Thimm, GuilleminSternbergGC}]
The Gelfand--Zeitlin system ${\Phi}_{\lambda}$ is a completely integrable system on $\mcal{O}_{\lambda}$. Moreover, each component ${\Phi}_{i,j}$ generates a Hamiltonian $\bS^1$-action on its smooth loci.
\end{theorem}

The prescribed eigenvalues of an $(n+3) \times (n+3)$ matrix give restrictions on the eigenvalues of the submatrices by the min-max principle. In particular, the choice~\eqref{equ_lambda} says that ${\Phi}_{n+2,1}$ is the constant function with the value $| \mathbf{r} |$. Moreover, the image of the GZ system is determined by the following system of inequalities$\colon$
\begin{equation}\label{equation_gzpattern}
	\begin{alignedat}{27}
			  & &&&& 0 &&&&  \,\, \leq \,\ &&&& {u}_{1,2}  &&&& \,\, \leq \,\, &&&&  {u}_{2,2} &&&& \,\, \leq \,\, &&&& \dots &&&& \,\, \leq \,\, &&&& {u}_{n,2} &&&& \,\, \leq \,\, &&&& {u}_{n+1,2} \\
			  & &&&& &&&& &&&& \,\,\, \ddge &&&& &&&&  \,\,\, \ddge &&&& &&&& &&&&  &&&& \,\,\, \ddge  &&&& &&&& \,\,\,\,\, \ddge \\
			  			  & &&&&  &&&& \,\,  \,\, &&&& {u}_{1,1}  &&&& \,\, \leq \,\, &&&& {u}_{2,1} &&&& \,\, \leq \,\, &&&& \dots &&&& \,\, \leq \,\, &&&& {u}_{n,1} &&&& \,\, \leq \,\, &&&& {u}_{n+1,1}  &&&& \,\,  \leq \,\, &&&& | \mathbf{r} | \,\,	
			  \end{alignedat}
\end{equation}
For later purposes, we label the half spaces in~\eqref{equation_gzpattern} as follows$\colon$
\begin{equation}\label{equ_halfspacesgzpoly}
\begin{cases}
k_{i,0}(\mathbf{u}) \coloneqq u_{i,1} - u_{i,2} \geq 0 \quad \mbox{ for $i = 1, 2, \dots, n+1$}, \\
k_{i,1}(\mathbf{u}) \coloneqq u_{i+1,1} - u_{i,1} \geq 0 \quad \mbox{ for $i = 1, 2, \dots, n+1$}, \\
k_{i,2}(\mathbf{u}) \coloneqq u_{i+1,2} - u_{i,2} \geq 0 \quad \mbox{ for $i = 0, 1, \dots, n$}
\end{cases}
\end{equation}
where $u_{0,2} \coloneqq 0$ and $u_{n+2,1} \coloneqq | \mathbf{r} |$. The polytope given by~\eqref{equation_gzpattern} is denoted by $\Delta_{\lambda}$ in $\R^{2(n+1)}$ and called a \emph{Gelfand--Zeitlin polytope}. 

We now recall the bending system on the polygon space $\mcal{M}_\mathbf{r}$. Recall that each element in $\mcal{M}_\mathbf{r}$ can be represented by a polygon in $\R^3$. For each pair $(i,j)$ with $1 \leq i < j \leq n+3$ and $j - i \not\equiv 1 \mod n+3$, let 
$$
d_{i,j} \coloneqq v_i + \dots + v_{j-1} \in \R^3
$$
and define
\begin{equation}\label{equ_phiij}
\Psi_{i,j} = \Psi_{d_{i,j}}  \colon \mcal{M}_\mathbf{r} \to \R, \quad [\mathbf{v}] = [(v_1, \dots, v_{n+3})] \mapsto \mbox{(the length of the diagonal $d_{i,j}$)}.
\end{equation}
For the sake of notational simplicity, we set $\Psi_{j} \coloneqq \Psi_{1,j+2}$ for $j = 1, 2, \cdots, n$, that is, 
\begin{equation}\label{equ_phij}
\Psi_{j} \colon \mcal{M}_\mathbf{r} \to \R, \quad [\mathbf{v}] = [(v_1, \dots, v_{n+3})] \mapsto \mbox{(the length of the diagonal $d_j$)}
\end{equation}
where 
\begin{equation}\label{equ_djdj1}
d_j \coloneqq d_{1, j+2} = v_1 + \dots + v_{j+1} \in \R^3.
\end{equation} 
If two diagonals $d$ and $d^\prime$ (of the polygon) do not cross each other, then the corresponding functions $\Psi_d$ and $\Psi_{d^\prime}$ in~\eqref{equ_phiij} are Poisson commutative. Therefore, a maximal collection of functions of the form~\eqref{equ_phiij} satisfying that the corresponding diagonals do not cross produces a completely integrable system on $\mcal{M}_\mathbf{r}$. As a result, there are Catalan number of such completely integrable systems arising from maximal collections of non-crossing diagonals. 

Each of these completely integrable systems is called a \emph{bending system}. There is a one-to-one correspondence between the set of bending systems on $\mcal{M}_\mathbf{r}$ and the set of triangulations of the $(n+3)$-gon by adding diagonals. A bending system consisting of the functions of the form~\eqref{equ_phij} corresponds to the diagonals $\mathbf{d} = (d_1, d_2, \cdots, d_n)$ in~\eqref{equ_djdj1}. We denote this bending system by 
\begin{equation}\label{equ_caterpillaerbeding}
\Psi_\mathbf{r} \coloneqq (\Psi_1, \Psi_2, \dots, \Psi_n) \colon \mcal{M}_\mathbf{r} \to \R^n.
\end{equation}
This completely integrable system is called a \emph{caterpillar bending system} on $\mcal{M}_\mathbf{r}$, see Figure~\ref{fig_caterpillar}. 

\begin{figure}[h]
	\scalebox{0.85}{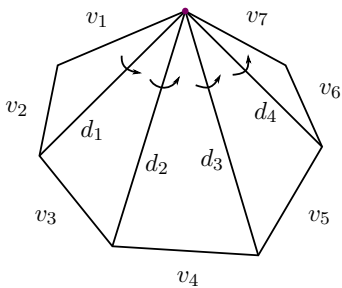}
	\caption{\label{fig_caterpillar} Caterpillar bending system}	
\end{figure}

Furthermore, each component $\Psi_{i,j}$ is an action variable on an open dense subset of $\mcal{M}_\mathbf{r}$. In other words, each component generates a ``bending" Hamiltonian $\bS^1$-action described below. The diagonal $d_{i,j}$ decomposes the polygon $\mathbf{v}$ into two pieces of polygons$\colon$ $(v_i, \dots, v_{j-1}, -d_{i,j})$ and $(d_{i,j}, v_{j}, \dots, v_{n+3}, v_1, \dots v_{i-1})$. By taking a slice of the diagonal $\mathrm{SO}(3)$-action, we may fix the first piece. Then the Hamiltonian $\bS^1$-action of $\Psi_{i,j}$ is to bend the other piece along the diagonal $d_{i,j}$ in the space $\R^3$. 

Let us summarize what we have discussed. 

\begin{theorem}[\cite{Klyachko, KapovichMillson}]
A maximal collection of functions of the form~\eqref{equ_phiij} satisfying that the corresponding diagonals do not cross each other forms a completely integrable system on $\mcal{M}_\mathbf{r}$. Moreover, each component ${\Psi}_{i,j}$ in the system generates a bending Hamiltonian $\bS^1$-action on its smooth loci.
\end{theorem}

Fix a maximal collection $\mathbf{d}$ of non-crossing diagonals of the $(n+3)$-gon with side length $\mathbf{r}$. Suppose that the maximal collection $\mathbf{d}$ consists of $n$-diagonals, say $d_{i_1, j_1}, d_{i_2, j_2}, \cdots, d_{i_n, j_n}$. We denote by $\Psi_{\mathbf{d}} \colon \mcal{M}_\mathbf{r} \to \R^{n}$ the bending system corresponding to the collection $\mathbf{d}$. In fact, the image of each bending system is a polytope as that of the GZ system is. We call the image of a bending system a \emph{bending polytope}. The bending polytope is defined by the system of triangle inequalities from the fixed triangulation.

We now explain the topology of the fiber over a given point in the bending polytope $\Delta \coloneqq \Psi_{\mathbf{d}}(\mcal{M}_\mathbf{r})$. Note that a choice of a point $\ell_\mathbf{d}$ in the bending polytope $\Delta$ determines the lengths of diagonals,
$$
\ell_\mathbf{d} = \left(|d_{i_1, j_1}|, |d_{i_2, j_2}|, \cdots, |d_{i_n,j_n}| \right) \in \Delta.
$$ 
In fact, the fiber of each point in the interior of $\Delta$ is a Lagrangian $T^n$-orbit. An interesting fiber occurs when some components in $\ell_\mathbf{d}$ vanish, that is, $\ell_\mathbf{d}$ is contained in a proper face of $\Delta$.

To describe the topology of the fiber of the given point $\ell_\mathbf{d}$, we take a ``polygonal" representative $\mathbf{v}$ of a point in the fiber of $\ell_\mathbf{d}$, that is, $[\mathbf{v}] \in \Psi_\mathbf{d}^{-1}(\ell_\mathbf{d})$. The representative $\mathbf{v}$ is a (possibly singular) polygon in $\R^3$ with the lengths of its diagonals specified by $\ell_\mathbf{d}$. Note that $\mathbf{v}$ consists of lines and polygons. For a natural number $\nu$ with $\nu \geq 2$, we set
$$
c_\nu \coloneqq
\begin{cases}
\mbox{the number of line segments in $\mathbf{v}$} &\mbox{if $\nu = 2$,}\\
\mbox{the number of $\nu$-gons in $\mathbf{v}$} &\mbox{if $\nu \geq 3$.}
\end{cases}
$$
Let 
\begin{itemize}
\item $m_1 \coloneqq \sum_{\nu \geq 4} (c_\nu \cdot (\nu - 3))$, the number of non-crossing diagonals of polygons in $\mathbf{v}$.
\item $m_2 \coloneqq c_2$, the number of line segments in $\mathbf{v}$,
\item $m_3 + 1 \coloneqq \sum_{\nu \geq 3} c_\nu$, the number of polygons in $\mathbf{v}$.
\end{itemize} 
The \emph{genericity} of $\mathbf{r}$ ensures the existence of at least one $\nu$-gon ($\nu \geq 3$), that is, $m_3 \geq 0$. It implies that every fiber is smooth.

\begin{theorem}[Theorem 4.4 in \cite{Bouloc}]\label{theorem_Bouloc}
For a generic $(n+3)$-tuple $\mathbf{r}$ which determines the lengths of edges, consider a bending system on the polygon space $\mcal{M}_\mathbf{r}$.
Then every fiber of the bending system is an isotropic submanifold. 

Moreover, for an $n$-tuple $\ell_\mathbf{d}$ which determines the length of diagonals, the fiber of the bending system over the $n$-tuple is diffeomorphic to 
\begin{equation}\label{equ_s1s2so3}
(\bS^1)^{m_1} \times (\bS^2)^{m_2} \times \mathrm{SO}(3)^{m_3}
\end{equation}
where $m_1, m_2, m_3$ can be read off from a polygon determined by the length of diagonals.
\end{theorem}

For instance, the fiber corresponding to the configuration polygon given in Figure~\ref{fig_singularfiber} is diffeomorphic to  $(\bS^1)^3 \times \bS^2 \times \mathrm{SO}(3)^{2-1}$.

\begin{figure}[h]
	\scalebox{0.85}{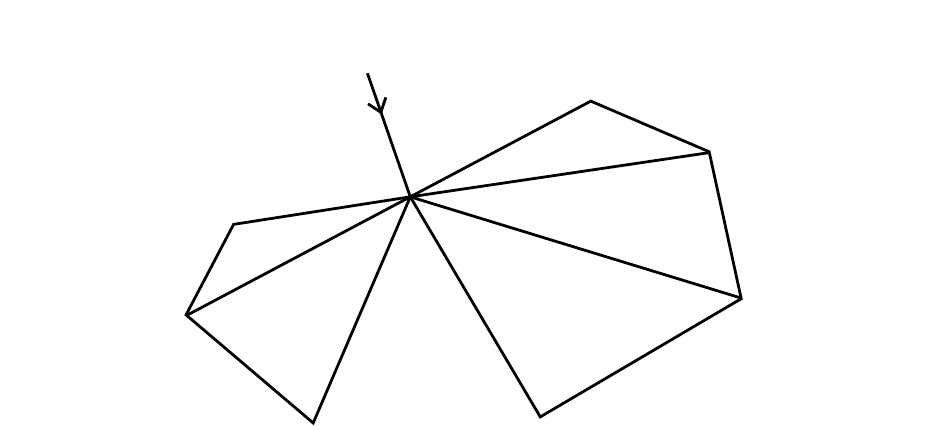}
	\caption{\label{fig_singularfiber} Singular fiber $(\bS^1)^3 \times \bS^2 \times \mathrm{SO}(3)^{2-1}$.}	
\end{figure}

We now explain a relation between the caterpillar bending system $\Psi_{\mathbf{r}}$ in~\eqref{equ_caterpillaerbeding} is related to the GZ system ${\Phi}_{\lambda}$ in~\eqref{equ_GZsystem}. 
Recall that the polygon space $\mathcal{M}_\mathbf{r}$ is symplectomorphic to the symplectic reduction ${T} \bbslash_{2\mathbf{r}} \mcal{O}_{\lambda}$ by Theorem~\ref{Proposition_HK}. Choose a set $\{\Phi_{1,2}, \Phi_{2,2}, \cdots, \Phi_{n,2}\}$ of components of the Gelfand--Zeitlin system $\Phi_{\lambda}$ on $\mcal{O}_\lambda$ in~\eqref{equ_GZsystem}. Those components restricted to the level set at $2\mathbf{r}$ factor through the symplectic reduction and the chosen components induce the completely integrable system
$$
\Phi_{\mathbf{r}} \colon {T} \bbslash_{2\mathbf{r}} \mcal{O}_{\lambda} \to \R^{n}.
$$ 
We denote by $\Phi_j$ the function on ${T} \bbslash_{2\mathbf{r}} \mcal{O}_{\lambda}$ induced by $\Phi_{j,2}$ for $j = 1, 2, \cdots, n$. Each component is a moment map for the residual $\bS^1$-action (defined on an open dense subset). The induced system agrees with the caterpillar bending system on $\mcal{M}_\mathbf{r}$ (up to some coordinate changes).

The moment map induced by~\eqref{equ_momentmaptorus} provides the relation 
\begin{equation}\label{equ_rjuu}
(u_{i+1,1} + u_{i,2}) - (u_{i,1} + u_{i-1,2} ) = r_{i+1}
\end{equation}
for $i = 0, 1, \dots, n+2$. From~\eqref{equation_gzpattern} and~\eqref{equ_rjuu}, the image of $\mcal{M}_\mathbf{r}$ under the caterpillar bending system can be described below.

Let $\mathbf{u} = (u_1, u_2, \cdots, u_n)$ be the coordinate system for $\R^n$. 
Consider a polytope in $\R^n$ given by the intersection of the following half-spaces$\colon$
\begin{equation}\label{equation_polypattern}
	\begin{alignedat}{27}
			  & &&&& 0 &&&& \,\, \leq \,\, &&&&  \,\,\, u_1  &&&& \,\, \leq \,\, &&&& \quad \quad u_{2} &&&& \,\, \leq \,\, &&&& \dots &&&& \,\, \leq \,\, &&&& \quad \quad u_n &&&& \,\, \leq \,\, &&&&  \,\,\, |\mathbf{r}| - 2r_{n+3} \\
			  & &&&& &&&& &&&& \,\,\, \ddge &&&& &&&&   \quad \quad \, \ddge &&&& &&&& &&&&  &&&&  \quad \quad \ddge  &&&& &&&& \,\, \quad \quad \ddge \\
			  			  & &&&& &&&& &&&& 2 |\mathbf{r}|_{1} &&&& \,\, \leq \,\, &&&& 2|\mathbf{r}|_{2} - u_1 &&&& \,\, \leq \,\, &&&& \dots &&&& \,\, \leq \,\, &&&& 2|\mathbf{r}|_{n} - u_{n-1} &&&& \,\, \leq \,\, &&&& 2|\mathbf{r}|_{n+1} - u_n &&&& \,\, \leq \,\, &&&& |\mathbf{r}| .
	\end{alignedat}
\end{equation}
We denote by $\Delta_{\mathbf{r}}$ the polytope determined by~\eqref{equation_polypattern} in $\R^n$. We label the inequalities in the columns, the inequalities in the first row, and the inequalities in the second row as follows. 
\begin{equation}\label{equ_elli0i1i2}
\begin{cases}
\ell_{i,0}(\mathbf{u}) \geq 0 \quad \mbox{ for $i = 1, 2, \dots, n+1$}, \\
\ell_{i,1}(\mathbf{u}) \geq 0 \quad \mbox{ for $i = 1, 2, \dots, n+1$}, \\
\ell_{i,2}(\mathbf{u}) \geq 0 \quad \mbox{ for $i = 0, 1, \dots, n$}
\end{cases}
\end{equation}
where
\begin{itemize}
\item $\ell_{1,0}(\mathbf{u}) = 2 |\mathbf{r}|_1 - u_1,\cdots, \, \ell_{n,0}(\mathbf{u}) = 2 | \mathbf{r} |_n  -u_{n-1} - u_n, \, \ell_{n+1,0}(\mathbf{u}) = 2|\mathbf{r}|_{n+1}  - | \mathbf{r} | +  2 r_{n+3} - u_n$,
\item $\ell_{1,1}(\mathbf{u}) = 2 r_2 - u_1,\cdots, \, \ell_{n,1}(\mathbf{u}) = 2r_{n+1} - u_n + u_{n-1}, \, \ell_{n+1,1}(\mathbf{u}) = | \mathbf{r} | - 2 | \mathbf{r} |_{n+1} +u_n $,
\item $\ell_{0,2}(\mathbf{u}) = u_1,\cdots, \, \ell_{n-1,2}(\mathbf{u}) = u_n - u_{n-1}, \, \ell_{n,2}(\mathbf{u}) = |\mathbf{r}| - 2r_{n+3} -  u_n$.
\end{itemize}

\begin{theorem}[\cite{HausmannKnutson}]\label{proposition_NNUgzbe}
The integrable system $( \Phi_{1,2}, \Phi_{2,2}, \cdots, \Phi_{n,2})$ on $\mcal{O}_\lambda$ induces a completely integrable system $\Phi_{\mathbf{r}} = (\Phi_1, \Phi_2, \cdots, \Phi_n)$ on $\mcal{M}_{\mathbf{r}}$. More precisely, the caterpillar bending system $\Psi_{\mathbf{r}}$ is related to the induced system $\Phi_{\mathbf{r}}$ as follows$\colon$
\begin{equation}\label{equ_underlinephi2j}
{\Phi}_{j} \circ \Upsilon = | \mathbf{r} |_{j+1} - \Psi_j \quad \mbox{ for $j = 1, 2, \cdots, n$}
\end{equation}
where $\Upsilon \colon \mcal{M}_\mathbf{r} \to {T} \bbslash_{2\mathbf{r}} \mcal{O}_{\lambda} $ is the symplectomorphism in Theorem~\ref{Proposition_HK}.

In particular, the image of $\mcal{M}_{\mathbf{r}}$ of the caterpillar bending system agrees with the polytope $\Delta_{\mathbf{r}}$ given by~\eqref{equation_polypattern} up to the affine transformation 
$$
u_{j,2} = | \mathbf{r} |_{j+1} - u_j \quad \mbox{for $j = 1, 2, \cdots, n$}.
$$
\end{theorem}

Thanks to Theorem~\ref{Proposition_HK} and~\ref{proposition_NNUgzbe}, from now on, we do not make any distinction on $(\Phi_j$ and $\Psi_j)$, and $({T} \bbslash_{2\mathbf{r}} \mcal{O}_{\lambda}$ and $\mcal{M}_\mathbf{r})$. The completely integrable system
\begin{equation}\label{equ_caterpillarbendingshifted}
\Phi_\mathbf{r} \coloneqq (\Phi_j  \colon j = 1, 2, \cdots, n ) \colon \mcal{M}_\mathbf{r} \simeq {T} \bbslash_{2\mathbf{r}} \mcal{O}_{\lambda} \to \R^n
\end{equation}
is also called a \emph{caterpillar bending system} and the polytope $\Delta_\mathbf{r}$ given by~\eqref{equation_polypattern} is called a \emph{caterpillar bending polytope}. 

For later usage, when $\mathbf{r}$ is equilateral, we figure out which hyperplane given by $\ell_{i,j}(\mathbf{u}) = 0$ in~\eqref{equ_elli0i1i2} contains a facet of the caterpillar bending polytope $\Delta_\mathbf{r}$. We emphasize that in contrast to the GZ case, not every hyperplane contain a facet of $\Delta_\mathbf{r}$.  

\begin{proposition}\label{proposition_monotone}
Suppose that $\mathbf{r}$ is equilateral in~\eqref{equ_equilateral111}. Consider the caterpillar bending polytope $\Delta_\mathbf{r}$ defined in~\eqref{equ_elli0i1i2}. 
\begin{enumerate}
\item If $(i,j) \neq (1,0), (1,1), (n,2)$, and $(n+1, 0)$, then the intersection of the hyperplane $\left\{ \mathbf{u} \in \R^n \mid \ell_{i,j}(\mathbf{u}) = 0 \right\}$ and the polytope $\Delta_\mathbf{r}$ is a facet of $\Delta_\mathbf{r}$.
\item If $(i,j) = (1,0), (1,1), (n,2)$, or $(n+1, 0)$, then the intersection of the hyperplane $\left\{ \mathbf{u} \in \R^n \mid \ell_{i,j}(\mathbf{u}) = 0 \right\}$ and the polytope $\Delta_\mathbf{r}$ is a face of codimension two of $\Delta_\mathbf{r}$.
\end{enumerate}
\end{proposition}

\begin{proof}
First, we claim that every hyperplane given by  $\ell_{i,j}(\mathbf{u}) = 0$ except the four cases $(i,j) \neq (1,0), (1,1), (n,2)$, and $(n+1, 0)$ contains a facet of $\Delta_\mathbf{r}$. We only prove that the hyperplane given by $\ell_{2,0}(\mathbf{u}) =0$ contains a facet because the other cases are similar. We denote by $\mathbf{u}_0$ the \emph{center} of the polytope $\Delta_\mathbf{r}$ determined by
\begin{equation}\label{equ_centerofdeltamathbfr}
u_1 = r_1, \,  u_2 = 2 \cdot r_1, \, \cdots, \, u_n = n \cdot r_1.
\end{equation}
In particular, the center $\mathbf{u}_0$ satisfies $\ell_{i,j}(\mathbf{u}) > 0$ in~\eqref{equ_elli0i1i2}. We consider the following point $\mathbf{u}_{2,0}$ by translating the center $\mathbf{u}_0$ in the outward normal direction of $\ell_{2,0}(\mathbf{u}) \geq 0$,
$$
u_1 = (1 + {1}/{2}) \cdot r_1, u_2 = ( 2 + {1}/{2} ) \cdot r_1, u_3 = 3 \cdot r_1, \cdots, u_n = n \cdot r_1.  
$$
For an $n$-tuple ${\boldsymbol \epsilon} = ( \epsilon_1 , - \epsilon_1, \epsilon_3, \cdots, \epsilon_n)$ sufficiently close to the origin, note that $\mathbf{u}_{\epsilon} \coloneqq \mathbf{u}_{2,0} + r_1 {\boldsymbol \epsilon}$ satisfies $\ell_{2,0}(\mathbf{u}_{\epsilon}) = 0$ and all the other inequalities $\ell_{i,j}(\mathbf{u}_{\epsilon}) > 0$. It yields that the point $\mathbf{u}_{\epsilon}$ is contained in a facet and the hyperplane given by $\ell_{2,0}(\mathbf{u}) = 0$. Hence, the claim is proven.

Next, we deal with the remaining four cases. If $\mathbf{r}$ is equilateral, then we have $\ell_{1,0}(\mathbf{u}) = \ell_{1,1}(\mathbf{u})$ and $\ell_{n,2}(\mathbf{u}) = \ell_{n+1,0}(\mathbf{u})$. It suffices to deal with the two cases$\colon \ell_{1,0}(\mathbf{u}) = 0$ and $\ell_{n,2}(\mathbf{u})= 0$. We only prove that the hyperplane given by $\ell_{1,0}(\mathbf{u}) = 0$ contains a codimension two face of $\Delta_\mathbf{r}$ since the other case is similar. Note that $\ell_{1,0}(\mathbf{u}) = 0$ yields $u_1 =2 r_1$. Since $u_1 =2 | \mathbf{r}|_{2} - u_1$ in~\eqref{equation_polypattern}, $\ell_{1,2} (\mathbf{u}) = 0$ and $\ell_{2,0}(\mathbf{u}) = 0$ must hold, and hence $u_2 = 2 r_1$. We consider the following point $\mathbf{u}_{1,0}$ by translating the center $\mathbf{u}_0$ in the outward normal direction of $\ell_{1,0}(\mathbf{u}) \geq 0$,
$$
u_1 = 2 \cdot r_1, u_2 = 2 \cdot r_1, u_3 = 3 \cdot r_1, \cdots, u_n = n \cdot r_1.
$$
For an $n$-tuple ${\boldsymbol \epsilon} = ( 0, 0, \epsilon_3, \cdots, \epsilon_n)$ sufficiently close to the origin, note that $\mathbf{u}_{\epsilon} \coloneqq \mathbf{u}_{1,0} + r_1 {\boldsymbol \epsilon}$ satisfies $\ell_{1,0}(\mathbf{u}_{\epsilon}) = \ell_{2,0}(\mathbf{u}_{\epsilon}) = \ell_{1,2}(\mathbf{u}_{\epsilon}) = \ell_{1,1}(\mathbf{u}_{\epsilon}) = 0$ and all the other inequalities $\ell_{i,j}(\mathbf{u}_{\epsilon}) > 0$. It yields that the point $\mathbf{u}_{\epsilon}$ is contained in a face of codimension two and the hyperplane given by $\ell_{1,0}(\mathbf{u}) = 0$. Hence, the claim is proven.
\end{proof}

\section{Potential functions of polygon spaces}

The aim of this section is to state the main results and discuss the results from the perspectives of equivariant Floer theory and the theory of cluster algebras and varieties. After reviewing the disk potential functions, we derive the disk potential function of a monotone Lagrangian torus fiber of the caterpillar bending system from equivariant Floer theory. By using the relation of generalized Gelfand--Zeitlin fibers, we construct a glued SYZ mirror of the equilateral and generic polygon space, which will be a cluster variety of type $A$.

\subsection{Disk potential functions}\label{subsection_diskpotential}

We begin by recalling the definition of a monotone Lagrangian submanifold. A Lagrangian submanifold $L$ of a symplectic manifold $(X, \omega)$ is called \emph{monotone} if the Maslov index and the symplectic area are positively proportional, that is, there exists a positive number $c > 0$ such that $\omega(\beta) = c \cdot \mu(\beta)$ for all $\beta \in \pi_2(X,L)$.

In order to prove that the Lagrangian torus fiber at the center of a caterpillar bending polytope is monotone, we recall the Maslov index formula for a gradient holomorphic disk in \cite{ChoKimMONO}, which is useful to verify the monotonicity of a Lagrangian submanifold when an ambient symplectic manifold admits a Hamiltonian (local) torus action. Using this formula, we shall prove the Lagrangian torus located at the center of the caterpillar bending polytope is monotone.

Consider a symplectic manifold $X$ equipped with a Hamiltonian $\bS^1$-action with a moment map $H$. Suppose that we have an $\bS^1$-invariant Lagrangian submanifold $L$ contained in a level set of $H$. To generate a gradient holomorphic disk, we pick the following data$\colon$
\begin{enumerate}
\item an $\bS^1$-invariant compatible almost complex structure $J$,
\item a free $\bS^1$-orbit $\vartheta$ in $L$.
\end{enumerate}
Let $X_H$ be the Hamiltonian vector field generated by $H$. Then the moment map $H$ generates the gradient vector field $\nabla H = J X_H$ using the Riemannian metric $g_J$ from $\omega$ and $J$, which is defined by $g_J( \nabla H, \cdot ) = d H (\cdot )$. Assume that the integral curve $\gamma$ of the vector field $\nabla H$ with an initial condition $p \in \vartheta$ converges to a fixed point. The collection of $\bS^1$-orbits of $\gamma$ forms a $J$-holomorphic disk whose boundary is $\vartheta$. We call such a holomorphic disk a \emph{gradient holomorphic disk} of $\vartheta$.

Suppose that a connected Lie group $G$ acts on the manifold $X$. The action is called \emph{semi-free} if the action is free on the complement of the set of fixed points. In the following circumstance, the Maslov index of a gradient holomorphic disk of $\vartheta$ can be computed by the codimension of the maximal fixed component of the action intersecting the disk. 

\begin{proposition}[Corollary 3.8 in \cite{ChoKimMONO}]\label{proposition_maslovformula}
Suppose that $X$ is a symplectic manifold equipped with a Hamiltonian $\bS^1$-action. Let $H$ be a moment map of the $\bS^1$-action and $L$ an $\bS^1$-invariant Lagrangian submanifold lying in a level set of $H$. Choose an $\bS^1$-invariant compatible almost complex structure and a free $\bS^1$-orbit $\vartheta$ in $L$. Suppose that a class $\beta \in \pi_2(X, L)$ is represented by the gradient holomorphic disk of $\vartheta$ and the map $H$ attains the maximum at the fixed point. If the $\bS^1$-action is semi-free near the fixed point of the gradient disk, then the Maslov index $\mu(\beta)$ is the real codimension of the maximal fixed component.	
\end{proposition}

\begin{remark}
More generally, the Maslov index of a gradient disk can be computed by the sum of weights at the unique fixed point even when $H$ does not attain the maximum value at the fixed point, see \cite[Theorem 3.7]{ChoKimMONO}. In this paper, we only need Proposition~\ref{proposition_maslovformula}.
\end{remark}

The fiber over the center $\mathbf{u}_0$ of the caterpillar bending system $\Phi_\mathbf{r}$ is a monotone Lagrangian torus where the center $\mathbf{u}_0$ of $\Delta_\mathbf{r}$ is located at~\eqref{equ_centerofdeltamathbfr}.

\begin{proposition}\label{lemma_monotonecbs}
If $\mathbf{r}$ is equilateral and generic, then the fiber $\Phi^{-1}_{\mathbf{r}}(\mathbf{u}_0)$ over the center of the caterpillar bending system on $\mcal{M}_\mathbf{r}$ is a monotone Lagrangian torus.
\end{proposition}

Later on, we apply Proposition~\ref{proposition_maslovformula} to prove Proposition~\ref{lemma_monotonecbs} in Section~\ref{sec_computationofhomotopygroups} after computing the relative homotopy group and describing its basis as a set of gradient disks.

Let us briefly recall the disk potential function of a monotone Lagrangian torus. Let $X$ be a closed symplectic manifold and $L$ a Lagrangian submanifold of $X$. For a compatible almost complex structure $J$ and a class $\beta \in \pi_2(X,L)$, we denote by $\mcal{M}_1(X, L, \beta; J)$ the moduli space of stable $J$-holomorphic maps from a bordered genus $0$ Riemann surface into $X$ with boundary condition $L$ in $\beta$ with one boundary marked point $z_0$. Let 
$$
\mathrm{ev}_0 \colon \mcal{M}_1(X, L, \beta; J) \to L, \quad \varphi \mapsto \varphi(z_0)
$$
be the evaluation map at the marked point $z_0$.

Assume that the Lagrangian submanifold $L$ is orientable, (relative) spin, and monotone. To orient the moduli spaces, we choose an orientation and a spin structure of $L$. Under the monotonicity assumption on $L$, every disk class of Maslov index two is Fredholm regular for a \emph{generic} compatible almost complex structure $J$.  Moreover, $L$ does not bound any non-constant holomorphic disks of Maslov index zero. Then, for each class $\beta$ of Maslov index two, the moduli space $\mcal{M}_1(X, L, \beta; J)$ is an oriented manifold \emph{without} boundary. The counting $n(\beta; L, J)$ is defined to be the degree of the evaluation map $\mathrm{ev}_0$. This number $n(\beta; L, J)$ is called the \emph{open Gromov--Witten invariant} of $\beta$.

In this monotone case, the counting number $n(\beta; L, J)$ is \emph{invariant} under different choices of a generic compatible almost complex structure. The parametrized moduli space in the class $\beta$ of Maslov index two arising from a generic path from $J$ to $J^\prime$ in the space of compatible almost complex structures does not bubble off because of the monotonicity condition. It produces an oriented cobordism between $\mcal{M}_1(X, L,\beta;J)$ and $\mcal{M}_1(X, L,\beta;J^\prime)$.

\begin{lemma}[\cite{EliashbergPolterovich}] \label{lemma_invarianceofcounting}
Suppose that $L$ is a monotone Lagrangian submanifold. For any class $\beta \in \pi_2(X,L)$ of Maslov index two and generic compatible almost complex structures $J$ and $J^\prime$, the counting invariants $n(\beta; L,J)$ and $n(\beta; L,J^\prime)$ are equal. 
\end{lemma}

For the sake of notational simplicity, we write $n_\beta = n(\beta; L,J)$ and $\mcal{M}_1(L, \beta) = \mcal{M}_1(X, L, \beta; J)$ if $X$ and $L$ are clear in the context.

A Lagrangian submanifold $L$ might bound a $J$-holomorphic disk in general, which causes an obstruction on Floer cohomology. When $L$ is a monotone Lagrangian torus of a compact non-symplectically aspherical manifold $X$, the obstruction does not cause any trouble for defining Floer cohomology. As we have observed above, for any class $\beta$ of Maslov index two, the moduli space $\mcal{M}_1(\beta)$ represents a cycle and the push-forward by $\mathrm{ev}_0$ is a multiple of the fundamental cycle $[L]$. It in turn implies that the composition of consecutive differentials vanishes so that Floer cohomology is well-defined since $[L]$ is a unit of the Fukaya $A_\infty$-algebra associated to $L$. 

Assume in addition that the monotone Lagrangian $L$ is a torus. To implement deformations of Lagrangian Floer theory, we equip the Lagrangian torus $L$ with the complex trivial line bundle $\scr{L}$ and we take a flat non-unitary $\C^*$-connection $\nabla$ on the line bundle, see \cite{ChoNonunitary}. Choose a set of oriented loops $\{ \vartheta_1, \dots, \vartheta_n \}$ which forms a basis for $\pi_1(L) \simeq \Z^n$ as a $\Z$-module. Two connections $\nabla$ and $\nabla^\prime$ are called \emph{gauge equivalent} if $\mathrm{hol}_\nabla (\vartheta_j) =  \mathrm{hol}_{\nabla^\prime} (\vartheta_j)$ for each $j = 1, 2, \cdots, n$. We denote by $\mcal{M}(\scr{L})$ the moduli space of flat $\C^*$-connections modulo the gauge equivalence. The choice of oriented loops gives rise to a coordinate system on $\mcal{M}(\scr{L})$ consisting of the holonomy variables
\begin{equation}\label{equ_zjvariable}
z_j \coloneqq \mathrm{hol}_\nabla (\vartheta_j) \in \C^* \quad \mbox{ ($j = 1, 2, \cdots, n$)}.
\end{equation} 
In other words, the moduli space $\mcal{M}(\scr{L})$ is identified with an algebraic torus $(\C^*)^n$. 

Let
$$
\Lambda_0 =\left\{\sum_{i=0}^\infty a_i T^{A_i} \mid A_i \geq 0 \textrm{ increase to } +\infty, a_i \in \C \right\}
$$ 
be the Novikov ring. Here, the variable $T$ is the formal parameter which will be used to record symplectic energy of $J$-holomorphic curves. Then the obstruction of Lagrangian Floer theory is expressed as
\begin{equation}\label{equ_obstrLFT}
\frak{m}^\nabla_0(1) = \sum_{\beta} \mathrm{hol}_\nabla (\partial \beta) \cdot (\mathrm{ev}_0)_* [\mcal{M}_1(\beta)] \cdot T^{\omega(\beta)}
= \left( \sum_{\beta} \mathrm{hol}_\nabla (\partial \beta) \cdot n_\beta \cdot T^{\omega(\beta)} \right) [L] .
\end{equation}
Since the counting invariants of $L$ can be non-empty only when the Maslov index is two, the symplectic areas of the disk classes contributing to~\eqref{equ_obstrLFT} are all equal. We omit the $T^{\omega(\beta)}$ for simplicity in this case.

\begin{definition}\label{def_diskpotentfunct}
The \emph{disk potential function} of a monotone Lagrangian torus $L$ is defined by
$$
W_L \colon \mcal{M}(\scr{L}) \simeq (\C^*)^n \to \C, \quad [\nabla] \simeq (z_1, \dots, z_n) \mapsto \sum_{\beta; \mu(\beta) = 2}  n_\beta \cdot z^{\partial \beta}.
$$
\end{definition}

Turning to the caterpillar bending system $\Phi_\mathbf{r}$, the disk potential function of the monotone Lagrangian torus $\Phi^{-1}_{\mathbf{r}}(\mathbf{u}_0)$ can be written as follows.

\begin{theorem}\label{theorem_main}
Suppose that $\mathbf{r}$ is equilateral and generic. Then the disk potential function of the monotone Lagrangian torus fiber $\Phi^{-1}_{\mathbf{r}}(\mathbf{u}_0)$ of the caterpillar bending system is
\begin{equation}\label{equ_theorem_main}
W_\mathbf{r} (\mathbf{z}) = \left(z_1 + \frac{2}{z_1} + z_n + \frac{2}{z_n} \right) + \sum_{j=1}^{n-1} \left( \frac{z_j}{z_{j+1}} + \frac{z_{j+1}}{z_j} + \frac{1}{z_{j}z_{j+1}}  \right).
\end{equation}
\end{theorem}

\begin{remark}
The generic and equilateral condition yields that $n$ is \emph{even}. If $n$ is odd, then $\mcal{M}_\mathbf{r}$ is a singular variety. It would be interesting to make sense of and compute the disk potential function in this case.
\end{remark}

The proof of Theorem~\ref{theorem_main} will be given at the end of Section~\ref{Section_computationofopenGWinv}.

\subsection{Equivariant Floer theory and disk potential functions of reduced spaces}\label{sec_equiFlertheory}

In this section, we discuss Theorem~\ref{theorem_main} from the aspect of equivariant Lagrangian Floer theory in an informal way.

We first recall some facts on Lie group quotients.  The reader is referred to \cite[Section 4]{GuilleminSternbergQuanmulti} for details and proofs. Let $G$ be a compact connected Lie group. Suppose that we have a K\"{a}hler manifold $(X, \omega, J)$ satisfying 
\begin{enumerate}
\item the $G$-action on $(X, \omega)$ is Hamiltonian and
\item the integrable almost complex structure $J$ is invariant under the $G$-action. 
\end{enumerate} 
Let $\mathbb{G}$ be the complexification of $G$, a connected complex Lie group which is uniquely characterized by the properties that
the Lie algebra of $\mathbb{G}$ is the complexified Lie algebra $\frak{g} \oplus \sqrt{-1} \frak{g}$ and
$G$ is a maximal compact subgroup of $\mathbb{G}$.
Under this circumstance, the $G$-action naturally extends to a holomorphic $\mathbb{G}$-action on $X$.

Let $\Phi_G \colon X \to \frak{g}^*$ be a moment map of the $G$-action.
Fix any regular level $\mu \in \frak{g}^*$ such that $G$ acts freely on the level set $Z \coloneqq \Phi_G^{-1}(\mu)$. The complexification $\mathbb{G}$ freely acts on the $\mathbb{G}$-orbit $Z_\mathbb{G}$ of $Z$. Moreover, the inclusion $\iota \colon Z \to Z_\mathbb{G}$ descends to a diffeomorphism
$$
\underline{\iota} \colon Z/G \to Z_\mathbb{G}/\mathbb{G}
$$
Then $Z/G \simeq  Z_\mathbb{G}/\mathbb{G}$ is also a K\"{a}hler manifold with the reduced symplectic form $\omega_\mathrm{red}$ on $Z/G$ and the complex structure ${J}_\mathrm{red}$ on $Z_\mathbb{G}/\mathbb{G}$. Let $L$ be a $G$-invariant Lagrangian submanifold.  Suppose further that
\begin{enumerate}
\item[(3)] the group $G$ acts freely on $L$.
\end{enumerate}
The quotient map
\begin{equation}\label{equ_principlegbundle}
\pi \colon (Z_\mathbb{G}, {L}) \to (Z_\mathbb{G}/\mathbb{G} , {L} / G).
\end{equation}
is a holomorphic principal $\mathbb{G}$-bundle.  
For each holomorphic map $\varphi \colon (\mathbb{D}, \partial \mathbb{D}) \to (Z_\mathbb{G}, {L})$, we have a holomorphic map $\pi \circ \varphi \colon (\mathbb{D}, \partial \mathbb{D}) \to (Z_\mathbb{G} / \mathbb{G}, {L}/G)$. 
For a homotopy class $\beta$ in $\pi_2 (Z_\mathbb{G}, L)$, the correspondence induces the map   
\begin{equation}\label{equ_principleGbundlemoduli}
\pi_* \colon \mcal{M}(Z_\mathbb{G}, L, \beta; J) \to \mcal{M}(Z_\mathbb{G}/\mathbb{G}, L/G, \pi_* \beta; {J}_\mathrm{red}) \quad \varphi \mapsto \pi \circ \varphi.
\end{equation}
The correspondence helps us to understand holomorphic disks in the quotient from those in the pre-quotient.

A particular nice class of examples in such a setting is provided by toric manifolds, in which $G = T$ and $X=\C^N$. Givental \cite{Givental} derived the $T$-equivariant quantum cohomology on a smooth toric manifold and expressed it in terms of a $T$-equivariant superpotential. In \cite{KimLauZheng}, for a semi-Fano toric manifold, the present authors defined the $T$-equivariant disk potential of a toric fiber using $T$-equivariant Lagrangian Floer theory in Morse model.  They computed the equivariant terms in the model, and verified that it agrees with Givental's equivariant superpotential via the mirror map. 

In a similar fashion, let $T$ be the quotient of the maximal torus $T_{\mathrm{U}(n+3)}$ by the diagonal $\mathbb{S}^1$-action and consider the torus $T$ action on $X \coloneqq \mathrm{Gr}(2,\C^{n+3}) \simeq \mcal{O}_{\lambda}$, we have the following.

\begin{proposition}
Let $L$ be a Gelfand--Zeitlin torus fiber of $\mcal{O}_\lambda \simeq \mathrm{Gr}(2,\C^{n+3})$. Let $\vartheta_{i,j}$ be an $\mathbb{S}^1$-orbit contained in $L$ and generated by the action variable $\Phi_{i,j}$ in~\eqref{equ_GZsystem} and $y_{i,j}$ the holonomy variable associated with $\vartheta_{i,j}$. Set $y_{0,1} = y_{0,2} = y_{-1,2} = 1$. Then the $T$-equivariant disk potential of $L$ is equal to
	\begin{equation} \label{eqn_equiv_disk_potential}
	W_T (\mathbf{y}) = W (\mathbf{y}) + \sum_{i=1}^{n+2} \lambda_i \log \left( \frac{y_{i,1}y_{i-1,2}}{y_{i-1,1}y_{i-2,2}} \right)
	\end{equation}
	where 
	\begin{enumerate}
	\item $W(\mathbf{y})$ is the disk potential function of $L \subset \mcal{O}_\lambda$, which was computed in \cite{NishinouNoharaUeda},
	\item $\lambda = (\lambda_1, \lambda_2, \cdots, \lambda_{n+2})$ is the equivariant parameters for the $T$-action, and 
	\item ${y_{i,1}y_{i-1,2}}/{y_{i-1,1}y_{i-2,2}}$ is the holonomy variable associated to an orbit generated by the $i$-th factor of the torus $T$.
	\end{enumerate}
\end{proposition}

\begin{proof}[Sketch of Proof]
Let us first briefly recall the definition of $T$-equivariant disk potential introduced in \cite{KimLauZheng}. Consider the $\mathbb{S}^1$-orbits generated by
\begin{enumerate}
\item $\Phi_{i,1} + \Phi_{i-1,2} - \Phi_{i-1,1} - \Phi_{i-2, 2}$ for $i = 1, 2, \cdots, n+2$,
\item $\Phi_{i,2}$ for $i = 1, 2, \cdots, n$.  
\end{enumerate}
We denote them by $\vartheta_i$ for $i = 1, 2, \cdots, n+2$ and $\vartheta_i^\prime$ for $i = 1, 2, \cdots, n$, respectively. The action generated by $(1)$ can be identified with the $T$-action on $\mcal{O}_\lambda$. We identify $T$ with $(\bS^1)^{n+2}$ and denote by $L_T$ the Borel construction $L\times_T (\mathbb{S}^{\infty})^{n+2}$. Let us choose a perfect Morse function $f \colon L \to \R$ in such a way that the unstable manifold of a critical point $p_i$ (resp. $p_i^\prime$) of $f$ intersects transversally to the orbit $\vartheta_{i}$ (resp. $\vartheta_i^\prime$) where $\{ p_i \mid i = 1,2, \cdots, n+2 \} \cup \{p_i^\prime \mid i =1, 2, \cdots, n\}$ are degree $1$ critical points of $f$.

We fix a family Morse function $f_T$ on the approximation spaces of $L_T$ whose critical points are of the form $(p,q)$ where $p$ is a critical point of the chosen perfect Morse function $f$ on $L$ and $q$ is a critical point of a perfect Morse function $f_B$ on the classifying space $BT=(\CP^{\infty})^{n+2}$.
We denote by $\mathbf{1}_{L}$ and $\mathbf{1}_{BT}$ the minimal points of $f$ and $f_B$, respectively. Then $\{ (p_{i},\mathbf{1}_{BT}) \mid i = 1, 2, \cdots, n+2 \}$ and $\{ (p_{i}^\prime,\mathbf{1}_{BT}) \mid i = 1, 2, \cdots, n \}$ are the degree $1$ critical points of $f_T$. Let $\lambda_1, \lambda_2, \cdots , \lambda_{n+2}$ be the equivariant parameters of the $T$-action, which are the degree $2$ critical points of $f_B$. We will abuse notation and view $\lambda_i $ as a degree $2$ critical point of $f_T$ of the form $(\mathbf{1}_{L},\lambda_i)$.

Then the $T$-equivariant disk potential function 
$$
W_T \colon \Lambda_0^{\times 2n +2} \to \Lambda_0 [ \lambda_1, \cdots, \lambda_{n+2} ]
$$ 
in \cite[Definition 3.11]{KimLauZheng} of $L$ is defined by counting pearly trees in $L_T$ emanated from the the degree $1$ critical points  $(p_{i},\mathbf{1}_{BT})$, and weighted by coefficients $x_{i}, x^\prime_{i} \in \Lambda_{0}^{\times 2n+2}$. The weight is given by a weak bounding cochain $b = \sum_{i=1}^{n+2} x_{i} p_{i} + \sum_{i=1}^n x_i^\prime p_i^\prime$. By \cite[Proposition 4.3 and Lemma 4.4]{KimLauZheng}, $W_T$ is of the form
\begin{equation}
W_T(\mathbf{x})=W(\mathbf{x})+ \sum_{i=1}^{n+2} \lambda_i x_i.
\end{equation}
The function $W$ agrees with the disk potential function of $L$ and a unique gradient flow line from $p_i$ to $\lambda_i$ contributes to the term $\lambda_i x_i$. 

We set $y_i \coloneqq \exp (x_i)$ and $y_i^\prime \coloneqq \exp (x^\prime_i)$, which play a role of the holonomy variables. Also, we consider the holonomy variable $y_{i,j}$ corresponding to the orbit generated by $\Phi_{i,j}$. We then have the following coordinate change formulas
\begin{enumerate}
\item $y_i = \frac{y_{i,1}y_{i-1,2}}{y_{i-1,1}y_{i-2,2}}$ for $i = 1, 2, \cdots, n+2$ and 
\item $y_i^\prime = y_{i,2}$ for  $i = 1, 2, \cdots, n$.
\end{enumerate}
We then obtain the expression \eqref{eqn_equiv_disk_potential}. In particular, $W (\mathbf{y})$ (in terms of variables $y_{i,j}$'s) agrees with the disk potential defined in Definition \ref{def_diskpotentfunct} that was computed in \cite{NishinouNoharaUeda}. 
 \end{proof}

Via the Lagrangian correspondence between $L \subset X$ and $L/T \subset Z/G$ in \cite{Fukayaunobstructed,WoodwardXu}, assuming monotonicity and that the correspondence does not involve disk bubbling by pseudo-holomorphic disks of Maslov index zero, the disk potential of $L/T$ should be obtained by restricting the $T$-equivariant disk potential of $L$ to $\log \left( \frac{y_{i,1}y_{i-1,2}}{y_{i-1,1}y_{i-2,2}} \right)=0$, that is, 
\begin{equation}\label{equ_lgpotentialinduced}
W (\mathbf{y}) \big{|}_{{y_{i+1,1}y_{i,2}} = {y_{i,1}y_{i-1,2}}}.
\end{equation}
This is well-known as the Givental--Hori--Vafa's recipe \cite{Givental, HoriVafa}.

By Theorem~\ref{Proposition_HK} and~\ref{proposition_NNUgzbe}, the quotient $L/T$, a fiber of the caterpillar bending system, fits into the above situation. Then \eqref{equ_theorem_main} can be obtained by expressing~\eqref{equ_lgpotentialinduced} in terms of $z_i \coloneqq y_{i,2}$.

\begin{example}
Let $n = 2$ and $\mathbf{r}= (1,1,1,1,1)$. The disk potential function for the Gelfand--Zeitlin torus fiber is
\begin{equation}\label{diskpotentialGZgr25}
W(\mathbf{y}) = \frac{y_{1,1}}{y_{1,2}} + \frac{y_{2,1}}{y_{2,2}} + \frac{y_{3,1}}{y_{3,2}} + \frac{y_{2,1}}{y_{1,1}} + \frac{y_{3,1}}{y_{2,1}} + \frac{1}{y_{3,1}} + y_{1,2} + \frac{y_{2,2}}{y_{1,2}} + \frac{y_{3,2}}{y_{2,2}}
\end{equation}
Since the equivariant disk potential function has four additional logarithmic terms
$$
\lambda_1 \log y_{1,1} + 
\lambda_2 \log \frac{y_{2,1}y_{1,2}}{y_{1,1}} +  \lambda_3 \log \frac{y_{3,1}y_{2,2}}{y_{2,1}y_{1,2}} + \lambda_4 \log \frac{y_{3,2}}{y_{3,1}y_{2,2}},
$$
we then have 
\begin{equation}\label{equ_equfromlogterms}
1 = y_{1,1} = y_{2,1}y_{1,2} = y_{3,1}y_{2,2} = y_{3,2}.
\end{equation}
Respecting our choice of components in the induced completely integrable system in Theorem~\ref{proposition_NNUgzbe}, we set 
\begin{equation}\label{equ_y12y1y22y2}
z_1 \coloneqq y_{1,2} \mbox{ and } z_2 \coloneqq y_{2,2}.
\end{equation}
Combining~\eqref{diskpotentialGZgr25},~\eqref{equ_equfromlogterms}, and~\eqref{equ_y12y1y22y2}, we obtain
$$
W(\mathbf{z}) = \frac{1}{z_1} + \frac{1}{z_1 z_2} + \frac{1}{z_2} + \frac{1}{z_1} + \frac{z_1}{z_2} + z_2 + z_1 + \frac{z_2}{z_1} + \frac{1}{z_2}.
$$
It is exactly same as the expression in~\eqref{equ_theorem_main} for $n =2$.
\end{example}

However, we do not derive Theorem~\ref{theorem_main} by using the correspondence $\pi_*$ in~\eqref{equ_principleGbundlemoduli} because of the following reasons. First of all, depending on a choice of a stability condition (or a level set), the orbit $Z_\mathbb{G}$ changes. Notice that the holomorphic disk bounded by $L$ in $Z_\mathbb{G}$ (not $\mcal{O}_{\lambda}$) gives rise to a holomorphic disk in the quotient. Thus, to see which holomorphic disks in the pre-quotient survives in the disk potential function in the quotient, one needs to check which stratum intersects with the level set $Z$. Also, after taking the reduction, some additional disks not coming from the pre-quotient may occur. For instance, the disk with a sphere bubble can occur in the quotient. For these reasons, the equivariant disk potential function of a Gelfand--Zeitlin torus fiber does \emph{not} directly lead to the disk potential function of a caterpillar bending torus fiber through the process in~\eqref{equ_lgpotentialinduced}.

\subsection{Realization of cluster varieties via SYZ mirror symmetry}\label{section_realizcluster}

In this section, we derive an SYZ mirror of the polygon space $\mcal{M}_\mathbf{r}$ for an equilateral and  generic tuple $\mathbf{r}$ arising from bending systems and describe its cluster structure. The main statement is stated below.

\begin{theorem}\label{theorem_SYZmirrorpolygon}
Suppose that $\mathbf{r}$ is equilateral and generic, and hence $n$ is {even}. Let 
\begin{equation}\label{equcheckYMR}
\check{Y} \coloneqq \mathrm{Gr}(2, \C^{n+3}) \backslash \{ p_{1,2} p_{2,3} \cdots p_{n+2,n+3} p_{1,n+3} = 0\},
\end{equation} 
regarded as an embedded subvariety of $\CP(\wedge^2 \C^{n+3})$ via the Pl\"{u}cker embedding. Then an SYZ mirror of the bending systems on a polygon space $\mcal{M}_\mathbf{r}$ is a Landau--Ginzburg model consisting of
\begin{enumerate}
\item $\check{X} = \check{Y} \cap \{ p_{1,2} = p_{2,3} = \cdots = p_{n+2,n+3} = p_{1,n+3} \}$,
\item $W_{\mathrm{Bend}} = \sum_{i=1}^{n+3} p_{i, i+2} /p_{1,2}$ where indices are taken modulo $n+3$,
\end{enumerate}
where $p_{i,j}$ is the Pl\"{u}cker variables of $\CP(\wedge^2 \C^{n+3})$.  The double indices for $W_{\mathrm{Bend}}$ are read off from labels of the diagonals on the left of Figure~\ref{fig_bendpotential}.
\end{theorem}

\begin{remark}
In general, the union of cluster charts arising from bending systems does not cover the whole mirror space $\check{X}$. It is expected that the mirror charts of immersed Lagrangian submanifolds can cover the missing piece of information as in \cite{HongKimLau}.  The bending system may have a Lagrangian fiber that has an $\R P^3 (\simeq \mathrm{SO}(3))$-factor and one needs to produce the corresponding immersed Lagrangians.
\end{remark}

To motive this result, we take the following natural question into account. 

\begin{question}\label{question_gluedcluster}
Find a symplectic manifold and Lagrangian torus fibrations whose SYZ mirror is a given cluster variety.
\end{question}

A Lagrangian torus fibration naturally produces an algebraic torus as a mirror Maurer--Cartan deformation space. Moreover, if two Lagrangian torus fibrations are related by wall-crossing or a certain type of surgery, the mirror charts are glued by wall-crossing formula, which can be interpreted as a cluster transformation, see \cite{AurouxSpecial, PascaleffTonkonog}. The glued mirror arising from Lagrangian torus fibrations is naturally endowed with a cluster structure. Question~\ref{question_gluedcluster} asks whether one can produce a symplectic model for the given complex cluster mirror variety. 

Perhaps, the most basic step toward to answering this question is to find a symplectic manifold for a cluster variety of \emph{finite type}. Theorem~\ref{theorem_SYZmirrorpolygon} implies that a polygon space with bending systems provides a symplectic model for the cluster variety of $A$ type \emph{without} frozen variables as a mirror.

\begin{corollary}\label{corollary_SYZmirrorpolygon}
The polygon space with the bending systems produces a cluster variety of $A$ type without frozen variables.
\end{corollary}

To explain Theorem~\ref{theorem_SYZmirrorpolygon} and Corollary~\ref{corollary_SYZmirrorpolygon}, we begin by reviewing a construction of Marsh--Rietsch (MR) mirror, restricting to the Grassmannian of two planes in \cite{MarshRietsch}.
The MR mirror of $Y \coloneqq \mathrm{Gr}(2, \C^{n+3})$ is a Landau--Ginzburg (LG) model $\left(\check{Y}, W_\mathrm{MR} \colon \check{Y} \to \C\right)$, whose construction will be recalled below.

The mirror space $\check{Y}$ is given by the complement of an anti-canonical divisor of the dual Grassmannian $\mathrm{Gr}(n+1, \C^{n+3})$. Identifying $\mathrm{Gr}(n+1, \C^{n+3})$ with $\mathrm{Gr}(2, \C^{n+3})$, the anti-canonical divisor is chosen as $p_{1,2} \,  p_{2,3}  \cdots  p_{n+2,n+3} \,  p_{1, n+3} = 0$. We then have the mirror space $\check{Y}$ in~\eqref{equcheckYMR}.
To describe $W_\mathrm{MR}$, we need to review some quantum Schubert calculus. Let $\mcal{I}$ be the set of indices for the Pl\"{u}cker variables of $\check{Y}$. Then $\mcal{I}$ is grouped into the set of indices for the frozen variables and the set of indices for the unfrozen variables$\colon$
\begin{enumerate}
\item (frozen variables) $\mcal{I}_{\mathrm{fr}} = \{(1,2), (2,3), \cdots, (n+2,n+3), (1, n+3) \}$,
\item (unfrozen variables)  $\mcal{I}_{\mathrm{uf}} = \mcal{I} \backslash \mcal{I}_{\mathrm{fr}}$.
\end{enumerate}
Each index $I$ can be regarded as a Young diagram by the following rule. Consider the rectangular box consisting of $(n+1) \times 2$ unit boxes. To each ${I}$, we associate a (unique) shortest path from the vertex at the upper-right corner to the vertex at the lower-left corner in the rectangular box in such a way that its horizontal steps correspond to the components of $I$. The path divides the rectangular box into two pieces. We then obtain a Young diagram corresponding to ${I}$ by taking the upper-left part above the path. Conversely, we can associate a shortest path in the rectangular box to each Young diagram by locating the diagram at the upper-left corner in the rectangular box. We have a double index by reading its horizontal step. 

For each index $I \in \mcal{I}_{\mathrm{fr}}$ we associate the index $I^\prime \in \mcal{I}$ as follows. Consider the action of the first Chern class  $c_1 (T \mathrm{Gr}(2, \C^{n+3}))$ on a Schubert class via the quantum cup product, which leads to a sum of Schubert classes. The output class can be described in terms of combinatorics of Young diagrams via the quantum Monk's rule. The quantum cup product of the first Chern class and  a Schubert class corresponding to each \emph{frozen} variable corresponds to a \emph{single} Young diagram. We denote by $I^\prime$ the index for this output Young diagram. Specifically, for $I \neq (n+2, n+3)$, the Young diagram of ${I}^\prime$ is obtained by adding a single box to the Young diagram associated to ${I}$. For $I = (n+2, n+3)$, the Young diagram of ${I}^\prime$ is obtained by removing a rim hook from the Young diagram associated to ${I}$. Then the superpotential $W_\mathrm{MR}$ is a regular function on $\check{Y}$ determined by
\begin{equation}\label{equation_MRmirror}
W_\mathrm{MR} (\mathbf{p})  \coloneqq \sum \frac{p_{{I}^\prime}}{p_{{I}}} \colon \check{Y} \to \C
\end{equation}
where the summation runs over all $I \in \mcal{I}_{\mathrm{fr}}$. The above combinatorial rule yields that $I^\prime = (j, j +2)$ for $I = (j, j+1)$. Therefore, the superpotential~\eqref{equation_MRmirror} becomes
$$
W_\mathrm{MR} (\mathbf{p}) \coloneqq \sum \frac{p_{{j,j+2}}}{p_{{j,j+1}}} \colon \check{Y} \to \C.
$$

The MR mirror $(\check{Y}, W_\mathrm{MR})$ has a cluster structure. For each triangulation $\mcal{T}$ of an $(n+3)$-gon obtained by adding diagonals, let  $\mcal{I}_\mcal{T}$ be the set of indices associated to a diagonal in $\mcal{T}$. Then the cluster chart associated to $\mcal{T}$ is  
$$
\mcal{U}_\mcal{T} \coloneqq (\C^*)^{2n+2} \simeq \mathrm{Spec} \left( \left( \frac{p_{I}}{p_{1, 2}} \right)^{\pm} \mid I \in \mcal{I}_{\mcal{T}} \cup \mcal{I}_{\mathrm{fr}}  \right)
$$
By restricting $W_\mathrm{MR}$ to the chart $\mcal{U}_\mcal{T}$ and expressing $W_\mathrm{MR}$ in terms of Pl\"{u}cker variables associated to $\mcal{I}_{\mcal{T}} \cup \mcal{I}_{\mathrm{fr}}$, we obtain a Laurent polynomial $W_\mcal{T} \coloneqq W_\mathrm{MR} |_{\mcal{U}_\mcal{T}} \colon \mcal{U}_\mcal{T}  \to \C$. The pair $(\mcal{U}_\mcal{T}, W_\mcal{T})$ forms a cluster chart for the MR mirror. Furthermore, suppose that two triangulations $\mcal{T}$ and $\mcal{T}^\prime$ are related by Whitehead move so that
$$
\mcal{I}_{\mcal{T}^\prime} = \left( \mcal{I}_{\mcal{T}} - \{(a,c)\}  \right) \cup \{ (b, d) \}
$$ 
where $\mcal{I} - \mcal{I} \cap \mcal{I}^\prime = \{ (a,c) \}$ and $\mcal{I}^\prime - \mcal{I} \cap \mcal{I}^\prime = \{ (b,d) \}$ as in Figure~\ref{fig_whiteheadmove}. These cluster charts are related by the Pl\"{u}cker relation $p_{a,b} p_{c,d} - p_{a,c} p_{b,d} + p_{a,d} p_{b,c} = 0$.

\begin{figure}[h]
	\scalebox{0.85}{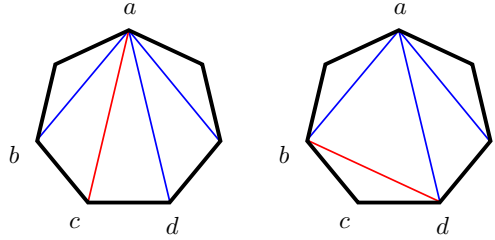}
	\caption{\label{fig_whiteheadmove} Whitehead move}	
\end{figure}

We now take the triangulation with the diagonals $d_1, d_2, \cdots, d_n$ in~\eqref{equ_djdj1} as an initial seed. By considering its dual graph and giving an orientation in a suitable manner, the triangulation can be turned into a quiver. The mutation relation given by the Pl\"{u}cker relation is compatible with the quiver mutation. By computing the exchange matrix of the initial seed and its Cartan counterpart, the cluster structure is of $A$ type. There is a one-to-one correspondence between the set of triangulations and the set of seeds for the cluster structure.

We discuss a symplectic model for the MR mirror $(\check{Y}, W_\mathrm{MR})$ respecting the cluster structure. By a series of works by Nohara--Ueda, its symplectic model is a Grassmannian with generalized Gelfand--Zeitlin (GGZ) systems in the following sense. For each triangulation $\mcal{T}$ of the $(n+3)$-gon, Nohara--Ueda \cite{NoharaUedaGrasPoly} constructed a generalized Gelfand--Zeitlin (GGZ) system $\Phi_{\lambda, \mcal{T}}$ on $\mcal{O}_{\lambda}$ and computed the disk potential function of the GGZ system. In \cite{NoharaUedaGraclus}, they showed that the Floer theoretical disk potential function on each cluster chart agrees with the superpotential $W_\mathrm{MR}$ restricted to the chart (up to a coordinate change). Moreover, the wall-crossing formulas of the torus fibers of GGZ arising are equal to the cluster mutations of the MR mirror $(\check{Y}, W_\mathrm{MR})$. Therefore, in this case, the SYZ mirror of the Grassmannian agrees with the Lie theoretical MR mirror (up to a codimension two).

What we have discussed can be summarized below.

\begin{theorem}[\cite{NoharaUedaGraclus}]\label{theorem_GGZcluster}
Let $\mcal{T}$ be a triangulation of $(n+3)$-gon by inserting diagonals.
We denote by ${\Phi}_{\lambda, \mcal{T}}$ the generalized Gelfand--Zeitlin system associated to $\mcal{T}$. 
Let $\mathbf{u}_0$ be the center of the polytope $\Phi_{\lambda, \mcal{T}}(\mcal{O}_\lambda)$. Then the disk potential function of the torus fiber ${\Phi}_{\lambda, \mcal{T}}^{-1}(\mathbf{u}_0)$ is
$$
W_\mathrm{MR} (\mathbf{p}) |_{\mcal{U}_{\mcal{T}}}
$$
in terms of the Pl\"{u}cker variables corresponding to edges in $\mcal{T}$.
\end{theorem}

We discuss a symplectic model for the cluster variety of type $A$ \emph{without} frozen variables. According to Theorem~\ref{proposition_NNUgzbe}, the Gelfand--Zeitlin system on $\mcal{O}_{\lambda}$ descends to the caterpillar bending system on $\mcal{M}_\mathbf{r}$. Moreover, on the polygon space $\mcal{M}_\mathbf{r}$, a different choice of triangulation of the $(n+3)$-gon leads to a different bending system. More generally, a GGZ system on $\mcal{O}_{\lambda}$ descends to a bending integrable system on $\mcal{M}_{\mathbf{r}}$. Using this relation, we want to understand mirror symmetry of $\mcal{M}_\mathbf{r}$.

On the complex mirror side, let us produce a LG model for a polygon space $\mcal{M}_\mathbf{r}$ from the MR mirror for $\mcal{O}_{\lambda}$. The mirror space for the LG model is obtained by the intersection $\check{Y}$ with the hyperplanes obtained by equating all frozen variables. The superpotential is obtained by clearing out all denominators in~\eqref{equation_MRmirror}. That is, the LG model $(\check{X}, W_{\mathrm{Bend}})$ consists of
\begin{enumerate}
\item $\check{X} = \check{Y} \cap \{ p_{1,2} = p_{2,3} = \cdots = p_{n+2,n+3} = p_{1,n+3} \}$ and
\item $W_{\mathrm{Bend}}  (\mathbf{p}) = \sum p_{I^\prime} / p_{1,2}$ where the summation is runs over all $I \in \mcal{I}_\mathrm{frozen}$.
\end{enumerate}
The Pl\"{u}cker relations on $\mcal{O}_{\lambda}$ induce the mutation relations on $\check{X}$. Since all frozen variables are eliminated, we realize the cluster variety of type $A$ without frozen variable on the complex side.

Now, we explain why the above LG model $(\check{X}, W_{\mathrm{Bend}})$ is an SYZ mirror for $\mcal{M}_\mathbf{r}$. By \cite[Proposition 6.10]{MarshRietsch}, recall that the GZ cluster chart $\mcal{U}_{\mcal{T}_0}$ is embedded into the MR mirror via
$$
y_{j,1} \mapsto \frac{p_{j+1,j+2}}{p_{1,j+1}},\, y_{j,2} \mapsto \frac{p_{1,j+2}}{p_{1,2}}.
$$
Regarding $\mcal{U}_{\mcal{T}_0}$ as a Maurer--Cartan deformation space of the torus fiber of $\Phi_{\lambda, \mcal{T}_0}$, $y_{i,j}$ is the holonomy variable of a loop of the $\mathbb{S}^1$-action generated by $\Phi_{i,j}$ in~\eqref{equ_GZsystem}.

The moment map of $T_{\mathrm{U}(n+3)}$-action on $\mcal{O}_{\lambda}$  can be taken as
\begin{equation}\label{equ_components}
\Phi_{1,1}, \Phi_{1,2} + \Phi_{2,1} - \Phi_{1,1}, \cdots, \Phi_{n+1,2} + | \mathbf{r} | - \Phi_{n,2} - \Phi_{n+1,1}, 2| \mathbf{r} | - \Phi_{n+1,2} - |\mathbf{r}|.
\end{equation}
Each component generates an $\mathbb{S}^1$-action on $\mcal{O}_{\lambda}$. We denote by the loops
$$
\vartheta_1, \vartheta_2, \cdots, \vartheta_{n+2},  \vartheta_{n+3} 
$$
generated by the components~\eqref{equ_components}, respectively. Note that $[\vartheta_{n+3}] = [- \vartheta_1 - \cdots - \vartheta_{n+2}]$ as an element of the fundamental group of each bending torus fiber.

In terms of the holonomy variables on $\mcal{U}_{\mathcal{T}_0}$, the holonomy variables are equal to
$$
\mathrm{hol}_{\nabla} (\vartheta_j) = \frac{y_{j,1} y_{j-1,2}}{y_{j-1,1} y_{j-2,2}} \quad \mbox{ for $j = 1, 2, \cdots, n+3$}.
$$
Under the identification $\mathrm{Gr}(n+1, \C^{n+3}) \simeq \mathrm{Gr}(2, \C^{n+3})$, the holonomy variables can be expressed
\begin{align*}
&\mathrm{hol}_{\nabla} (\vartheta_j) = \frac{p_{j+1, j+2}}{p_{j,j+1}} \quad \mbox{ for $j = 1, 2, \cdots, n+1$}, \\
&\mathrm{hol}_{\nabla} (\vartheta_{n+2}) = \frac{p_{1,n+3}}{p_{n+2, n+3}},  \, \displaystyle \mathrm{hol}_{\nabla} (\vartheta_{n+3}) = \frac{p_{1,2}}{p_{1,n+3}}
\end{align*} 
in terms of the Pl\"{u}cker variables on $\mathrm{Gr}(2, \C^{n+3})$. All the variables for those holonomy variables arising from the the $T_{\mathrm{U}(n+3)}$-action are frozen variables. Therefore, by taking the symplectic reduction of the $T_{\mathrm{U}(n+3)}$-action, those corresponding holonomy variables are set to be trivial, which leads to the mirror space $\check{X}$. In particular, for each triangulation $\mcal{T}$, its charts are described as 
$$
{\mcal{V}}_\mcal{T} \coloneqq (\C^*)^{n} \simeq \mathrm{Spec} \left( \left( \frac{p_{I}}{p_{1, 2}} \right)^{\pm 1} \mid I \in \mcal{I}_{\mcal{T}}  \right).
$$

The residual torus action defined on a dense subset of $\mcal{M}_\mathbf{r}$ is generated by $\{ \Phi_{1}, \Phi_2, \cdots, \Phi_n \}$ in~\eqref{equ_caterpillarbendingshifted}. Via the following map 
$$
z_i \mapsto \frac{p_{1, i+2}}{p_{1,2}}
$$
from the Maurer--Cartan space of the caterpillar bending system to the cluster chart ${\mcal{V}}_{\mcal{T}_0} \subset \check{X}$, the disk potential in~\eqref{equ_theorem_main} agrees with $W_{\mathrm{Bend}} |_{{\mcal{V}}_{\mcal{T}_0}}$.

To recover the mutation relation on the charts of $\mcal{M}_\mathbf{r}$, we recall that the wall-crossing formula is derived from Floer cohomology between two Lagrangians. Namely, the relation obtained from moduli spaces of holomorphic strips making two objects quasi-isomorphic gives rise to the wall-crossing formula, see \cite{Seidellec, PascaleffTonkonog}. In our situation, take two triangulations $\mcal{T}_1$ and $\mcal{T}_2$ of $(n+3)$-gon by inserting diagonals that are related by  Whitehead move. For $i = 1, 2$, the triangulation $\mcal{T}_i$ generates the GGZ system $\Phi_{\lambda, \mcal{T}_i}$ on $\mcal{O}_\lambda$. The fiber at the center of the polytope $\Phi_{\lambda, \mcal{T}_i}(\mcal{O}_\lambda)$ is a monotone GGZ torus fiber, which is denoted by $L_i$. All GGZ systems on $\mcal{O}_\lambda$ are all invariant under the maximal torus action and moreover so are their interpolations, see \cite[(10.4)]{NoharaUedaGraclus}. 
We then have the moduli space $\mcal{M}({L}_1, {L}_2; \beta)$ of pseudo-holomorphic strips bounded by $L_1$ and $L_2$ is orientation-preserving diffeomorphic to $\mcal{M}({L}_1 / T, {L}_2 / T; \beta)  \times T$. The relation in the pre-quotient gets down that in the quotient, the polygon space $\mcal{M}_\mathbf{r}$ with the set of bending systems. 

\begin{proof}[Proof of Theorem~\ref{theorem_SYZmirrorpolygon}]
Theorem~\ref{theorem_main} computes the disk potential function of one chamber. The mutation relation induced from Theorem~\ref{theorem_GGZcluster} computes the disk potential functions of the other chambers. It completes the proof.
\end{proof}

Let $\mcal{T}$ be a triangulation of $(n+3)$-gon by inserting diagonals. We denote by ${\Phi}_{\mathbf{r}, \mcal{T}}$ the bending system associated to $\mcal{T}$. Each component is a moment map for the $\mathbb{S}^1$-action bending the polygon along a diagonal in $\mcal{T}$. Let $\mathbf{u}_{0,\mcal{T}}$ be the center of the polytope ${\Phi}_{\mathbf{r}, \mcal{T}}(\mcal{M}_\mathbf{r})$. Regard the variable $p_{i,j}$ in the affine chart $p_{1, 2} \neq 0$ as affine coordinate $p_{i,j}/p_{1,2}$ by setting $p_{i,j} \coloneqq p_{i,j}/p_{1,2}$. Then the disk potential function of a monotone torus fiber ${\Phi}_{\mathbf{r}, \mcal{T}}^{-1}(\mathbf{u}_{0,\mcal{T}})$ is obtained from
$$
W_{\mathbf{r}, \mcal{T}} \coloneqq W_\mathrm{Bend} (\mathbf{p}) |_{\mcal{U}_{\mcal{T}}}
$$
by expressing it in terms of the Pl\"{u}cker variables corresponding to diagonals in $\mcal{T}$.

\begin{example}
For the tuple $\mathbf{r} = (1,1,1,1,1)$, the LG model for $\mcal{M}_\mathbf{r}$ is 
\begin{align*}
\check{X} = \check{Y} \cap \{ p_{12} = p_{23} = p_{34} = p_{45} = p_{15} \} \\
W_{\mathrm{Bend}}  (\mathbf{p}) = {p}_{13} + {p}_{24} + {p}_{35} + {p}_{14} + {p}_{25}.
\end{align*}
For $\mcal{T}_0 = \{(1,3), (1,4), (1,5), (1,6)\}$, the disk potential function of ${\Phi}_{\mathbf{r}, \mcal{T}_0}^{-1}(\mathbf{u}_{0,\mcal{T}_0} )$ is
$$
W_{\mathbf{r}, \mcal{T}_0} = p_{13} + \frac{2}{p_{13}} + p_{16} + \frac{2}{p_{16}} + \frac{p_{13}}{p_{14}} + \frac{p_{14}}{p_{13}} + \frac{1}{p_{13}p_{14}} + \frac{p_{14}}{p_{15}} + \frac{p_{15}}{p_{14}} + \frac{1}{p_{14}p_{15}} + \frac{p_{15}}{p_{16}} + \frac{p_{16}}{p_{15}} + \frac{1}{p_{15}p_{16}}.  
$$
For $\mcal{T} = \{(1,3), (1,4), (4,6), (1,6)\}$, the disk potential function of ${\Phi}_{\mathbf{r}, \mcal{T}}^{-1}(\mathbf{u}_{0,\mcal{T}})$ is
$$
W_{\mathbf{r}, \mcal{T}} = p_{13} + \frac{2}{p_{13}} + p_{16} + \frac{2}{p_{16}} + \frac{p_{14}}{p_{13}} + \frac{p_{13}}{p_{14}} + \frac{1}{p_{13}p_{14}} + p_{46} + \frac{2}{p_{46}} + \frac{p_{16}}{p_{14}p_{46}}+  \frac{p_{14}}{p_{16}p_{46}} +  \frac{p_{46}}{p_{14}p_{16}}.
$$
\end{example}

The correspondence we have discussed is summarized as follows. 

\begin{center}
\begin{tabular}{| c || c |} 
\hline
Symplectic side & Complex side\\ 
 \hline \hline
Grassmannian $\mcal{O}_{\lambda}$  & Marsh--Rietsch mirror $(\check{Y}, W_\mathrm{MR})$   \\ 
 \hline
Generalized GZ system $\Phi_{\lambda, \mcal{T}}$ & Cluster chart $\mcal{U}_\mcal{T} \simeq (\C^*)^{2n+2}$   \\ 
 \hline
holonomy variables of global $\mathbb{S}^1$-action &  frozen variables  \\
  \hline
holonomy variables of local $\mathbb{S}^1$-action &  unfrozen variables  \\
 \hline \hline
 Polygon space $\mcal{M}_\mathbf{r}$ & LG model  $(\check{X}, W_{\mathrm{Bend}})$ \\
 \hline
 Bending system ${\Phi}_{\mathbf{r}, \mcal{T}}$ & Cluster chart ${\mcal{V}}_\mcal{T} \simeq (\C^*)^n$ \\
 \hline
 holonomy variables of residual $\mathbb{S}^1$-action &  unfrozen variables  \\
  \hline
\end{tabular}
\end{center}

\section{Caterpillar bending toric varieties}\label{Toricgeometry}

The goal of this section is to compute the second rational relative homology group of the pair consisting of the toric variety associated to the caterpillar bending polytope and its toric fiber. For the purpose of classifying effective disk classes later on, as preliminaries, we shall describe explicit generators for the relative homology group.

\subsection{The toric variety associated to the Gelfand--Zeitlin polytope}

To begin with, we briefly recall the second relative homology group of a compact symplectic toric manifold $X$ as a warm-up. Suppose that the toric manifold $X$ is associated to a Delzant polytope $\Delta$. Let $L$ be a Lagrangian toric fiber of a point in the interior of the polytope $\Delta$ under the moment map $\Phi \colon X \to \Delta$. The second relative homotopy group $\pi_2(X,L) \simeq H_2(X,L; \Z)$ of the pair $(X,L)$ is isomorphic to $\Z^m$  where $m$ is the number of facets of $\Delta$, see \cite[Chapter 1]{FOOOToric2} for instance.

In order to describe a basis for $\pi_2(X,L)$ explicitly, enumerate the facets $f_1, f_2, \cdots, f_m$ of $\Delta$. We denote by $F_i$ the irreducible component of  the toric anti-canonical divisor corresponding to the $1$-cone generated by an inward normal vector to $f_i$. The outward primitive vector normal to $f_i$ generates a Hamiltonian $\bS^1$-action, which is a subtorus action of the global torus action on $X$. A moment map on $X$ for this $\bS^1$-action generates a gradient vector field with respect to the Riemannian metric induced by the toric complex structure and the toric symplectic form. By flowing a chosen free $\bS^1$-orbit in $L$ along the gradient vector field, we produce a gradient holomorphic disk having the chosen orbit as its boundary, intersecting $F_i$ exactly once, and not intersecting the other irreducible components. It has Maslov index two by Proposition~\ref{proposition_maslovformula}. Such a gradient holomorphic disk will be called a \emph{basic disk}. Also, a homotopy class represented by a basic disk is said to be \emph{basic}. We denote by $\beta_i$ the basic class intersecting the divisor $F_i$. Note that there are $m$ many basic classes $\beta_1, \beta_2, \dots, \beta_m$. Since each facet of $\Delta$ can be translated without changing the normal fan structure, the symplectic area of $\beta_i$ can be independently controlled without changing the diffeomorphism type of $X$. It implies that the set of basic classes is linearly independent in $\pi_2(X,L)$.  

We now turn to the Gelfand--Zeitlin polytope $\Delta_{\lambda}$ determined by~\eqref{equation_gzpattern} with $3(n+1)$ facets. We denote by $X$ the Gelfand--Zeitlin toric variety, the toric variety associated to the normal fan of $\Delta_{\lambda}$. By \cite[Section 9]{NishinouNoharaUeda}, we have $\pi_2(X) \simeq \Z$ and $\pi_2(X, L) \simeq \Z^{2n + 3}$. In order to describe a generating set for $\pi_2(X,L)$ explicitly, we take a look at the fan structure of the toric variety $X$. The fan structure for $X$ can be described by the ladder diagram 
\begin{equation}\label{equ_ladderdiagramgr2n}
\Gamma_{\lambda} = \{ (a,b) \in (\Z \times \R) \cup (\R \times \Z) \mid \, 0 \leq a \leq n+1, \, 0 \leq b \leq 2 \}.
\end{equation}
The diagram $\Gamma_{\lambda}$ consists of $2 \times (n+1)$ unit boxes. Note that the number of unit boxes agrees with the complex dimension of $X$. 

Let $\square^{(i,j)}$ be the unit box such that it is contained in $(\Z \times \R) \cup (\R \times \Z)$ and its top-right vertex is located at $(i,j)$. The center of $\square^{(i,j)}$ is denoted by $v_{i,j}$, that is, $v_{i,j} = (i-1/2, j-1/2)$. Then the dual ladder diagram $\Gamma^\vee_{\lambda}$ is defined by the induced subgraph of the grid graph $((\Z + 1/2) \times \R) \cup (\R \times (\Z + 1/2))$ whose vertex set is $ \{{v}_{n+2,1}, {v}_{0,2}\} \cup \{ {v}_{i,j} \mid \square^{(i,j)} \subset \Gamma_{\lambda} \}$. Each edge of $\Gamma^\vee_{\lambda}$ turns into an arrow by orienting a vertical (resp. horizontal) edge from north to south (resp. from west to east). Then the dual ladder diagram $\Gamma^\vee_{\lambda}$ becomes a quiver consisting of $m = 3(n + 1)$ oriented arrows. The first diagram in Figure~\ref{fig_ladderanddualladder} depicts the ladder diagram $\Gamma_{\lambda}$ and the dual ladder diagram $\Gamma^\vee_{\lambda}$ when $n= 2$.

\begin{figure}[h]
	\scalebox{0.85}{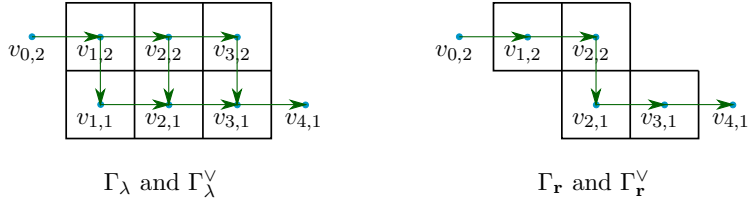}
	\caption{\label{fig_ladderanddualladder} The ladder diagram and dual ladder diagram}	
\end{figure}

Let $N \simeq \Z^{2(n+1)}$ be the free abelian group with the basis $\{ {v}_{i,j} \mid \square^{(i,j)} \subset \Gamma_{\lambda} \}$. 
Let $M$ be the dual lattice of $N$. The dual basis $\{ v_{i,j}^* \mid  \square^{(i,j)} \subset \Gamma_{\lambda} \}$ of the basis $\{ v_{i,j} \mid  \square^{(i,j)} \subset \Gamma_{\lambda} \}$ is a basis for the dual lattice $M$ and it gives rise to an identification $\R^{2(n+1)} \simeq M_\R \coloneqq M \otimes \R$. We denote the coordinate system with respect to the basis $\{ {v}^*_{i,j} \mid \square^{(i,j)} \subset \Gamma_{\lambda} \}$ on $\R^{2(n+1)} \simeq M_\R \coloneqq M \otimes \R$ by $\{ {u}_{i,j} \mid \square^{(i,j)} \subset \Gamma_{\lambda}  \}$. Then the GZ polytope defined by the system~\eqref{equ_halfspacesgzpoly} of inequalities is regarded as a polytope in $M_\R \simeq \R^{2(n+1)}$.


The arrows in the quiver $\Gamma^\vee_{\lambda}$ are in one-to-one correspondence with the inequalities in~\eqref{equ_halfspacesgzpoly} as follows. Setting ${v}_{n+2, 1}$ and ${v}_{0,2}$ to be the zero vectors, we associate to each arrow $\overrightarrow{e}$ the integral vector
\begin{equation}\label{equ_normalvectorthgz}
\mathbf{v}_{\overrightarrow{e}} \coloneqq t_{\overrightarrow{e}} - h_{\overrightarrow{e}} \in N
\end{equation}
where $t_{\overrightarrow{e}}$ (resp. $h_{\overrightarrow{e}}$) is the vector associated to the vertex at the tail (resp. head) of ${\overrightarrow{e}}$. Then each arrow $\overrightarrow{e}$ corresponds to one of the above inequalities whose normal vector is $\mathbf{v}_{\overrightarrow{e}}$. To record the double index, the arrow corresponding to $k_{i,j} (\mathbf{u}) \geq 0$ is denoted by $\overrightarrow{e_{i,j}}$. The integral vector $\mathbf{v}_{\overrightarrow{e}}$ is also denoted by $\mathbf{v}_{i,j}$, which is an outward normal vector to $k_{i,j}(\mathbf{u}) \geq 0$. Observe that the normal vectors satisfy the following relation
\begin{equation}\label{equ_relationvijv}
\mathbf{v}_{i,0} + \mathbf{v}_{i,1} -  \mathbf{v}_{i+1,0} - \mathbf{v}_{i,2} = \mathbf{0} \quad \mbox{for $i = 1, 2, \cdots, n$}.
\end{equation}

Each vector $\mathbf{v}_{i,j}$ generates a Hamiltonian $\bS^1$-action on the smooth loci of $X$. Using this $\bS^1$-action, we produce a gradient holomorphic disk intersecting a toric divisor. We denote by $\beta_{\overrightarrow{e}}$ the basic class intersecting the toric divisor component corresponding to the $1$-cone generated by $\mathbf{v}_{\overrightarrow{e}}$. Namely, the class $\beta_{\overrightarrow{e}}$ is represented by a gradient holomorphic disk generated by the $\bS^1$-action of $\mathbf{v}_{\overrightarrow{e}}$. Again, $\beta_{\overrightarrow{e}}$ is also denoted by $\beta_{i,j}$ if $\mathbf{v}_{\overrightarrow{e}} = \mathbf{v}_{i,j}$. In total, we produce $3(n+1)$ basic disks in $3(n+1)$ basic classes.

We set $m \coloneqq 3(n + 1)$, the number of facets of $\Delta_{\lambda}$. Consider the free abelian group $\Z^m$ generated by the arrows in $\Gamma^\vee_{\lambda}$. We take two group homomorphisms 
\begin{itemize}
\item $\pi \colon \Z^m \to N$ given by $\overrightarrow{e} \mapsto - \mathbf{v}_{\overrightarrow{e}}$,
\item $\pi_\beta \colon \Z^m \to \pi_2(X, L)$ given by $\overrightarrow{e} \mapsto \beta_{\overrightarrow{e}}$.
\end{itemize}
Since the normal vectors of the form~\eqref{equ_normalvectorthgz} generate the lattice $N$, the toric variety $X$ is simply connected by~\cite[Theorem 12.1.10]{CoxLittleSchenckToric} for instance.
We then have the following commuting diagram with exact rows$\colon$  
\begin{equation}\label{equ_exactlattice}
\xymatrix{
0 \ar[r] & {\ker \pi} \ar@{->>}[d]^{\pi_\beta |_{\ker}} \ar[r] & \Z^m  \ar@{->>}[d]^{\pi_\beta}  \ar[r]^{\pi \,\,\,\,\,\,\,\,\,\,\,\,} & N \simeq \Z^{2(n+1)} \ar[d]^\simeq \ar[r] & 0 \\
0 \ar[r] & \pi_2(X)  \ar[r] & \pi_2(X, L)   \ar[r] & \pi_1(L)  \ar[r] & 0.}
\end{equation}

The map $\pi_\beta$ is surjective, but not injective when $n > 1$ unlike the smooth toric case. In other words, the basic classes generate $\pi_2({X},{L})$ where ${L}$ is a Lagrangian toric fiber in ${X}$, but the basic classes are \emph{not} linearly independent in $\pi_2(X,L)$. The relative homotopy group $\pi_2(X,L)$, the image of abelian group, is abelian. Because of~\eqref{equ_relationvijv}, the kernel of $\pi_\beta$ is a free abelian group with a basis consisting of 
\begin{equation}\label{equ_blacksquarejj}
\overrightarrow{e_{i,0}} + \overrightarrow{e_{i,1}} -  \overrightarrow{e_{i+1,0}} - \overrightarrow{e_{i,2}}  \quad \mbox{for $i = 1, 2, \cdots, n$}.
\end{equation} 
The kernel $\ker \pi$ contains $\ker \pi_\beta$ because of the commutativity. The relations~\eqref{equ_blacksquarejj} together with 
$$
\overrightarrow{e_{1,1}} + \overrightarrow{e_{2,1}} + \dots + \overrightarrow{e_{n+1,1}} + \overrightarrow{e_{1,0}} + \overrightarrow{e_{0,2}}
$$
generate the subgroup $\ker \pi$. In sum, we have
\begin{itemize}
\item $\pi_2(X) \simeq \Z \langle \beta_{1,1} + \beta_{2,1} + \dots + \beta_{n+1,1} + \beta_{1,0} + \beta_{0,2} \rangle$.
\item $\pi_2(X, L) \simeq \Z \langle \beta_{1,0}, \cdots,  \beta_{n+1,0}, \beta_{1,1}, \cdots, \beta_{n+1,1}, \beta_{0,2}, \cdots, \beta_{n,2} \rangle / \langle \blacksquare_1, \blacksquare_2, \cdots,  \blacksquare_n \rangle$
where
\begin{equation}\label{equ_blacksquarbetajj}
\blacksquare_i \coloneqq {\beta_{i,0}} + {\beta_{i,1}} -  {\beta_{i+1,0}} - {\beta_{i,2}}  \quad \mbox{for $i = 1, 2, \cdots, n$}.
\end{equation} 
\end{itemize}

\subsection{The toric variety associated to the caterpillar bending polytope}\label{subsectionthetoricvarcater}

Let $X$ be the toric variety associated to the caterpillar bending polytope $\Delta_\mathbf{r}$, the image of the caterpillar bending system (up to a translation) with the choice $\mathbf{r}$ in~\eqref{equ_equilateral111}. Let $L$ be a Lagrangian toric fiber, located at a point in the interior of  $\Delta_\mathbf{r}$. In order to compute the second rational relative homology group of $(X, L)$, let us describe the normal fan of $\Delta_\mathbf{r}$.

For $a, b, c$, and $d \in \Z$, we denote by $\overline{(a,b), (c,d)}$ the line segment connecting $(a,b)$ to $(c,d)$ in $\Z^2$. The ladder diagram $\Gamma_\mathbf{r}$ is defined as
$$
\Gamma_\mathbf{r} \coloneqq \Gamma_{\lambda} - \left\{ \overline{(0,0), (1,0)}, \overline{(0,0), (0,1)}, \overline{(n+1,1), (n+1,2)}, \overline{(n,2), (n+1,2)} \right\}
$$
where $\lambda$ is defined in~\eqref{equ_lambda}. The diagram $\Gamma_\mathbf{r}$ consists of $2n$ unit boxes. Namely, $\Gamma_\mathbf{r}$ is obtained by removing the unit box at the left bottom corner and the unit box at the right top corner from $\Gamma_{\lambda}$ as in the second diagram of Figure~\ref{fig_ladderanddualladder}.

Recall the description of $\Delta_\mathbf{r}$ in~\eqref{equation_polypattern}, set ${w}_{n+2,1} \coloneqq | \mathbf{r}|, {w}_{n+1,2} \coloneqq | \mathbf{r}| - 2 r_{n+3}$, ${w}_{0,2} \coloneqq 0$, $w_{1,1} \coloneqq 2 | \mathbf{r} |_1$, and 
$$
\begin{cases}
w_{i,2} \coloneqq u_i &\mbox{ for $i = 1, 2, \dots, n$,} \\
w_{i,1} \coloneqq 2 | \mathbf{r} |_{i} - {u}_{i-1}  &\mbox{ for $i = 2, 3, \dots, n+1$.}
\end{cases}
$$ 
By Proposition~\ref{proposition_monotone}, recall that the polytope $\Delta_\mathbf{r}$ is defined by the list of inequalities$\colon$
\begin{equation}\label{equ_definingequellijpoly}
\begin{cases}
\ell_{i,0} ({\mathbf{u}}) = {w}_{i,1} - {w}_{i,2} \geq 0 &\mbox{ for $i = \mathbf{2}, \dots, \mathbf{n}$,}\\
\ell_{i,1} ({\mathbf{u}}) = {w}_{i+1,1} - {w}_{i,1} \geq 0 &\mbox{ for $i = \mathbf{2}, \dots, n, n+1$,}\\
\ell_{i,2} ({\mathbf{u}}) = {w}_{i+1,2} - {w}_{i,2} \geq 0 &\mbox{ for $i = 0, 1, \dots, \mathbf{n-1}$,}
\end{cases}
\end{equation}
and each hyperplane contains a facet. Set $m \coloneqq  (3n + 3 - 4)$, the number of facets of $\Delta_\mathbf{r}$.

Let $v_{i,j} \coloneqq (i - 1/2, j - 1/2)$ and $\square^{(i,j)}$ be the unit box whose top-right vertex is located at $(i,j)$. Let $\Gamma_\mathbf{r}^\vee$ be the dual ladder diagram, the induced subgraph of the gird graph $((\Z + 1/2) \times \R) \cup (\R \times (\Z + 1/2))$ whose vertex set is $\{{v}_{1, n+2}, {v}_{2,0}\} \cup \{ {v}_{i,j} \mid \square^{(i,j)} \subset \Gamma_\mathbf{r} \}$. Each edge of $\Gamma_\mathbf{r}^\vee$ turns into an arrow by orienting a vertical (resp. horizontal) edge from north to south (resp. from west to east).

Let $N \simeq \Z^n$ be the free abelian group with the basis 
\begin{equation}\label{equ_abasisforzn}
\{ v_i \coloneqq v_{i,2} \mid i = 1, 2, \cdots, n \}.
\end{equation}
Let $M$ be the dual lattice of $N$. The dual basis $\{ v_{i}^* \mid i = 1, 2, \cdots, n \}$ of the basis $\{ v_{i} \mid  i = 1, 2, \cdots, n \}$ is a basis for the dual lattice $M$ and the dual basis gives rise to the identification $\R^{n} \simeq M_\R \coloneqq M \otimes \R$. We denote the coordinate system associated with the basis $\{ {v}^*_{i} \mid  i = 1, 2, \cdots, n \}$ on $\R^{n} \simeq M_\R \coloneqq M \otimes \R$ by $\{ {u}_{i} \mid  i = 1, 2, \cdots, n \}$. The caterpillar bending polytope defined by the system~\eqref{equ_elli0i1i2} of inequalities is regarded as a polytope in $M_\R \simeq \R^{n}$.

The arrows in $\Gamma^\vee_{\mathbf{r}}$ are in one-to-one correspondence with the inequalities in~\eqref{equ_definingequellijpoly} as follows. To reflect the description of $\Delta_\mathbf{r}$ in~\eqref{equ_definingequellijpoly}, set $v_{i+1,1} \coloneqq - v_{i}$ for $i = 1, 2, \cdots, n$ and set $v_{n+2,1}$ and $v_{0,2}$ to be the zero vectors. We associate to each arrow $\overrightarrow{e}$ the integral vector 
\begin{equation}\label{equ_normalvectorth}
\mathbf{v}_{\overrightarrow{e}} \coloneqq t_{\overrightarrow{e}} - h_{\overrightarrow{e}}\in N
\end{equation}
as in~\eqref{equ_normalvectorthgz}
where $t_{\overrightarrow{e}}$ (resp. $h_{\overrightarrow{e}}$) is the vector associated to the vertex at the tail (resp. head) of ${\overrightarrow{e}}$. Namely, each arrow $\overrightarrow{e}$ corresponds to one of the above inequalities whose normal vector is $\mathbf{v}_{\overrightarrow{e}}$. To record the double index, the arrow corresponding to $\ell_{i,j} (\mathbf{u}) \geq 0$ is denoted by $\overrightarrow{e_{i,j}}$. The integral vector $\mathbf{v}_{\overrightarrow{e}}$ is also denoted by $\mathbf{v}_{i,j}$ if $\overrightarrow{e} = \overrightarrow{e_{i,j}}$.

Each vector $\mathbf{v}_{i,j}$ generates a gradient holomorphic disk $\varphi_{i,j} \colon (\mathbb{D}, \partial \mathbb{D}) \to (X, L)$ generated by the Hamiltonian $\bS^1$-action corresponding to $\mathbf{v}_{i,j}$. We call the class generated by a gradient disk $\varphi_{i,j}$ \emph{basic} in $\pi_2(X, L)$. The class represented by a gradient disk $\varphi_{i,j}$ via the Hurewicz map in $H_2(X,L)$ is also called \emph{basic}. This class is denoted by $\beta_{i,j}$.

Then the second relative homology group of the pair $(X, L)$ is characterized as follows.

\begin{proposition}\label{proposition_pi2xl}
Let $X$ be the toric variety associated to the caterpillar bending polytope $\Delta_\mathbf{r}$ and $L$ a Lagrangian toric fiber. Then we have 
$$
H_2(X,L;\Q) \simeq \Q^{2n+1}.
$$ 
Moreover, the homology group $H_2(X,L;\Q)$ is generated by the basic gradient disks.
\end{proposition}

\begin{proof}
By~\cite[Theorem 12.1.10]{CoxLittleSchenckToric}, the fundamental group of the toric variety $X$ is a torsion subgroup. Then, by the universal coefficient theorem, $H_1(X; \Q)$ is trivial. By the long exact sequence of homology groups of the pair $(X,L)$, we then have
$$
H_2(X,L;\Q) \simeq H_2(X;\Q) \oplus H_1(L;\Q) \simeq H_2(X;\Q) \oplus \Q^{n}.
$$

To verify that $H_2(X;\Q) \simeq \Q^{n+1}$, consider a simplicialization $\widetilde{X}$ of $X$, obtained by refining the fan associated to $X$ \emph{without} adding any $1$-cones. It comes with the toric morphism $\Pi \colon \widetilde{X} \to X$, see \cite[Proposition 11.1.7]{CoxLittleSchenckToric}. In other words, $X$ is covered by the toric orbifold $\widetilde{X}$. Although the toric morphism $\Pi$ does not preserve symplectic structures, the inverse image $\Pi^{-1}(L) \eqqcolon \widetilde{L}$ is a Lagrangian toric fiber as $\Pi$ sends the algebraic torus of $\widetilde{X}$ to that of $X$. Moreover, the basic classes of $(\widetilde{X}, \widetilde{L})$ are in one-to-one correspondence to those of $(X, L)$ via the induced map $\Pi_*$.

Note that $H_2(X;\Q)$ is generated by torus invariant two cycles of $X$. They can be lifted to two cycles in $\widetilde{X}$. Each two cycle in the toric orbifold $\widetilde{X}$ can be expressed as a non-negative rational linear combination of basic classes of $H_2( \widetilde{X}, \widetilde{L}; \Q)$ by \cite[Section 9.1]{ChoPoddar}. Every two cycle of $X$ can be also expressed as a non-negative rational linear combination of basic classes of $H_2( {X}, {L}; \Q)$ by taking the image under $\Pi_*$. In other words, the linear map  
\begin{equation}\label{equ_compositionqm}
\pi_\beta \colon \Q^m \to H_2(X,L;\Q) \quad \mbox{ $\overrightarrow{e_{i,j}} \mapsto \beta_{i,j}$}
\end{equation}
is surjective.  If $n \geq 3$, this map is \emph{not} injective so that the map has a non-trivial kernel. As the kernel contains the subspace generated by
\begin{equation}\label{equ_blacksquarbetajjcbs}
\blacksquare_i \coloneqq {\beta_{i,0}} + {\beta_{i,1}} -  {\beta_{i+1,0}} - {\beta_{i,2}}  \quad \mbox{for $i = 2, 3, \cdots, n-1$},
\end{equation} 
we then have
\begin{equation}\label{equletdim}
\dim_\Q H_2(X,L;\Q) \leq m - (n-2) = 2n + 1 \mbox{ and hence } \dim_\Q H_2(X; \Q) \leq n+1.
\end{equation}
Thus, it remains to show that $\dim_\Q H_2(X; \Q) = n+1$.

To show that the dimension of $H_2(X,L;\Q)$ is exactly ${2n+1}$, setting $\beta_{1,1} \coloneqq (\beta_{2,0} + \beta_{1,2})/2$ and $\beta_{n,2} \coloneqq (\beta_{n,1} + \beta_{n,0})/ 2$, let 
\begin{equation}\label{equclassholosphere}
\square_i \coloneqq \beta_{i,1} + \beta_{i-1,2}  \quad  \mbox{ for $i = 1, 2, \cdots, n+1$}.
\end{equation}
For $i = 2, 3, \cdots, n$, consider Hamiltonian $\bS^1$-actions corresponding to $\textbf{v}_{i,1} = - \textbf{v}_{i-1,2}$ and to $\textbf{v}_{i-1,2}$, respectively. We take an $\bS^1$-orbit in $L$ and then flow the orbit into two directions via the gradient flows obtained by the moment maps for the above actions. It generates a gradient sphere intersecting the inverse image of a facet given by $\ell_{i,1}(\mathbf{u}) = 0$ and that by $\ell_{i-1,2}(\mathbf{u}) = 0$. In particular, the class is the sum of basic classes. When $i = 1$ or $n+1$, the classes $2 \square_1$ and $2 \square_{n+1}$ can be expressed as follows$\colon$
$$
\begin{cases}
2 \square_1 = 2 \beta_{1,1} + 2 \beta_{0,2} = \beta_{2,0} + \beta_{1,2} + 2 \beta_{0,2} \\
2 \square_{n+1} = 2 \beta_{n+1,1} + 2 \beta_{n+1,2} = \beta_{n,1} + \beta_{n,0} + 2 \beta_{n+1,2}.
\end{cases}
$$
Thus, it can be represented as the sum of basic classes. By translating a single facet of $\Delta_\mathbf{r}$ slightly, we obtain the same toric variety as the translating does not change the normal fan structure. Also, by adjusting $\mathbf{r}$, one can control the area of the above spherical classes linearly independently. Therefore, 
\begin{equation}\label{equgetdim}
\dim_\Q H_2(X; \Q) \geq n+1.
\end{equation}
In conclusion, by~\eqref{equletdim} and~\eqref{equgetdim}, we have shown that $H_2(X; \Q) \simeq \Q^{n+1}$.
\end{proof}

In sum, we have the following commutative diagram$\colon$
\begin{equation}
\xymatrix{
0 \ar[r] & {\ker \pi_\Q} \simeq \Q^{m-n} \ar[r] \ar[d]^{\pi_\beta |_\mathrm{ker}} & \Q^m \ar[r]^{\pi_\Q \quad} \ar[d]^{\pi_\beta} & N \otimes \Q \simeq \Q^n \ar[r] \ar[d] & 0 \\
0 \ar[r] & {H_2(X;\Q)} \simeq \Q^{n+1} \ar[r] & H_2(X,L; \Q) \simeq \Q^{2n+1} \ar[r]_{\partial } & H_1(L; \Q) \simeq \Q^n \ar[r] & 0
} 
\end{equation}
where $\pi_\Q \colon \Q^m \to N \otimes \Q$ is determined by $\overrightarrow{e} \mapsto - \mathbf{v}_{\overrightarrow{e}}$. Moreover,
$$
H_2(X, L; \Q) \simeq \Q \langle \beta_{i,j} \mid \overrightarrow{e_{i,j}} \in \Gamma^\vee_{\mathbf{r}}  \rangle / \Q \langle \blacksquare_i \mid i = 2, 3, \cdots, n-1\rangle.
$$

\section{Classification of effective disk classes}\label{sec_classificationofeffec}

The aim of this section is to classify homotopy classes that can be realized as a holomorphic disk bounded by a monotone Lagrangian torus fiber of the caterpillar bending system on the polygon space. Each Laurent monomial corresponds to a lattice point in $N$ and the lattice points that can be realized as a holomorphic disk of Maslov index two will be classified.

\subsection{Toric degenerations of caterpillar bending systems}

We relate the toric completely integrable system on the toric variety of the caterpillar bending polytope in Section~\ref{subsectionthetoricvarcater} with the caterpillar bending system on the polygon space via a toric degeneration.  

We begin by reviewing a notion of a toric degeneration of completely integrable system in \cite{NishinouNoharaUeda}. Consider a smooth projective variety $X$ equipped with a K\"{a}hler form $\omega$. Suppose that $\Phi \colon X \to \R^N$ be a completely integrable system on $(X, \omega)$ such that $\Phi(X)$ is a polytope $\Delta$. A \emph{toric degeneration} of the system $\Phi$ consists of the following data$\colon$ 
\begin{enumerate}
\item a flat family $\frak{X} \coloneqq \{ {X}^\tau \}_{\tau \in \C}$ of algebraic varieties over $\C$,
\item a K\"{a}hler form $\omega_\frak{X}$ on $\frak{X}$,
\item a piecewise smooth path $\gamma \colon [0,1] \to \C$ from $\gamma(0) = 0$ to $\gamma(1) = 1$  such that
\begin{itemize}
\item for each $t \in [0,1]$, $({X}^t, \omega^t) \coloneqq ({X}^{\gamma(t)}, \omega_\frak{X} |_{{X}^{\gamma(t)}})$ is a K\"{a}hler variety,
\item $ ({X}^1, \omega^1)$ is isomorphic to the given projective manifold $({X}, \omega)$,
\item $({X}^0, \omega^0)$ is the toric variety associated to the polytope $\Delta$,
\end{itemize}
\item a family $\Phi^{t} \colon {X}^t \to \R^N$ of completely integrable systems such that
\begin{itemize}
\item for each $t$, the image of $X^t$ under the map $\Phi^t$ is $\Delta$,
\item $\Phi^1$ agrees with $\Phi$ under the isomorphism $X \simeq X^1$,
\item $\Phi^0$ is the moment for the toric action on $X^0$ with the moment polytope $\Delta$.
\end{itemize}
\item a map $\psi_{t,s} \colon {{X}}^t \to {{X}}^s$ satisfying $\Phi^t = \Phi^s \circ \psi_{t,s}$.
\end{enumerate}

Based on toric degenerations of partial flag manifolds constructed in \cite{GonciuleaLakshmibai, KoganMiller}, Nishinou--Nohara--Ueda constructed a toric degeneration of Gelfand--Zeitlin system.

\begin{theorem}[Theorem 1.2 in \cite{NishinouNoharaUeda}]\label{theorem_toricdegegz} 
The Gelfand--Zeitlin system on $\mcal{O}_{\lambda}$ admits a toric degeneration. In particular, for each pair $(t,s)$ of real numbers in $[0,1]$ with $t > s$, we have the following commutative diagram
\begin{equation}\label{equ_commutativitiyofcompletelyinteGZ}
\xymatrix{
{{X}}^t \ar[rr]^{\psi_{t,s}} \ar[rd]_{\Phi^t} & & {{X}}^s \ar[ld]^{\Phi^s}  \\
& {\Delta}_{\lambda} &  }
\end{equation} 
\begin{itemize}
\item $\Phi^1$ agrees with the Gelfand--Zeitlin system $\Phi_{\lambda}$ in~\eqref{equ_GZsystem} via $X^1 \simeq \mcal{O}_{\lambda}$,
\item $\Delta_{\lambda}$ is the Gelfand--Zeitlin polytope,
\item $\Phi^0$ is the moment map associated to ${\Delta}_{\lambda}$ for the torus action on ${X}^0$,
\item $\psi_{t,s}$ is a globally continuous map and is a densely symplectomorphism.
\end{itemize}
\end{theorem}

The toric degeneration of~\eqref{equ_commutativitiyofcompletelyinteGZ} is ${T}_{\mathrm{U}(n+3)}$ invariant. It induces the toric degeneration of the caterpillar bending system on $\mcal{M}_\mathbf{r}$ by Theorem~\ref{Proposition_HK} and~\ref{proposition_NNUgzbe} as in \cite{NoharaUedaGrasPoly}. By abuse of notation, we employ the same notations for the toric degeneration~\eqref{equ_commutativitiyofcompletelyinteGZ} to denote the toric degeneration  of the caterpillar bending system on $\mcal{M}_\mathbf{r}$ in~\eqref{equ_commutativitiyofcompletelyinte}.

\begin{corollary}\label{theorem_toricdegepolygon}
The caterpillar bending system on $\mcal{M}_\mathbf{r}$ admits a toric degeneration. In particular, for each pair $(t,s)$ of real numbers in $[0,1]$ with $t > s$, we have the following commutative diagram
\begin{equation}\label{equ_commutativitiyofcompletelyinte}
\xymatrix{
{{X}}^t \ar[rr]^{\psi_{t,s}} \ar[rd]_{\Phi^t} & & {{X}}^s \ar[ld]^{\Phi^s}  \\
& {\Delta}_{\mathbf{r}} &  }
\end{equation} 
\begin{itemize}
\item $\Phi^1$ agrees with the caterpillar bending system $\Phi_{\mathbf{r}}$ in~\eqref{equ_caterpillarbendingshifted} via $X^1 \simeq {T} \bbslash_{2\mathbf{r}} \mcal{O}_{\lambda} $,
\item $\Delta_\mathbf{r}$ is the caterpillar bending polytope.
\item $\Phi^0$ is the moment map associated to ${\Delta}_{\mathbf{r}}$ for the torus action on ${X}^0$.
\item $\psi_{t,s}$ is a globally continuous map and is a densely symplectomorphism.
\end{itemize}
\end{corollary}

\subsection{Computation of homotopy groups}\label{sec_computationofhomotopygroups}

Fix an equilateral and generic tuple $\mathbf{r}$ so that the polygon space $\mathcal{M}_ \mathbf{r}$ associated to $\mathbf{r}$ is a monotone symplectic manifold. Recall that $\mcal{M}_\mathbf{r}$ is isomorphic to ${T} \bbslash_{2\mathbf{r}} \mcal{O}_{\lambda}$. Let $L$ be any torus fiber of the caterpillar bending system $\Phi_\mathbf{r}$.

We first compute the relative homotopy group of the pair $(\mathcal{M}_ \mathbf{r}, L)$.

\begin{lemma}\label{lemma_secondhomotopygpz2n3}
If a tuple $\mathbf{r}$ is equilateral and generic, then the second homotopy group of the pair $\left( \mathcal{M}_\mathbf{r}, L \right)$ is isomorphic to $\Z^{2n+3}$, i.e., 
\begin{equation}\label{equ_pi2x1l1}
\pi_2 \left( \mathcal{M}_\mathbf{r}, L \right) \simeq \pi_1\left(L \right) \oplus \pi_2 \left(  \mathcal{M}_\mathbf{r} \right) \simeq \Z^n \oplus \Z^{n+3}.
\end{equation}
\end{lemma}

\begin{proof}
The symplectic reduction of the Grassmannian $\mathrm{Gr}(2, \C^{n + 3}) \simeq \mcal{O}_{\lambda}$, which is a monotone symplectic manifold, at the equilateral level is also monotone. The polygon space $\mathcal{M}_\mathbf{r}$ is monotone and hence it is a Fano manifold. Since every Fano manifold is simply connected, the polygon space $\mathcal{M}_\mathbf{r}$ is simply connected. Then the long exact sequence of the pair $\left( \mathcal{M}_\mathbf{r}, L \right)$ yields that $\pi_2 \left( \mathcal{M}_\mathbf{r}, L \right) \simeq \pi_1\left(L \right) \oplus \pi_2 \left( \mathcal{M}_\mathbf{r} \right)$. 

The cohomology groups $H^\bullet \left( \mathcal{M}_\mathbf{r} ; \Z \right)$ vanish for all odd degrees, see \cite[Corollary 2.2.2]{Klyachko} \cite[Lemma 4.1]{HausmannKnutson2} for instance. In particular, by the universal coefficient theorem, the vanishing theorem yields that $H_2 (\mathcal{M}_\mathbf{r} ; \Z)$ has no torsion. By the calculation of Poincar\'{e} polynomials in \cite[Corollary 4.3]{HausmannKnutson2} and \cite[Theorem 2.2.4]{Klyachko}, in the equilateral case, the rank of the second cohomology group of the polygon space $\mathcal{M}_\mathbf{r}$ is equal to the number of edges, $n+3$. From the Poincar\'{e} duality and the Hurewicz theorem, \eqref{equ_pi2x1l1} then follows. 
\end{proof}

To analyze the factors of~\eqref{equ_pi2x1l1}, we produce gradient disks as in Section~\ref{subsectionthetoricvarcater} to figure out explicit generators for $\pi_2 \left( \mathcal{M}_\mathbf{r}, L \right)$. By using the caterpillar bending system $\Phi_\mathbf{r}$ in~\eqref{equ_caterpillarbendingshifted} and the defining equations for the caterpillar bending polytope in~\eqref{equ_definingequellijpoly}, we define
$$
\widetilde{\ell}_{i,j} \colon \mcal{M}_\mathbf{r} \to \R,  \quad \widetilde{\ell}_{i,j} \coloneqq \ell_{i,j} \circ \Phi_\mathbf{r}.
$$

Although each component $\widetilde{\ell}_{i,j}$ is \emph{not} smooth on the whole space $\mcal{M}_\mathbf{r}$ in general, it is smooth on a dense open subset of $\mcal{M}_\mathbf{r}$. On its smooth locus, moreover, the function $\widetilde{\ell}_{i,j}$ is a moment map for a Hamiltonian $\bS^1$-action, see \cite[Section 4]{KapovichMillson}. Let $f_{i,j}$ be the facet of $\Delta_\mathbf{r}$ contained in the hyperplane given by $\ell_{i,j} ({\mathbf{u}}) = 0$ in $\R^n$ by Proposition~\ref{proposition_monotone}. Let ${V}_{i,j}$ be the inverse image of the relative interior of $f_{i,j}$ under the system $\Phi_{\mathbf{r}}$. To confirm that the family of $\bS^1$-orbits obtained by flowing an $\bS^1$-orbit contained in $L^1$ generates a \emph{disk}, we need the following lemma.

\begin{lemma}\label{lemma_smoothstratavij}
The map $\widetilde{\ell}_{i,j}$ is smooth on each point of ${V}_{i,j}$. 
\end{lemma}

\begin{proof}
Comparing~\eqref{equation_gzpattern} with~\eqref{equation_polypattern}, for each facet $f_{i,j}$ of the caterpillar bending polytope $\Delta_\mathbf{r}$, there is a unique corresponding facet $\widehat{f}_{i,j}$ of the GZ polytope $\Delta_{\lambda}$, which is contained in $k_{i,j}(\textbf{u}) = 0$. On the inverse image $\widehat{{V}}_{i,j}$ of the relative interior of $\widehat{f}_{i,j}$ under the GZ system $\Phi_{\lambda}$, 
We define $\widetilde{k}_{i,j} \colon \mcal{O}_{\lambda} \to \R$ by $\widetilde{k}_{i,j} \coloneqq k_{i,j} \circ \Phi_\lambda$. 

It is a consequence of the implicit function theorem that $\widetilde{k}_{i,j}$ is smooth at each point of $\widehat{{V}}_{i,j}$, see \cite[Corollary 4.5]{ChoKimMONO} for details. Let $T$ be the quotient of the maximal torus $T_{\mathrm{U}(n+3)}$ of $\mathrm{U}(n+3)$ consisting of diagonal matrices by the diagonal $\bS^1$-action. The map $\widetilde{k}_{i,j}$ restricted to the level set $2 \mathbf{r}$ of the moment map of $T$ is still smooth. Since $\mathbf{r}$ is generic, the torus $T$ freely acts on this level set and the induced map $\widetilde{\ell}_{i,j}$ on each point of ${V}_{i,j}$ is also smooth.
\end{proof}

Consider the $\bS^1$-action corresponding to the lattice point $\mathbf{v}_{i,j} \in N$ defined in~\eqref{equ_normalvectorth} in the toric variety ${X}^0$ of the toric degeneration $\{ {X}^{t} \}$ in Corollary~\ref{theorem_toricdegepolygon}. Its moment map can be chosen as the composition $- \ell_{i,j} \circ \Phi^0$ where $\Phi^0 \colon {X}^0 \to \Delta_\mathbf{r}$ is the toric moment map. Let $L^1 \coloneqq L$. We denote by $L^0$ a toric fiber $\psi_{1,0}(L^1)$ where $\psi_{1,0}$ is defined in~\eqref{equ_commutativitiyofcompletelyinte}. Flow an $\bS^1$-orbit $\vartheta_{i,j}$ contained in the toric fiber $L^0$ along the gradient Hamiltonian flow generated by $- \ell_{i,j} \circ \Phi^0$. The maximal component occurs at $f_{i,j}$. Then the family of $\bS^1$-orbits forms a gradient holomorphic disk bounded by $L^0$ in ${X}^0$, which intersects the toric divisor over $f_{i,j}$ exactly once. 

We turn to the polygon space ${X}^1 \simeq \mcal{M}_\mathbf{r}$ and take the $\bS^1$-action generated by $- \widetilde{\ell}_{i,j}$ into account. By the commutativity of~\eqref{equ_commutativitiyofcompletelyinte}, $\psi_{1,0}^{-1} (\vartheta_{i,j})$ becomes an $\bS^1$-orbit contained in $L^1 \coloneqq \psi_{1,0}^{-1} (L^0)$. Let us retake a particular almost complex structure on ${X}^1$ in order to produce a gradient disk that is compatible with the constructed gradient disk bounded by $L^0$. Let ${U}^0$ be the smooth loci of the singular toric variety ${X}^0$, which contains the inverse image of the relative interior of each facet $f_{i,j}$. First, we take a  complex structure by pulling the complex structure on ${U}^0$ back to $\psi_{1,0}^{-1}({U}^0)$. Next, we extend the pull-backed complex structure into an almost complex structure on the whole space ${X}^1$. 

The maximal component of the function $- \widetilde{\ell}_{i,j}$ occurs at $f_{i,j}$. By Lemma~\ref{lemma_smoothstratavij}, the $\bS^1$-action generated by $- \widetilde{\ell}_{i,j}$ makes sense on ${V}_{i,j}$, which is a fixed component. By taking an interpolation region sufficiently small, we may assume that a gradient disk associated to $- \widetilde{\ell}_{i,j}$ is fully contained in an open set (say ${V}^1$) on which the interpolated almost complex structure agrees with the pull-backed complex structure. Again by the commutativity of the diagram~\eqref{equ_commutativitiyofcompletelyinte} and our choice of almost complex structure, the image of the gradient disk in ${X}^1$ under $\psi_{1,0}$ agrees with the one generated by $\mathbf{v}_{i,j}$ in ${X}^0$. Thus, each gradient disk is contained in ${V}^1$.

We call the class generated by a gradient disk \emph{basic}. We denote by $\alpha_{i,j}$ the class represented by the gradient disk generated by the $\bS^1$-action corresponding to $- \widetilde{\ell}_{i,j}$. 

The following lemma explores a relation on the basic classes.

\begin{lemma}\label{lemma_relationonblacksq}
Let
\begin{equation}\label{equ_blacksquarbetajjcbs2}
\blacksquare_i \coloneqq {\alpha_{i,0}} + {\alpha_{i,1}} -  {\alpha_{i+1,0}} - {\alpha_{i,2}} \,\, \mbox{ for $i = 2, 3, \cdots, n-1$}.
\end{equation}
Then the class $\blacksquare_i$ is trivial.
\end{lemma}

\begin{proof}
Let $g_{i}$ be the face of codimension three contained in the hyperplane given by $u_i - |\mathbf{r}|_{i+1} = 0$. Namely, $g_{i}$ is the intersection of $\Delta_\mathbf{r}$ and the four hyperplanes given by $\ell_{i,0}(\mathbf{u}) = \ell_{i+1,0}(\mathbf{u}) = \ell_{i,1}(\mathbf{u}) = \ell_{i,2}(\mathbf{u})$. 
Let $\mathbf{u}_i$ be the center of the face $g_i$. By Theorem~\ref{theorem_Bouloc}, the fiber over $\mathbf{u}_i$ is a Lagrangian submanifold diffeomorphic to $\mathrm{SO}(3) \times T^{n-3}$. 

Let $[\mathbf{u}_{i}, \mathbf{u}_0]$ be the line segment from the center $\mathbf{u}_{i}$ of $g_i$ to the center $\mathbf{u}_0$ of $\Delta_\mathbf{r}$. We choose a point $\mathbf{u}$ in the interior of $[\mathbf{u}_{i}, \mathbf{u}_0]$. The choice leads to a Lagrangian isotopy from $\Phi_\mathbf{r}^{-1}(\mathbf{u}_0)$ to $\Phi_\mathbf{r}^{-1}(\mathbf{u})$ such that
\begin{equation}\label{equation_pi2mr}
\pi_2(\mcal{M}_\mathbf{r}, \Phi_\mathbf{r}^{-1}(\mathbf{u}_0)) \simeq \pi_2(\mcal{M}_\mathbf{r}, \Phi_\mathbf{r}^{-1}(\mathrm{Int}(\Delta_\mathbf{r})) \simeq \pi_2(\mcal{M}_\mathbf{r}, \Phi_\mathbf{r}^{-1}(\mathbf{u})).
\end{equation}

Note that each basic class can be represented by a gradient holomorphic disk. As the chosen point $\mathbf{u}$ gets closer to $\mathbf{u}_i$, the areas of the gradient disks in ${\alpha_{i,0}}, {\alpha_{i,1}},{\alpha_{i+1,0}}, {\alpha_{i,2}}$ converge to zero. By taking $\mathbf{u}$ sufficiently close to $\mathbf{u}_i$, we may choose a Darboux neighborhood of the fiber over $\mathbf{u}_i$ containing the four disks by the monotonicity lemma. Thus it can be embedded into $T^* (\mathrm{SO}(3) \times T^{n-3})$. The class $\blacksquare_i$ in~\eqref{equ_blacksquarbetajjcbs2} represents a spherical class because $\partial \blacksquare_i = 0$. Since $\pi_2(\mathrm{SO(3)} \times T^{n-3}) = 0$, this class $\blacksquare_i$ is trivial in a disk bundle of $T^* (\mathrm{SO}(3) \times T^{n-3})$. So is in $\mcal{M}_\mathbf{r}$.
\end{proof}

As the $1$-cones of the fan associated to the toric variety ${X}^0$ do not generate the lattice $N$, it is convenient to work over the field $\Q$ of rational numbers in order to nullify the torsion of its fundamental group. Let us consider $\pi_2 \left({X}^{t}, L^{t}\right) \otimes_\Z \Q \simeq H_2({X}^{t}, L^{t}; \Q)$ instead of $\pi_2 \left({X}^{t}, L^{t}\right)$ where $L^t \coloneqq \psi_{1,t}(L^1)$ and $\psi_{1,t}$ is given in Corollary~\ref{theorem_toricdegepolygon}. 

Let $\mcal{G}$ be the subspace of $H_2({X}^1, L^1; \Q)$ generated by the set $\{ \alpha_{i,j} \}$ of basic classes. Lemma~\ref{lemma_relationonblacksq} implies that  
\begin{equation}\label{eqjdimqmcalg}
\dim_\Q \mcal{G} \leq m - (n-2) = (3n - 1) - (n-2) = 2n+1.
\end{equation}
By Proposition~\ref{proposition_pi2xl}, $H_2({X}^0, L^0; \Q) \simeq \Q^{2n+1}$ is generated by the gradient disks. There is a natural one-to-one correspondence between the basic disks in ${X}^1$ and the basic disks in ${X}^0$ via the map $\psi_{1,0}$ as observed above. It leads to the map from $\alpha_{i,j}$ to $\beta_{i,j}$. By extending the correspondence between basic classes linearly, we obtain a homomorphism 
\begin{equation}\label{equpsi10*}
\psi_{1,0,*} \colon \mcal{G} \to H_2({X}^0, L^0; \Q) \quad \alpha_{i,j} \mapsto \beta_{i,j}.
\end{equation}
since $\blacksquare_i$ in~\eqref{equ_blacksquarbetajjcbs2} maps into $\blacksquare_i$ in~\eqref{equ_blacksquarbetajjcbs}. Moreover, because of~\eqref{eqjdimqmcalg} and Proposition~\ref{proposition_pi2xl}, $\psi_{1,0,*}$  is isomorphic. In particular, $\dim_\Q \mcal{G} = 2n + 1$. 

Let $g_{i}$ be the face of codimension two contained in the hyperplane given by $u_i - | \mathbf{r} |_{i+1} = 0$ for $i =1$ and $n$. Let $\mathbf{u}_i$ be the center of the face $g_i$. By Theorem~\ref{theorem_Bouloc},  the fiber over $\mathbf{u}_i$ is a Lagrangian submanifold diffeomorphic to $\bS^2 \times T^{n-2}$. Note that the fiber of $\Phi^0$ at $\mathbf{u}_i$ is $T^{n-2}$ and hence the spherical factors collapse through the degeneration. The classes represented by  spherical factors over $\mathbf{u}_1$ and $\mathbf{u}_n$ are denoted by $\alpha_1 \coloneqq [\bS^2 \times \{pt\} ]$ and $\alpha_n \coloneqq [\bS^2 \times \{ pt\}]$, respectively. 

Another way of expressing those spherical classes $\alpha_1$ and $\alpha_n$ is in order. Recall that the diffeomorphism type of $\mathcal{M}_\mathbf{r}$ depends on a choice of $\mathbf{r}$. The parameter space for $\mathbf{r}$ has a chamber structure given by the equations of the form~\eqref{equ_generic}. Thus, the diffeomorphism type of $\mathcal{M}_\mathbf{r}$ does not change as long as $\mathbf{r}$ stays in a chamber. In other words, the diffeomorphism type persists under a small perturbation of $\mathbf{r}$. Since $\mathbf{r}$ is generic, the equilateral tuple is the interior of the chamber containing the tuple. We now make a different choice $\mathbf{r}$ within the chamber. For instance, making $r_1$ and $r_{n+3}$ smaller introduces two additional facets, say $g_{\mathrm{new},1}$ and $g_{\mathrm{new},n}$. There are two basic classes $\alpha_{\mathrm{new},1}$ and $\alpha_{\mathrm{new}, n}$ generated by a gradient disk emanating from $g_{\mathrm{new},1}$ and $g_{\mathrm{new},n}$, respectively. 

Let $\alpha_1^\prime \coloneqq \alpha_{1,2} + \alpha_{2,0} - 2 \alpha_{\mathrm{new}, 1}$ and $\alpha_n^\prime \coloneqq \alpha_{n,0} + \alpha_{n,1} - 2 \alpha_{\mathrm{new}, n}.$ We claim that the spherical classes $\alpha_1$ and $\alpha_n$ correspond to $\alpha^\prime_1$ and $\alpha^\prime_n$ under the diffeomorphism, respectively. We only deal with the case where $n = 1$ because the other case is similar. Note that the area of $\alpha^\prime_1$ can be computed from the areas of classes represented by a gradient disk. Also, the area of gradient disk generated by a Hamiltonian $\mathbb{S}^1$-action can be computed from the difference of its moment map. Note that the area $\omega(\alpha^\prime_1) > 0$ is positive and hence $\alpha^\prime_1$ is non-trivial. Since $\partial \alpha^\prime_1 = 0$, the class $\alpha^\prime_1$ is spherical, that is, $\alpha^\prime_1 \in \pi_2 (X^1) \otimes_\Z \Q \simeq H_2(X^1; \Q)$. By isotoping the sum of disk classes to the new facet $g_{\mathrm{new},1}$, we may take a representative of the non-trivial class $\alpha^\prime_1$ supported on the inverse image of $g_{\mathrm{new},1}$. Since the area of the class $\alpha_1^\prime$ is equal to the area of the sphere class corresponding to $\alpha_1$ over $g_{\mathrm{new},1}$, by changing the sign if necessary, we conclude that $\alpha^\prime_1$ corresponds to the spherical class $\alpha_1$.

Since $r_1$ and $r_{n+3}$ can be adjusted independently, the areas of $\alpha^\prime_1$ and $\alpha^\prime_n$ are independent, which in turn implies that two classes are independent. The classes $\alpha_1$ and $\alpha_n$ are vanishing through the induced homomorphism $\psi_{1,0,*}$. Thus, the classes $[\alpha_1]$ and $[\alpha_n]$ form a basis for the quotient space $H_2({X}^1, L^1; \Q) / \mcal{G} \simeq \Q^2$. In sum, we have the induced homomorphism 
\begin{equation}\label{equpsi101}
\psi_{1,0,*} \colon H_2({X}^1, L^1; \Q) \simeq \mcal{G} \oplus \Q^2  \to H_2({X}^0, L^0; \Q)
\end{equation}
maps the spherical classes $\alpha_1$ and $\alpha_n$ to zero. 

We now give the proof of Proposition~\ref{lemma_monotonecbs}.

\begin{proof}[Proof of Proposition~\ref{lemma_monotonecbs}]
For each basic class $\beta_{i,j}$, we compute
\begin{enumerate}
\item (Maslov index) $\mu (\beta_{i,j}) = 2$ because of Proposition~\ref{proposition_maslovformula},
\item (Symplectic area) $\omega_\mathbf{r} (\beta_{i,j}) = r_1$. The area of gradient disk generated by a Hamiltonian $\mathbb{S}^1$-action can be computed from the difference of its moment map due to Archimedes. Recall that $\mathbf{u}_0$ is the center in~\eqref{equ_centerofdeltamathbfr} and it follows from the expression~\eqref{equ_elli0i1i2} that
$$
\omega_\mathbf{r} (\beta_{i,j}) = \ell_{i,j} (\mathbf{u}_0) = r_1.
$$
\end{enumerate}
Thus, for each $\alpha \in \mcal{G}$, we have
\begin{equation}\label{equ_monotonequeg}
\mu (\alpha) = \frac{2}{r_1}  \omega_\mathbf{r} (\alpha).
\end{equation}
Since the spherical classes are represented by an isotropic submanifold and the ambient symplectic manifold is positively monotone, the area of each class is zero and hence its Chern number is zero. For each $\alpha \in H_2({X}^1, L^1; \Q) / \mcal{G} \simeq \Q^2$, we have
$$
2 c_1 (T \mcal{M}_\mathbf{r})( \alpha ) = \frac{2}{ r_1} \omega_\mathbf{r} (\alpha).
$$ 
Hence,~\eqref{equ_monotonequeg} holds for all $\alpha \in H_2({X}^1, L^1; \Q)$ as desired.
\end{proof}

\subsection{Effective disk classes}

Throughout this subsection, we take 
$$
r_1 = r_2 = \dots = r_{n+3} = 1
$$
and let $L^1$ be the monotone Lagrangian torus in Proposition~\ref{lemma_monotonecbs}. Suppose that $\alpha \in H_2({X}^1, L^1; \Q)$ is realized as a holomorphic disk of Maslov index two. Consider the image $\beta \coloneqq \psi_{1,0,*}( \alpha) \in H_2({X}^0, L^0; \Q)$ under the map~\eqref{equpsi101}. In this subsection, we shall characterize such homology classes.

By Proposition~\ref{proposition_pi2xl}, the class $\beta$ can be expressed as a linear combination
$$
\sum_{i=2}^n a_{i,0} \beta_{i,0} + \sum_{i=2}^{n+1} a_{i,1} \beta_{i,1} + \sum_{i=0}^{n-1} a_{i,2} \beta_{i,2}.
$$
However, the expression is \emph{not} unique due to the relation~\eqref{equ_blacksquarbetajjcbs}.
The following lemma claims that there exists a representative of $\beta$ such that all $a_{i,j}$ are \emph{non-negative}. 

\begin{proposition}\label{proposition_nonnegativeness}
Suppose that $\alpha \in H_2({X}^1, L^1; \Q)$ is a class represented by a holomorphic disk of Maslov index two. Then the class $\beta = \psi_{1,0,*}( \alpha)$ can be expressed as 
 $$
\sum_{i=2}^n a_{i,0} \beta_{i,0} + \sum_{i=2}^{n+1} a_{i,1} \beta_{i,1} + \sum_{i=0}^{n-1} a_{i,2} \beta_{i,2}
$$
such that each $a_{i,j}$ is a non-negative rational number. 
\end{proposition}

Let us extend the Maslov homomorphism $\pi_2({X}^1, L^1) \to \Z$ to 
\begin{equation}\label{equ_pi21maslovhomo}
\mu \colon H_2({X}^1, L^1; \Q) \simeq \pi_2({X}^1, L^1) \otimes \Q \to \Q.
\end{equation}
Recall that ${X}^0$ can be resolved by a toric orbifold $\widetilde{{X}}^0$ \emph{without} adding any $1$-cones and there is a toric morphism $\Pi \colon \widetilde{{X}}^0 \to {X}^0$, see \cite[Proposition 11.1.7]{CoxLittleSchenckToric}. The inverse image $\widetilde{L}^0 \coloneqq \Pi^{-1}(L^0)$ is a Lagrangian toric fiber. Note that $H_2(\widetilde{{X}}^0, \widetilde{L}^0; \Q) \simeq \Q^{3n-1}$ since the toric orbifold has $(3n-1)$ many $1$-cones. Let $\widetilde{\beta}_{i,j}$ be the basic class in $H_2(\widetilde{{X}}^0, \widetilde{L}^0; \Q)$ corresponding to the $1$-cone $\mathbf{v}_{i,j}$. Note that the class $\widetilde{\beta}_{i,j}$ maps into the class ${\beta}_{i,j}$ via the induced homomorphism of $\Pi$. 

Take the quotient of $H_2(\widetilde{{X}}^0, \widetilde{L}^0; \Q)$ by the subspace generated by 
\begin{equation}\label{equ_tildeblacksquarbetajj}
\blacksquare_i \coloneqq \widetilde{\beta}_{i,0} + \widetilde{\beta}_{i,1} -  \widetilde{\beta}_{i+1,0} - \widetilde{\beta}_{i,2} \mbox{ for $i =2, \cdots, n-1$}.
\end{equation} 
We denote the resulting quotient space by $\overline{H}_2(\widetilde{{X}}^0, \widetilde{L}^0; \Q)$. 
We have the induced map 
\begin{equation}\label{equ_piinq}
\Pi_* \colon \overline{H}_2(\widetilde{{X}}^0, \widetilde{L}^0; \Q) \to  H_2({{X}}^0, {L}^0; \Q),
\end{equation}
an isomorphism of vector spaces. The Maslov index on $\overline{H}_2(\widetilde{{X}}^0, \widetilde{L}^0; \Q)$ still makes sense because the Maslov index of $\blacksquare_i$ is zero. Moreover, the induced map preserves the Maslov index since the Maslov indices of $\beta_{i,j}$ and $\widetilde{\beta}_{i,j}$ are same. 

Note that the subspace $\mcal{G}$ is isomorphic to the quotient space of $H_2({{X}}^1, {L}^1; \Q)$ by the two dimensional vector spaces generated by $\alpha_1$ and  $\alpha_n$. Since the spherical classes are represented by an isotropic submanifold, the area of each class is zero and hence its Chern number is zero because of the monotonicity of the ambient symplectic manifold ${X}^1$. Therefore, the Maslov index homomorphism in~\eqref{equ_piinq} factors through the quotient space $\simeq \mcal{G}$ as follows.
\begin{equation}\label{equ_maslovfactorthruocomm}
\xymatrix{
H_2({X}^1, L^1; \Q) \ar[rd] \ar[rr]^{\quad \,\,\, \mu} & & \Q \\
& \mathcal{G} \ar[ru]_{\mu} &
}.
\end{equation}
By abuse of notation, the induced Maslov homomorphism on $\mcal{G}$ is still denoted by $\mu$. Also, we denote by $\mu$ the Maslov homomorphism over $\Q$ on $\overline{H}_2(\widetilde{{X}}^0, \widetilde{L}^0; \Q)$.

\begin{lemma}\label{lemma_maslovpreserving}
The composition $\Pi_*^{-1} \circ {\psi}_{1,0,*}$ preserves the Maslov index. Namely, 
\begin{equation}\label{lemma_maslovpresereq}
{\mu} (\alpha) = {\mu} ((\Pi_*^{-1}  \circ {\psi}_{1,0,*}) (\alpha))
\end{equation}
for all $\alpha \in H_2({X}^1, L^1; \Q) \simeq \pi_2({X}^1, L^1) \otimes \Q$.
\end{lemma}

\begin{proof}
By~\eqref{equ_maslovfactorthruocomm}, it suffices to prove~\eqref{lemma_maslovpresereq} for all $\alpha \in \mcal{G}$. The composition of two homomorphisms~\eqref{equpsi10*} and the inverse of~\eqref{equ_piinq} is 
\begin{equation}\label{equ_compositionpipsi}
\Pi_*^{-1} \circ {\psi}_{1,0,*} \colon \mcal{G}  \simeq \Q^{2n+1} \to \overline{H}_2(\widetilde{{X}}^0, \widetilde{L}^0; \Q)  \simeq \Q^{2n+1}
\end{equation}
is an isomorphism.
It is enough to show that the Maslov index of each basic class is preserved. The basic classes of $({X}^1, L^1)$ are in one-to-one correspondence with those of $(\widetilde{{X}}^0, \widetilde{L}^0)$. Since every basic class is of Maslov index two,~\eqref{lemma_maslovpresereq} follows.
\end{proof}

\begin{remark}
The map merely sending basic classes in $\mcal{G} \simeq {H}_2({{X}}^0, {L}^0; \Q)$ to basic classes in ${H}_2(\widetilde{{X}}^0, \widetilde{L}^0; \Q)$  does not induce a well-defined homomorphism. For instance, the class $\widetilde{\beta}$ represented by a strict transformation of $\beta \coloneqq {\beta}_{i,0} + {\beta}_{i,1}  =  {\beta}_{i+1,0} + {\beta}_{i,2}$ in ${H}_2({{X}}^0, {L}^0; \Q)$ is either $ \widetilde{\beta}_{i,0} + \widetilde{\beta}_{i,1}$ or $\widetilde{\beta}_{i+1,0} + \widetilde{\beta}_{i,2}$ depending on a choice of simplicialization, while $\widetilde{\beta}_{i,0} + \widetilde{\beta}_{i,1} \neq  \widetilde{\beta}_{i+1,0} + \widetilde{\beta}_{i,2}$. The map in~\eqref{equ_compositionpipsi} is a homomorphism if replacing ${H}_2(\widetilde{{X}}^0, \widetilde{L}^0; \Q)$ with $\overline{H}_2(\widetilde{{X}}^0, \widetilde{L}^0; \Q)$.
\end{remark}

\begin{proof}[Proof of Proposition~\ref{proposition_nonnegativeness}]
For each nonzero $t$, $\psi_{1,t}(L^1) \eqqcolon L^t$ is monotone since $\psi_{1,t}$ is a symplectomorphism. A generic path of two generic almost complex structures between $J^1$ and $J^t$ leads to an oriented cobordism between the moduli space of $J^1$-holomorphic disks in $\alpha^1 \coloneqq \alpha$ and  the moduli space of $J^t$-holomorphic disks in $\alpha^t \coloneqq \psi_{1,t,*}(\alpha^1)$. In particular, we have a family $\{\varphi^t \}$ of $J^t$-holomorphic disks bounded by $L^t$ of Maslov index two and area one. 

The limit of the holomorphic disks is a holomorphic map having the area one bounded by the toric fiber $L^0$ at the center in the toric variety ${X}^0$. Since ${X}^0$ is \emph{neither} a symplectic manifold nor a symplectic orbifold, we cannot directly apply Gromov's compactness theorem to a sequence of pseudo-holomorphic curves. Following a trick of Nishinou--Nohara--Ueda in \cite{NishinouNoharaUeda}, regard ${X}^t$ as a family of subvarieties in a product of projective spaces and extend a Lagrangian submanifold into a totally real submanifold. Applying a family version of the compactness theorem in \cite{Ye} to the sequence, we then obtain a holomorphic curve bounded by $L^0$ at the limit. We denote by $\widetilde{\beta}$ the class of the strict transformation bounded by $\widetilde{L}^0$.

By Lemma~\ref{lemma_maslovpreserving}, the Maslov index is preserved under the above procedure. By Cho--Poddar \cite[Proposition 9.4]{ChoPoddar}, every class in $H_2(\widetilde{{X}}^0, \widetilde{L}^0; \Q)$ represented by a holomorphic (orbi-)disk can be expressed as a non-negative rational combination of basic classes. By pushing forward of $\widetilde{\beta}$ via $\Pi_*$, we obtain a desired expression.
\end{proof}

\begin{lemma}\label{lemma_nontrivialofboundary}
Suppose that a homology class $\alpha \in H_2({X}^1, L^1; \Q)$ of Maslov index two is represented by a holomorphic disk. Then the boundary of the class $\alpha$ is non-trivial, that is, $\partial \alpha \neq 0$. 
\end{lemma}

\begin{proof}
Suppose that there exists a holomorphic disk of Maslov index two in $\alpha$ such that $\partial \alpha = 0$. Then we have the strict transformation of the limit of disks, which also has Maslov index two. Again by the classification result of \cite[Proposition 9.4]{ChoPoddar}, the boundary cannot be non-zero.
\end{proof}

Suppose that $\alpha \in H_2({X}^1, L^1; \Q) \simeq \pi_2({X}^1, L^1) \otimes \Q$ is a class represented by a  $J^1$-holomorphic disk of Maslov index two. By Lemma~\ref{lemma_secondhomotopygpz2n3}, we have
\begin{equation}\label{equ_alphasum}
\alpha = \sum_{i=2}^n a_{i,0} \alpha_{i,0} + \sum_{i=2}^{n+1} a_{i,1} \alpha_{i,1} + \sum_{i=0}^{n-1} a_{i,2} \alpha_{i,2} + a_1  \alpha_1 + a_2 \alpha_n.
\end{equation}
One can show that the Maslov index of $\alpha_{i,j}$ is two by applying the Maslov index formula in Proposition~\ref{proposition_maslovformula} to the gradient disk $\varphi_{i,j}$. The monotonicity of $L^1$ yields that the area of each $\alpha_{i,j}$ is one. Hence, 
\begin{equation}\label{equ_identitiyareaone}
1 = \omega(\alpha) = \sum_{i=2}^n a_{i,0} + \sum_{i=1}^{n+1} a_{i,1} + \sum_{i=0}^n a_{i,2} .
\end{equation}
Set $\psi_{1,0,*}(\alpha) \eqqcolon \beta$. Since $\psi_{1,0,*}(\alpha_{i,j}) = \beta_{i,j}$, the boundary of a disk must be a loop and hence
\begin{equation}\label{equ_aaaN}
\sum_{i=2}^n a_{i,0} \cdot \partial \beta_{i,0} + \sum_{i=1}^{n+1} a_{i,1} \cdot \partial \beta_{i,1} + \sum_{i=0}^n a_{i,2} \cdot \partial \beta_{i,2} \in N.
\end{equation}
By Proposition~\ref{proposition_nonnegativeness}, we may assume that every coefficient $a_{i,j}$ is non-negative. Also, Lemma~\ref{lemma_nontrivialofboundary} yields that the expression~\eqref{equ_aaaN} is \emph{non-zero}.

Recall that $\{v_1, v_2, \cdots, v_n\}$ forms a basis for $N$ in~\eqref{equ_abasisforzn}. The basis gives rise the identification $N_\Q \coloneqq N \otimes \Q \simeq \Q^n$ given by
$$
\mathbf{v} = \sum_{j=1}^n {a_j} v_j \in N_\Q \mapsto (a_1, a_2, \cdots, a_n)
$$
For each $k = 1, 2, \cdots, n$, we define the map 
$$
\delta_k \colon N_\Q \simeq \Q^n \to \Q \quad \mbox{ $\delta_k (\sum_{j=1}^n {a_j} v_j) = a_k$},
$$ 
the projection onto the $k$-th component.

We now explain how to compute the map $\delta_k$ using the ladder diagram. As each lattice point $N$ generates a torus action on $X^0$, the fundamental group of the torus fiber $L^0$ can be identified with $N$. Each basic disk $\beta_{i,j}$ corresponds to the arrow $\overrightarrow{e_{i,j}}$ in $\Gamma^\vee_\mathbf{r}$. The boundary of $\beta_{i,j}$ specifies an element of $\pi_1(L^0)$ and lies in $N$. Thus, $\delta_k(\partial \beta_{i,j})$ makes sense. The output $\delta_k(\partial \beta_{i,j})$ can be read off from the ladder diagram as follows. The class $\partial \beta_{i,j}$ can contribute to the output only when the $\overrightarrow{e_{i,j}}$ intersects with the diagonal line passing through $(k - 1/2, 3/2)$ and $(k + 1/2, 1/2)$. More precisely,
$$
\delta_k(\partial \beta_{i,j}) = 
\begin{cases}
1 &\begin{split}\mbox{The head of $\overrightarrow{e_{i,j}}$ intersects at $(k-1/2,3/2)$ or} \\ \mbox{The tail of $\overrightarrow{e_{i,j}}$ intersects at $(k+1/2,1/2)$}\end{split} \\ \\
-1 &\begin{split}\mbox{The head of $\overrightarrow{e_{i,j}}$ intersects at $(k+1/2,1/2)$ or} \\ \mbox{The tail of $\overrightarrow{e_{i,j}}$ intersects at $(k-1/2,3/2)$.}\end{split}
\end{cases}
$$
As an example, let us consider Figure~\ref{Fig_deltakmap} in the case where $n = 5$ and $\beta = \frac{1}{2} \beta_{2,0} + \frac{1}{2} \beta_{2,1} + 3 \beta_{2,2}$. As 
$\partial \beta_{0,2}, \partial \beta_{2,1}, \partial \beta_{2,2}$ corresponds to $(-1, -1, 0, 0, 0), (1, -1, 0, 0, 0), (0, -1, 1, 0, 0)$ resp., we have 
$$
\delta_{1} (\partial \beta) = \delta_{4} (\partial \beta) = \delta_5 (\partial \beta) = 0, \delta_{2} (\partial \beta) = -4, \delta_{3} (\partial \beta) = 3.
$$
\begin{figure}[h]
	\scalebox{0.85}{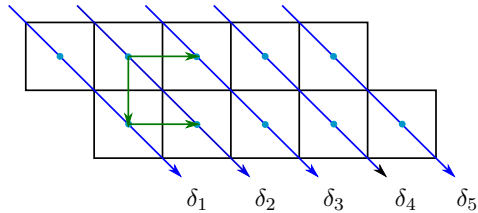}
	\caption{\label{Fig_deltakmap} The map $\delta_k$}	
\end{figure}

For a general $\beta \in \pi_2 (X^0, L^0) \otimes_\Z \Q$, the boundary of $\beta$ lies in $N_\Q$ and express $\partial \beta$ as a linear combination of $\beta_{i,j}$'s over $\Q$. The expression is \emph{not} unique, but the map $\delta_k (\partial \beta)$ is well-defined because $\partial  \blacksquare_i= 0$ in~\eqref{equ_tildeblacksquarbetajj}.

From now on, we assume that for each effective class $\alpha \in \pi_2({X}^1, L^1)$ of Maslov index two, $\beta \coloneqq \psi_{1,0,*}( \alpha)$ and there are non-negative rational numbers $a_{i,j}$'s satisfying
\begin{enumerate}
\item $\sum_{i=2}^n a_{i,0} + \sum_{i=1}^{n+1} a_{i,1} + \sum_{i=0}^n a_{i,2} = 1$,
\item $\sum_{i=2}^n a_{i,0} \cdot \partial \beta_{i,0} + \sum_{i=1}^{n+1} a_{i,1} \cdot \partial \beta_{i,1} + \sum_{i=0}^n a_{i,2} \cdot \partial \beta_{i,2} \in N - \{0\}$.
\end{enumerate}
Under this circumstance, we characterize the \emph{possible} list of the image of $\partial \beta$ under $(\delta_1, \delta_2, \dots, \delta_n)$. 

\begin{lemma}\label{lemma_possibledeltaouput}
For each $k$ with $1 \leq k \leq n$, the output of $\partial \beta$ under $\delta_k$ is one of $1, 0$, or $-1$.
\end{lemma}

\begin{proof}
For every double index $(i,j)$, each component of $\partial \beta_{i,j} \in N \simeq \Z^{n}$ is one of $0$, $1$, and $-1$, that is, $\delta_k (\partial \beta_{i,j}) \in  \{ 0, 1, -1 \}$. Thus, 
$$
\delta_k (\partial \beta) = \sum_{i=2}^n c_{i,0} a_{i,0} + \sum_{i=1}^{n+1} c_{i,1} a_{i,1} + \sum_{i=0}^n c_{i,2} a_{i,2} 
$$
for some $c_{i,j} \in \{ 0, 1, -1 \}$.
By Proposition~\ref{proposition_nonnegativeness} and~\eqref{equ_identitiyareaone}, $- 1 \leq \delta_k (\partial \beta) \leq 1$ and thus the lemma follows.
\end{proof}

The condition $(2)$ implies that at least one of $\delta_k (\partial \beta)$'s with $k = 1, 2, \dots, n$ is non-zero. Moreover, we have the following.

\begin{lemma}\label{lemma_threecannot}
For two indices $k, \ell \in \{1, 2, \dots, n\}$ with $| k - \ell | > 1$, at least one of $\delta_{k}(\partial \beta)$ and $\delta_{\ell}(\partial \beta)$ must be zero. 
\end{lemma}

\begin{proof}
To derive a contradiction, suppose that both $\delta_{\ell}(\partial \beta)$ and $\delta_{k}(\partial \beta)$ are nonzero. For $\delta_{k}$ to be non-zero, there exists a double index $(r,s)$ such that $a_{r,s} > 0$ and $\delta_{k} (\partial \beta_{r,s})\neq 0$. Then
\begin{equation}\label{lemmaell1ell2ell3}
\left( \sum_{i=2}^n a_{i,0} + \sum_{i=1}^{n+1} a_{i,1} + \sum_{i=0}^n a_{i,2} \right) - a_{r,s} < 1
\end{equation}
Since each $\partial \beta_{i,j}$ can contribute only \emph{two consecutive} diagonals, we have $\delta_{\ell} (\partial \beta_{r,s}) = 0$. Because of~\eqref{lemmaell1ell2ell3}, $-1 < \delta_{\ell}(\partial \beta) < 1$. Since $\delta_\ell(\partial \beta) \in \Z$, we have $\delta_{\ell}(\partial \beta) = 0$, which leads to a contradiction. 
\end{proof}

Here is one corollary of Lemma~\ref{lemma_possibledeltaouput} and~\ref{lemma_threecannot}.

\begin{corollary}\label{corollary_lemmathreecannot}
Either one of the following holds.
\begin{enumerate}
\item there exists $\ell \in \{1, 2, \dots, n\}$ such that $\delta_\ell(\partial \beta) = \pm 1$ and $\delta_k (\partial \beta) = 0$ for all $k (\neq \ell)$.
\item there exists $\ell \in \{1, 2, \dots, n-1\}$ such that $\delta_\ell(\partial \beta) = \pm 1$, $\delta_{\ell+1}(\partial \beta) = \pm 1$, and $\delta_k (\partial \beta) = 0$ for all $k (\neq \ell, \ell+1)$.
\end{enumerate}
\end{corollary}

We are ready to make a list of possible effective disk classes. 

\begin{proposition}\label{proposition_effectiveclasses}
The set of homotopy classes which can be represented by a holomorphic disk bounded by $L^1$ of Maslov index two in ${X}^1$ are contained in the set consisting of the following classes$\colon$
\begin{itemize}
\item $\alpha_{i,0}$ for $i= 2, 3, \cdots, n$,
\item $\alpha_{i,1}$ for $i= 2, 3, \cdots, n+1$,
\item $\alpha_{i,2}$ for $i= 0, 1, \cdots, n-1$,
\item $( \alpha_{i,0} + \alpha_{i,1} ) / 2 $  for $i= 2, \cdots, n-1$,
\item $( \alpha_{2,0} + \alpha_{1,2} ) / 2 + c_{1} \alpha_1$, $( \alpha_{n,0} + \alpha_{n,1} ) / 2 + c_{n} \alpha_n$ for $c_1$ and $c_n \in \Z$.
\end{itemize}
\end{proposition}

\begin{proof} The proof is simply a case-by-case analysis. Suppose that $\alpha$ is represented by a holomorphic disk bounded by $L^1$ of Maslov index two. Set $\beta = \psi_{1,0,*}(\alpha)$. By Corollary~\ref{corollary_lemmathreecannot}, it suffices to deal with the the following two cases.

\bigskip

\noindent \textbf{Case 1}. Assume that (1) in Corollary~\ref{corollary_lemmathreecannot} occurs. Then $\delta_{\ell} (\partial \beta)$ is either $1$ or $-1$.
\begin{itemize}
\item \textbf{Case 1-1}. Suppose that $\delta_{\ell} (\partial \beta) = 1$. 
\begin{itemize}
\item \textbf{Case 1-1-1}. Assume that $\ell = 1$. The only four classes $\alpha_{0,2}, \alpha_{1,2}, \alpha_{2,0}$, and $\alpha_{2,1}$ can contribute to $\delta_1$. Thus,
\begin{equation}\label{equ_delta111}
\delta_1 (\partial \beta) = - a_{2,0} + a_{2,1} + a_{0,2} - a_{1,2}. 
\end{equation}
It implies that 
$$
\begin{cases}
a_{2,1} + a_{0,2} = 1 \mbox{ and } \\
a_{i,j} = 0 \mbox{ for all $(i,j) \neq (2,1), (0,2)$}.
\end{cases}
$$
Furthermore, $a_{0,2} = 1$ and $a_{2,1} = 0$ since $\delta_2(\partial \beta) = - a_{2,1}$. In this case, we have $\beta_{0,2}$.
\item \textbf{Case 1-1-2}. Assume that $\ell = n$. In this case, by applying a similar argument as in Case 1-1-1, we have $\beta_{n+1,1}$. 
\item \textbf{Case 1-1-3}. Assume that $1 < \ell < n$. The only six classes $\alpha_{\ell,0}, \alpha_{\ell+1,0}, \alpha_{\ell,1}$, $\alpha_{\ell+1,1}$, $\alpha_{\ell-1,2}$, and $\alpha_{\ell,2}$ can contribute to $\delta_\ell$. Thus,  
\begin{equation}\label{equ_delta113}
\delta_\ell (\partial \beta) = - a_{\ell,0} - a_{\ell+1,0} - a_{\ell,1} + a_{\ell+1,1} + a_{\ell-1,2} - a_{\ell,2}. 
\end{equation}
It implies that 
$$
\begin{cases}
a_{\ell+1,1} + a_{\ell-1,2} = 1 \mbox{ and } \\
a_{i,j} = 0 \mbox{ for all $(i,j) \neq (\ell+1,1), (\ell-1,2)$}.
\end{cases}
$$
Since $\delta_{\ell+1}(\partial \beta) = - a_{\ell+1,1}$ and $\delta_{\ell-1}(\partial \beta) = - a_{\ell-1,2}$, either $\delta_{\ell+1}(\partial \beta) \neq 0 $ or $\delta_{
\ell-1}(\partial \beta) \neq 0$ holds. There are no corresponding classes in this case. 

\end{itemize}
\item \textbf{Case 1-2}. Suppose that $\delta_{\ell} (\partial \beta) = -1$.
\begin{itemize}
\item \textbf{Case 1-2-1}. Assume that $\ell = 1$. By~\eqref{equ_delta111}, we have 
$$
\begin{cases}
a_{2,0} + a_{1,2} = 1 \mbox{ and } \\
a_{i,j} = 0 \mbox{ for all $(i,j) \neq (2,0), (1,2)$.}
\end{cases} 
$$
Since $\delta_2(\partial \beta) = - a_{2,0} + a_{1,2}$, we have $a_{2,0} = a_{1,2} = 1/2$. Thus, we have $(\beta_{2,0} + \beta_{1,2})/2$. 
\item \textbf{Case 1-2-2}. Assume that $\ell = n$. In this case, by applying a similar argument as in Case 1-2-1, we have $(\beta_{n,0} + \beta_{n,1})/2$.
\item \textbf{Case 1-2-3}. Assume that $1 < \ell < n$. By~\eqref{equ_delta113}, 
$$
\begin{cases}
a_{\ell,0} + a_{\ell+1,0} + a_{\ell,1} + a_{\ell,2}= 1 \mbox{ and } \\
a_{i,j} = 0 \mbox{ for all $(i,j) \neq (\ell, 0), (\ell+1,0), (\ell, 1), (\ell, 2)$}.
\end{cases}
$$ 
Since $\delta_{\ell -1} (\partial \beta) = 0$ and $\delta_{\ell+1} (\partial \beta) = 0$, we have $a_{\ell, 0} = a_{\ell, 1}$ and $a_{\ell+1, 0} = a_{\ell, 2}$. As $\beta_{\ell, 0} + \beta_{\ell, 1} - \beta_{\ell+1, 0} - \beta_{\ell, 2} = 0$, we may assume that $2 (a_{\ell+1, 0} + a_{\ell, 2}) = 1$. Therefore, in this case, the corresponding class is
$( \beta_{\ell+1, 0} + \beta_{\ell, 2} ) / 2 = ( \beta_{\ell, 0} + \beta_{\ell, 1} ) / 2$.
\end{itemize}
\end{itemize}
\textbf{Case 2.} Assume that (2) in Corollary~\ref{corollary_lemmathreecannot} occurs.
\begin{itemize} 
\item \textbf{Case 2-1}. Suppose that $\delta_{\ell} (\partial \beta) = 1$ and $\delta_{\ell+1} (\partial \beta) = 1$. By~\eqref{equ_delta113}, 
$$
\begin{cases}
a_{\ell+1, 1} + a_{\ell-1, 2} = 1 \mbox{ and } \\
a_{\ell+2, 1} + a_{\ell, 2} = 1.
\end{cases}
$$ 
We cannot make both $a_{\ell+1, 1} + a_{\ell-1, 2}$ and $a_{\ell+2, 1} + a_{\ell, 2}$ one. 
\item \textbf{Case 2-2}. Suppose that $\delta_{\ell} (\partial \beta) = -1$ and $\delta_{\ell+1} (\partial \beta) = -1$. By~\eqref{equ_delta113},  
$$
\begin{cases}
a_{\ell, 0} + a_{\ell+1, 0} + a_{\ell, 1} + a_{\ell, 2}= 1 \mbox{ and } \\
a_{\ell+1, 0} + a_{\ell+2, 0} + a_{\ell+1, 1} + a_{\ell+1, 2}= 1.
\end{cases}
$$
Thus, $a_{\ell+1, 0} = 1$ and $a_{i,j} = 0$ for all $(i,j) \neq (\ell+1, 0)$. The corresponding class is $\beta_{\ell+1, 0}$.
\item \textbf{Case 2-3}. Suppose that $\delta_{\ell} (\partial \beta) = 1$ and $\delta_{\ell+1} (\partial \beta) = -1$. By~\eqref{equ_delta113},  
$$
\begin{cases}
a_{\ell+1, 1} + a_{\ell-1, 2} = 1 \mbox{ and } \\
a_{\ell+1, 0} + a_{\ell+2, 0} + a_{\ell+1, 1} + a_{\ell+1, 2}= 1.
\end{cases}
$$
Thus, $a_{\ell+1, 1} = 1$ and $a_{i,j} = 0$ for all $(i,j) \neq (\ell+1, 1)$. The corresponding class is $\beta_{\ell+1, 1}$.
\item \textbf{Case 2-4}. Suppose that $\delta_{\ell} (\partial \beta) = -1$ and $\delta_{\ell+1} (\partial \beta) = 1$. By~\eqref{equ_delta113},  
$$
\begin{cases}
a_{\ell,0} + a_{\ell+1,0} + a_{\ell,1} + a_{\ell,2}= 1 \mbox{ and } \\
a_{\ell+2,1} + a_{\ell,2} = 1.
\end{cases}
$$
Thus, $a_{\ell,2} = 1$ and $a_{i,j} = 0$ for all $(i,j) \neq (\ell,2)$. The corresponding class is $\beta_{\ell,2}$.
\end{itemize}

Each basic class $\beta_{i,j}$ intersects the inverse image of a facet and does not intersect the inverse image of the stratum $g_1$ and $g_n$. Thus, each basic class $\alpha_{i,j}$ intersects neither the vanishing $\alpha_1$ nor the vanishing $\alpha_n$. An effective homology class that degenerates into $\beta_{i,j}$ is of the form $\alpha_{i,j}$.

Disks in the non-basic classes can intersect with $\alpha_1$ or $\alpha_n$. For instance, a disk in the class $( \beta_{1,0} + \beta_{1,1} ) / 2$ interacts with (the inverse image of the relative interior of) $g_1$, while any disk in the class $( \beta_{1,0} + \beta_{1,1} ) / 2$ cannot intersect with  (the inverse image of the relative interior of) $g_n$. Thus, an effective homology class degenerating into  $( \beta_{1,0} + \beta_{1,1} ) / 2$ is of the form $( \alpha_{1,0} + \alpha_{1,1} ) / 2 + c_{1} \alpha_1$ for some $c_1 \in \Z$. Similarly, the other cases can be dealt with. It completes the proof.
\end{proof}

\section{Computation of open Gromov--Witten invariants}\label{Section_computationofopenGWinv}

In this section, we shall compute the counting invariant for each class in the list of Proposition~\ref{proposition_effectiveclasses} and hence it completes the proof of Theorem~\ref{theorem_main}. To compute the counting invariants, we calculate the counting invariants of lower dimensional local models and then calculate the counting invariants of higher dimensional examples via the counting invariants of lower dimensional local models.

\subsection{Counting invariants for basic classes}

We begin by computing counting invariants for basic disks. Let $L^1$ be a monotone Lagrangian torus in $\mathcal{M}_\mathbf{r} = {X}^1$ and let $\psi_{1,0}(L^1) = L^0$. Recall the notions of relative homotopy classes $\alpha_{i,j}$'s in Proposition~\ref{proposition_effectiveclasses} and $\beta_{i,j} =  \psi_{1,0,*}(\alpha_{i,j})$, the image under the map~\eqref{equpsi101}.

\begin{proposition}\label{proposition_countinginvarinatone}
The counting invariant for each basic class is one. That is, the counting invariant for each class in the following list is one.
\begin{itemize}
\item $\alpha_{i,0}$ for $i= 2, 3, \cdots, n$,
\item $\alpha_{i,1}$ for $i= 2, 3, \cdots, n+1$, and
\item $\alpha_{i,2}$ for $i= 0, 1, \cdots, n-1$.
\end{itemize}
\end{proposition}

\begin{proof}
Let $\Pi$ be the toric morphism from a simplicialization $\widetilde{{X}}^0$ of ${X}^0$ to the toric variety ${X}^0$ of $\Delta_\mathbf{r}$. We denote by $\widetilde{L}^0$ the Lagrangian toric fiber such that $\Pi(\widetilde{L}^0) = L^0$.

For each class $\alpha_{i,j}$ in the above list, recall that we have a gradient  holomorphic disk in the class $\beta_{i,j}$ by Proposition~\ref{proposition_pi2xl}, which emanates from the smooth loci of the toric anti-canonical divisor. Since the gradient disk intersects the toric anti-canonical divisor in ${X}^0$ once, its strict transformation also intersects the toric anti-canonical divisor of $\widetilde{{X}}^0$ once. By the classification result of Cho--Poddar \cite{ChoPoddar}, the class $\tilde{\beta}_{i,j}$ represented by its strict transformation and the toric complex structure are Fredholm regular. Moreover, the counting invariant of $\tilde{\beta}_{i,j}$ is one if taking the standard spin structure and orientation from the torus action on $\widetilde{{X}}^0$. We then have an orientation preserving  diffeomorphism
$$
\Pi_* \colon \mcal{M}_1(\widetilde{L}^0, \widetilde{\beta}_{i,j}) \to \mcal{M}_1({L}^0, {\beta}_{i,j}).
$$
Also, the holomorphic disks in the moduli space do not intersect the singular loci. By pulling back the toric complex structure to an open subset containing the image of disks and extending it to an almost complex structure on ${X}^1$, we then have an isomorphism
$$
\psi_{1,0,*} \colon \mcal{M}_1({L}^1, {\alpha}_{i,j}) \to \mcal{M}_1({L}^0, {\beta}_{i,j}).
$$
Since $L^1$ is monotone, the invariant remains one for a generic choice of almost complex structures.
We obtain the following commutative diagram of orientation preserving diffeomorphisms$\colon$
$$
\xymatrix{
 \mcal{M}_1({L}^1, {\alpha}_{i,j})  \ar[r]^{\psi_{1,0,*}} \ar[d]_{\mathrm{ev}_0} & \mcal{M}_1({L}^0, {\beta}_{i,j}) \ar[d]_{\mathrm{ev}_0} &  \mcal{M}_1(\widetilde{L}^0, \widetilde{\beta}_{i,j}) \ar[l]_{\Pi_*} \ar[d]_{\mathrm{ev}_0}  \\
L^1 \ar[r]^{\psi_{1,0}} & L^0 & \widetilde{L}^0 \ar[l]_{\Pi}}.
$$
In summary, the degree of $\mathrm{ev}_0 \colon  \mcal{M}_1({L}^1, {\alpha}_{i,j})  \to L^1$ agrees with that of $\mathrm{ev}_0 \colon \mcal{M}_1(\widetilde{L}^0, \widetilde{\beta}_{i,j})  \to \widetilde{L}^0$. Therefore, the counting invariant of every basic class is one.
\end{proof}

For each $i = 1, 2, \cdots, n$, we have 
$$
\partial \left( ( \alpha_{i,0} + \alpha_{i,1} ) / 2\right) = \partial \left( ( \alpha_{i,0} + \alpha_{i,1} ) / 2 + c_{i,1} \alpha_1 + c_{i,n} \alpha_n \right)
$$
for all $c_{i,1}, c_{i,2} \in \Z$. Note that all $c_{i,1}$ and $c_{i,2}$ are zero except finitely many by the Gromov's compactness theorem. Therefore, for a generic choice of compatible almost complex structure, the disk potential function is then of the form$\colon$
 \begin{equation}\label{equ_potentialexp1}
 W_\mathbf{r}(\mathbf{z}) = \left( z_1 + z_n \right) + \left( \frac{\kappa_1}{z_1} + \frac{\kappa_2}{z_2} + \dots + \frac{\kappa_n}{z_n} \right) + \sum_{j=1}^{n-1} \left( \frac{z_j}{z_{j+1}} + \frac{z_{j+1}}{z_j} + \frac{1}{z_{j}z_{j+1}}  \right)
 \end{equation}
by Proposition~\ref{proposition_effectiveclasses} and~\ref{proposition_countinginvarinatone}. The remaining task is to determine $\kappa_1, \kappa_2, \dots, \kappa_n$. The goal of this section is to show that
\begin{equation}\label{equ_computationk1k2}
\begin{cases}
\kappa_1 = \kappa_n = 2, \\
\kappa_2 = \kappa_3 = \dots = \kappa_{n-1} = 0.
\end{cases}
\end{equation}

The computation~\eqref{equ_computationk1k2} leads to Theorem~\ref{theorem_main}, that is, when $\mathbf{r}$ is equilateral and generic, the disk potential function of the monotone caterpillar bending torus fiber is
$$
W_\mathbf{r} (\textbf{z}) = \left(z_1 + \frac{2}{z_1} + z_n + \frac{2}{z_n} \right) + \sum_{j=1}^{n-1} \left( \frac{z_j}{z_{j+1}} + \frac{z_{j+1}}{z_j} + \frac{1}{z_{j}z_{j+1}}  \right).
$$

\subsection{Local models}

Now, we build two local models, which will be employed to compute the above counting invariants~\eqref{equ_computationk1k2} in Section~\ref{section_Globalmodel}. 

The first local model is contained in the caterpillar bending system $\Phi_\mathbf{r}$ on the polygon space $\mcal{M}_\mathbf{r}$ of complex dimension two (i.e., $n = 2$) for $\mathbf{r} = (r_1,r_2, \cdots ,r_{n+3 = 5})$ with $r_1 = r_2 = \cdots = r_{n+3} = 1$. Set $\mathbf{u}_{\bS^2} \coloneqq (2,2)$. We now focus on the local model around $\mathbf{u}_{\bS^2} $ at which a Lagrangian two sphere $\bS^2$ is located.  
\begin{enumerate}
\item Let $\Delta_{\bS^2}$ be the subset of the caterpillar bending polytope $\Delta_\mathbf{r} \subset \R^2$ given by
\begin{equation}\label{equ_deltas1}
\ell_{2,0}(\mathbf{u}) = 4 - u_{1} - u_{2} \geq 0, \, \ell_{1,2}(\mathbf{u}) = - u_{1} + u_{2} \geq 0, \mbox{ and } -1 + u_{1} > 0
\end{equation}
as depicted in the first figure of Figure~\ref{fig_local}.
\item Consider the toric degeneration $\{\Phi^t\}$ of $\Phi_\mathbf{r}$ restricted to $(\Phi^t)^{-1} \left( \Delta_{\bS^2} \right)$ in Corollary~\ref{theorem_toricdegepolygon}. 
Let ${U}_{\bS^2} \coloneqq \Phi_\mathbf{r}^{-1} \left( \Delta_{\bS^2} \right) \subset \mathcal{M}_\mathbf{r}$ and ${U}^t_{\bS^2} \coloneqq \psi_{1,t}({U}_{\bS^2})$, that is, ${U}^t_{\bS^2} = (\Phi^t)^{-1} \left( \Delta_{\bS^2} \right)$.
\end{enumerate}
\begin{figure}[h]
	\scalebox{0.85}{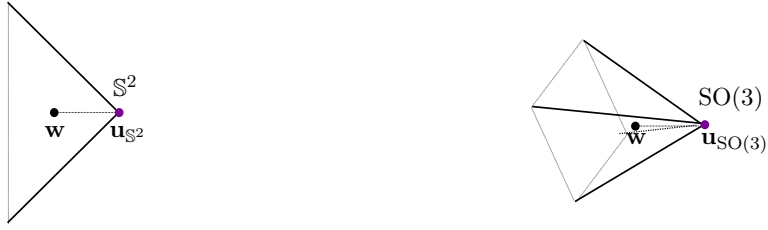}
	\caption{\label{fig_local} Two local models.}	
\end{figure}

For a point $\mathbf{u} = (u_{1}, u_{2}) \in \Delta_{\bS^2}$, let $L(\mathbf{u}) \coloneqq \Phi_\mathbf{r}^{-1}(\mathbf{u})$ and $L^t(\mathbf{u}) \coloneqq \psi_{1,t}(L(\mathbf{u}))$. By Theorem~\ref{theorem_Bouloc}, the fiber $L(\mathbf{u}_{\bS^2})$ is a Lagrangian two sphere and the subset ${U}_{\bS^2}$ is symplectically embedded into the cotangent bundle $(T^* \bS^2, \omega_{\mathrm{std}})$. 
Fix $\mathbf{w} = (w_1, 2) \in \Delta_{\bS^2}$ with $1.5 < w_1 < 2$. Consider the inverse image 
\begin{equation}\label{equ_cutboundary}
U_{w_1} \coloneqq \Phi_{\mathbf{r}}^{-1} \left( \{  \mathbf{u} \in \Delta_{\bS^2} \mid {u}_1 \geq - 2 +  2 w_1 \} \right) \subset {U}_{\bS^2} \mbox{ and } U^t_{w_1} \coloneqq \psi_{1,t}(U_{w_1}).
\end{equation}
Take the symplectic cut of $U^t_{w_1}$ along its boundary with respect to the $\bS^1$-action generated by $\Phi^t_{1}$. 
The resulting cut space is denoted by $\overline{{U}}_{w_1}^{t}$ and is a symplectic manifold equipped with the induced symplectic form from ${U}_{\bS^2}^t$. For a point $\mathbf{u} \in \{  \mathbf{u} \in \Delta_{\bS^2} \mid {u}_1 > - 2 +  2 w_1 \}$, we denote the Lagrangian torus in $\overline{{U}}^t_{w_1}$ corresponding to $L^t(\mathbf{u})$ by $\overline{L}^t(\mathbf{u})$.

\begin{lemma}\label{lemma_monotoneutw1}
For each $t$ with $t \in (0,1]$, the Lagrangian torus $\overline{L}^t(\mathbf{w})$ of $\overline{{U}}^t_{w_1}$ is monotone. In particular, the cut space $\overline{{U}}^t_{w_1}$ is monotone.
\end{lemma}

\begin{proof}
Note that ${U}_{\bS^2} = {U}^1_{\bS^2}$ and ${U}^{t}_{\bS^2}$ are symplectomorphic via the map $\psi_{1,t}$ restricted to ${{U}_{\bS^2}}$ in Corollary~\ref{theorem_toricdegepolygon}. Since the map $\psi_{1,t}$ is equivariant under the $\mathbb{S}^1$-action generated by the first component $\Phi^t_{1}$, it induces a symplectomorphism $\overline{\psi}_{1,t}$ between two cut spaces $\overline{U}_{w_1}$ and $\overline{U}^t_{w_1}$. Since $\overline{\psi}_{1,t}$ maps $\overline{L}(\mathbf{w})$ to $\overline{L}^t(\mathbf{w})$, it suffices to prove the statement for $t = 1$. 

To show that $\overline{L}(\mathbf{w})$ is monotone, besides the classes $\alpha_{2,0}$ and $\alpha_{1,2}$, one more gradient disk is necessary. Let $\overline{\alpha}$ be the class in a gradient disk generated by the $\mathbb{S}^1$-action associated with the moment map $(2w_1 - 2) - \Phi_1$. By Proposition~\ref{proposition_maslovformula}, the Maslov index of $\overline{\alpha}$ is two. The area of $\overline{\alpha}$ is $w_1 - (2w_1 - 2)  = 2 - w_1$. By the same argument in Proposition~\ref{lemma_monotonecbs}, we see that $\overline{L}(\mathbf{w})$ is monotone and hence the ambient space $\overline{{U}}_{w_1}$ is also monotone. 
\end{proof}

\begin{lemma}\label{lemma_sympleccp1cp1}
For each $t$ with $t \in (0,1]$,  the cut space $\overline{{U}}_{w_1}^{t}$ is symplectomorphic to $\CP^1 \times \CP^1$ equipped with a multiple of the product of Fubini--Study forms.
\end{lemma}

\begin{proof}
As in Lemma~\ref{lemma_monotoneutw1}, we may assume that $t = 1$. 
We briefly recall a well-known construction of $\CP^1 \times \CP^1$ from a disk bundle of the tangent bundle of $\mathbb{S}^2$. Identifying $T^{ }\mathbb{S}^2$ with $T^*\mathbb{S}^2$ via the round metric, we regard $T^{ }\mathbb{S}^2$ as a symplectic manifold. Then $\CP^1 \times \CP^1$ is obtained by the symplectic cut on $T^{ }\mathbb{S}^2$ along a level set of a moment map $\mu_{\mathbb{S}^1}$ of the Hamiltonian $\mathbb{S}^1$-action generated by the geodesic flow, see \cite{AF08, FukayaOhOhtaOnoS2S2} for more details.

We shall explain how our situation fits into the above well-known situation. Let us keep in mind that the choice of a point in $\Delta_{\mathbb{S}^2}$ determines the lengths of diagonals. Namely, the fiber $L(\mathbf{u})$ over a point $\mathbf{u} = (u_{1}, u_{2}) \in \Delta_{\mathbb{S}^2}$ is the set of points $[\mathbf{v}] = [(v_1, v_2, \cdots, v_5)]$ of $\mcal{M}_\mathbf{r}$ such that 
$$
|d_{1}| = | v_1 + v_2 | = 2 - {u}_{1} \mbox{ and } |d_{2}| = | v_1 + v_2 + v_3 | = 3 - {u}_{2}
$$ 
by Theorem~\ref{proposition_NNUgzbe}. 

First, we describe the projection map $\pi \colon U_{w_1} \to \mathbb{S}^2 (\subset \R^3)$ in~\eqref{equ_cutboundary} of the disk bundle. For an element $[\mathbf{v}] \in U_{w_1} \backslash L(\mathbf{u}_{\mathbb{S}^2})$, we denote by $\mathbf{u}_\mathbf{v}$ the position of $\mathbf{v}$, i.e. $\Phi_\mathbf{r}([\mathbf{v}]) = \mathbf{u}_{\mathbf{v}}$. We take the line segment connecting $\mathbf{u}_{\mathbb{S}^2}$ and $\mathbf{u}_\mathbf{v}$ as
$$
\mathbf{u}_s \coloneqq (1-s) \mathbf{u}_{\mathbb{S}^2} + s \mathbf{u}_\mathbf{v}.
$$
Note that the diagonal vector $d_{1}$ shrinks to the zero vector as $s \to 0$. Note that every representative $\mathbf{v}$ in $\R^3$ consists of three triangles
$$
(v_1, v_2, -d_1), (d_1, v_3, -d_2), \mbox{ and } (d_2, v_4, v_5).
$$
To specify a representative of an equivalence class $[\mathbf{v}]$ (modulo the orientation preserving isometries on $\R^3$), we take a particular slice of $[\mathbf{v}]$. By applying an orientation preserving isometry, we locate the vector $d_2$ in the positive $x$-axis and the second triangle $(d_1, v_3, - d_2)$ in $\{ (x,y,z) \in \R^3 \mid y \leq 0, z =0 \}$. Note that such a representative is uniquely determined because the triangle $(d_1, v_3, -d_2)$ is not degenerate provided $[\mathbf{v}] \in U_{w_1} \backslash L(\mathbf{u}_{\mathbb{S}^2})$. For simplicity, an equivalence class $[\mathbf{v}]$ will be identified with a representative $\mathbf{v}$ satisfying the above condition and the class $[\mathbf{v}]$ is denoted by $\mathbf{v}$.

To describe $\pi(\mathbf{v})$ under the projection map $\pi \colon U_{w_1} \to \mathbb{S}^2$ in~\eqref{equ_cutboundary} of the disk bundle, we take a family 
$$
\{ \mathbf{v}_s = (v_{s,1}, v_{s,2}, \cdots, v_{s,5}) \mid s \in (0,1]\}
$$ 
of points (representatives) in $U_{w_1}$ such that 
\begin{enumerate}
\item $\mathbf{v}_s$ is contained in the fiber over $\mathbf{u}_s$, that is, $\Phi_\mathbf{r}(\mathbf{v}_s) = \mathbf{u}_s$,
\item the given point $\mathbf{v}$ occurs at $s=1$, that is, $\mathbf{v}_{s=1} = \mathbf{v}$,
\item every $\mathbf{v}_s$ is contained in $\mathcal{M}_\mathbf{r}$, that is, $|v_{s,i}| = r_i = 1$ for all $s \in (0,1]$ and $i =1, 2, \cdots, 5$,
\item the angle from $(v_{s,1}, v_{s,2}, -d_{s,1})$ to $(d_{s,1}, v_{s,3}, -d_{s,2})$ is equal to that from $(v_1, v_2, -d_1)$ to $(d_1, v_3, -d_2)$.
\item the angle from $(d_{s,2}, v_{s,4}, v_{s,5})$ to $(d_{s,1}, v_{s,3}, -d_{s,2})$ is equal to that from $(d_2, v_4, v_5)$ to $(d_1, v_3, -d_2)$. 
\end{enumerate}
The conditions $(4)$ and $(5)$ uniquely determine the planes in which the triangles are respectively contained. Since the side lengths of three edges completely determine a triangle in a plane, $\mathbf{v}_s$ is uniquely determined from the above conditions. 

We now take a limit 
$$
\mathbf{v}_0 \coloneqq (v_{0,1}, v_{0,2}, \cdots, v_{0,5}) \coloneqq \lim_{s \to 0} (v_{s,1}, v_{s,2}, \cdots, v_{s,5}).
$$
The limit configuration consists of one line segment $v_{0,1} = - v_{0,2}$ and one triangle $(v_{0,3}, v_{0,4}, v_{0,5})$ since $d_{0,1}$ is the zero vector. We have located the vector $v_{0,3}$ at the positive $x$-axis by using (a part of) the $\mathrm{SO}(3)$-action. To take a \emph{unique} representative of $[\mathbf{v}_0]$, we additionally fix a slice of the rotational $\mathbb{S}^1$-action along the positive $x$-axis (because the triangle $(d_{0,1}, v_{0,3}, -d_{0,2})$ is degenerate). Let $\rho_\mathbf{v}$ be a rotation along the positive $x$-axis sending $(v_{0,3}, v_{0,4}, v_{0,5})$ to a congruent triangle contained in $\{ (x,y,z) \mid y \geq 0, z =0 \}$. Then the projection map $\pi$ is given by
$$
\pi (\mathbf{v}) = \rho_\mathbf{v} (v_{0,1}) \in \mathbb{S}^2 (\subset \R^3).
$$

Next, we define a bundle map between $U_{w_1}$ and a disk bundle of $\mathbb{S}^2$. Consider the vector
$$
{n}_{0,1} \coloneqq \lim_{\|d_{s,1}\| \to 0} \frac{{d}_{s,1}}{\|{d}_{s,1}\|}.
$$
Note that $v_{0,1}$ and $n_{0,1}$ are always perpendicular regardless of the initial position $\mathbf{u}_\mathbf{v}$ since $\lim_{s \to 0} |v_{s,1} - v_{s,2}| = 0$, while the angle between $n_{0,1}$ and $v_{0,3}$  varies depending on the starting position $\mathbf{u}_\mathbf{v}$. Specifically, the triangle $(d_{s,1}, v_{s,3}, -d_{s,2})$ has three edges of lengths $2- u_{s,1}, 1$, and $a (2- u_{s,1}) +1$ respectively where the slope $a$ is determined by the position of $\mathbf{u}_\mathbf{v}$ from $-1$ to $1$ and $\mathbf{u}_s = (u_{s,1}, u_{s,2})$. By the cosine law, it determines the angle between $n_{0,1}$ and $v_{0,3}$. In particular, if $a = 0$, then $n_{0,1}$ and $v_{0,3}$ are perpendicular. By rotating $\pi (\mathbf{v}) = \rho_\mathbf{v} (v_{0,1})$ along the axis $n_{0,1}$, we have an $\mathbb{S}^1$-family $\gamma_\mathbf{v}(\theta)$ of elements in $\mathbb{S}^2$ such that $\gamma_\mathbf{v}(0) = \pi (\mathbf{v})$. By differentiating the curve, we obtain a tangent vector to $\mathbb{S}^2$ at $\pi (\mathbf{v})$. The magnitude will be determined by the difference of the first component of $\mathbf{u}_{\mathbb{S}^2} - \mathbf{u}_\mathbf{v}$. Taking $r_1 = (2 - w_1)$, the bundle map is given by
$$
\phi \colon U_{w_1} \to D_{r_1}(T^{}\mathbb{S}^2), \quad \phi (\mathbf{v}) = \left(\pi (\mathbf{v}), \frac{|d_{1}|}{2} \cdot \frac{d}{d \theta} \Big{|}_{\theta = 0} \gamma_\mathbf{v} (\theta) \right).
$$ 
Note that $r_1$ and $|{d}_1| / 2$ are chosen to make $\phi$ symplectic. Regarding $\mathbb{S}^2$ as a submanifold of $\R^3$, the disk bundle $D_{r_1}(T^{}S^2)$ is 
$$
D_{r_1}(T^{}S^2) \simeq \{ (a_1, a_2) \in \R^3 \times \R^3 \mid \| a_1 \| = 1, \langle a_1, a_2 \rangle = 0 \mbox{ and } \| a_2 \| \leq r_1\}
$$
and $\phi$ can be also described as
$$
\phi \colon U_{w_1} \to D_{r_1}(T^{}\mathbb{S}^2), \quad \phi (\mathbf{v}) = \left(\pi (\mathbf{v}), \frac{|d_{1}|}{2} \, (n_{0,1} \times \pi (\mathbf{v})) \right).
$$ 

Finally, the map $\phi$ is $\mathbb{S}^1$-equivariant with respect to the bending $\mathbb{S}^1$-action along $d_{s,1}$ on $U_{w_1} \backslash L(\mathbf{u}_{\mathbb{S}^2})$ and the $\mathbb{S}^1$-action generated by the geodesic flow on $D_{r_1}(T^{}\mathbb{S}^2) \backslash \{ o_{\mathbb{S}^2} \}$. To see this, observe the bending $\mathbb{S}^1$-action rotates the triangle $(v_{s,1}, v_{s,2}, -d_{s,1})$ along the direction of $d_{s,1}$. It corresponds to the geodesic flow in the direction of $n_{0,1} \times \pi (\mathbf{v})$ because the direction of $d_{s,1}$ converges to the direction of $n_{0,1}$ as $s \to 0$. Hence, the symplectic cut of the bending $\mathbb{S}^1$-action is symplectomorphic to the symplectic cut of the $\mathbb{S}^1$-action arising from geodesic flow. 

Moreover, the induced symplectic form on the cut space $\overline{{U}}_{w_1}$ is K\"{a}hler since ${U}_{\bS^2}$ is K\"{a}hler. By Lemma~\ref{lemma_monotoneutw1}, the cut space $\overline{{U}}_{w_1}$ is monotone, the K\"{a}hler form associated to the symplectic cut is in the same de Rham cohomology class $k \cdot [ (\omega_{FS} \times \omega_{FS} ) ]$ for $k = 4(2 - w_1)$. By applying the Moser argument to the linear interpolation of the K\"{a}hler forms, the space $\overline{{U}}_{w_1}$ is symplectomorphic to $\CP^1 \times \CP^1$ equipped with a multiple of the product of Fubini--Study forms. 
\end{proof}

By applying the same arguments in Proposition~\ref{proposition_effectiveclasses} and~\ref{proposition_countinginvarinatone}, one can verify that the following homotopy classes
$$
\overline{\alpha}, \alpha_{2,0}, \alpha_{1,2}, \alpha{(c)} \coloneqq c \cdot \alpha_1 + (\alpha_{2,0} + \alpha_{1,2}) / 2  \mbox{ for $c \in \Z$} 
$$
can contribute the disk potential of $\overline{L}(\mathbf{w})$ and the counting invariants of $\overline{\alpha}, \alpha_{2,0}$, and $\alpha_{1,2}$ are all one, see Lemma~\ref{lemma_monotoneutw1} for the definition of the class $\overline{\alpha}$. Note that $\partial \overline{\alpha}, \partial \alpha_{2,0}, \partial  \alpha_{1,2}$, and $\partial \alpha{(c)}$  correspond to $(1, 0), (-1, -1), (-1, 1)$, and $(-1,0)$ in $N$, respectively. For $j = 1, 2$, let $y_j$ be the holonomy variable associated to the $\mathbb{S}^1$-orbit generated by the component $\Phi_j$.  Since the class in $\partial \alpha{(c)}$ is same regardless of the value of $c$, those classes contribute to the term of the form $1/y_1$. By Gromov's compactness theorem, the counting invariants $n_{\alpha{(c)}}$ are all zero except finitely many. Then the disk potential function of $\overline{L}(\mathbf{w})$ is of the form
$$
W_{\overline{L}(\mathbf{w})} (\mathbf{y})  = y_{1} + \frac{1}{y_{1}y_{2}} + \frac{y_{2}}{y_{1}} + \frac{\kappa}{y_{1}}
$$
for some $\kappa \in \Z$. The following lemma computes the number $\kappa$.

\begin{lemma}\label{lemma_diskpotus2}
For $\mathbf{w} = (w_1, 2) \in \Delta_{\bS^2}$ with $1.5 < w_1 < 2$, the sum of the counting invariants is
\begin{equation}\label{equ_countinginvariantssumell}
\kappa = \sum_{c  \in \Z} n_{ \alpha{(c)}} = 2.
\end{equation}
\end{lemma}

\begin{proof}
Regrading the quantum cohomology as a vector space, it has the eigenvalues $\{0, 4, -4\}$ of the linear transformation induced by the quantum cup product of the first Chern class of the tangent bundle of $\CP^1 \times \CP^1$. Then the object in the monotone Fukaya category can be non-zero only when the critical value of a deformation datum is $\{0, 4, -4\}$ by the structural theorem of the monotone Fukaya category by Sheridan \cite[Corollary 1.5]{SheridanFano}. It determines $\kappa = 2$, to see more details, the reader is referred to \cite[Section 6]{KimQuad}. Thus,~\eqref{equ_countinginvariantssumell} is derived. 
\end{proof}

\begin{remark}
As pointed out by the referee, by Lemma~\ref{lemma_sympleccp1cp1}, one can see that the monotone Lagrangian torus ${\overline{L}(\mathbf{w})}$ maps into the Chekanov torus under the symplectomorphism from $\overline{U}_{w_1}$ to $\CP^1 \times \CP^1$. Then~\eqref{equ_countinginvariantssumell} immediately follows from \cite{AurouxSpecial, FukayaOhOhtaOnoS2S2}.
\end{remark}

The second local model is contained in the caterpillar bending system $\Phi_\mathbf{r}$ on the polygon space $\mcal{M}_\mathbf{r}$ of complex dimension three (i.e., $n = 3$) for $\mathbf{r} = (r_1,r_2,\cdots,r_{n+3=6})$ with $r_1 = r_2 = \dots = r_{n+3} = 1$. Set $\mathbf{u}_{\mathrm{SO}(3)} \coloneqq (1,3,3)$. We now focus on the local model around $\mathbf{u}_{\mathrm{SO}(3)}$ at which a Lagrangian special orthogonal group $\mathrm{SO}(3)$ is located.
\begin{enumerate}
\item Let $\Delta_{\mathrm{SO}(3)}$ be the subset of the caterpillar bending polytope $\Delta_\mathbf{r} \subset \R^3$ given by
\begin{equation}
\begin{split}
&\ell_{2,2}(\mathbf{u}) = -u_{2} + u_{3} \geq 0, \,  \ell_{2,1}(\mathbf{u}) =  2+ u_{1} - u_{2} \geq 0, \, \\
&\ell_{2,0}(\mathbf{u}) = 4 - u_{1} - u_{2} \geq 0, \, \ell_{3,0}(\mathbf{u}) = 6 - u_{2} - u_{3} \geq 0, \mbox{ and } -2 + u_{2} > 0
\end{split}
\end{equation}
as depicted in the second figure of Figure~\ref{fig_local}.
\item Consider the toric degeneration $\{\Phi^t\}$ of $\Phi_\mathbf{r}$ restricted to the inverse image $(\Phi^t)^{-1} \left( \Delta_{\mathrm{SO}(3)} \right)$ in Corollary~\ref{theorem_toricdegepolygon}. Let ${U}_{\mathrm{SO}(3)} \coloneqq \Phi_\mathbf{r}^{-1}(\Delta_{\mathrm{SO}(3)}) \subset \mcal{M}_\mathbf{r}$ and ${U}^t_{\mathrm{SO}(3)} \coloneqq \psi_{1,t}({U}_{\mathrm{SO}(3)})$, that is, ${U}^t_{\mathrm{SO}(3)} = (\Phi^t)^{-1}(\Delta_{\mathrm{SO}(3)})$.
\end{enumerate}

\begin{remark}
Additional explanation on $(2)$ is in order. Since $\mathbf{r}$ is not generic in this case, $\mcal{M}_\mathbf{r}$ is a singular variety. Restricting to $u_2 > 2$, the inverse image ${U}^t_{\mathrm{SO}(3)}$ is smooth. The map $\psi_{t,s} \colon {U}^t_{\mathrm{SO}(3)} \to {U}^s_{\mathrm{SO}(3)}$ and $\Phi^t \colon {U}^t_{\mathrm{SO}(3)}  \to \Delta_{\mathrm{SO}(3)}$  are obtained from the toric degeneration of $\mcal{O}_\lambda$ restricted to the subset given by $u_{2,2} > 2$.
\end{remark}

For a point $\mathbf{u} = (u_{1}, u_{2}, u_{3}) \in \Delta_{\mathrm{SO}(3)}$, let $L(\mathbf{u}) \coloneqq \Phi_\mathbf{r}^{-1}(\mathbf{u})$ and $L^t(\mathbf{u}) \coloneqq \psi_{1,t}(L(\mathbf{u}))$. 
By Theorem~\ref{theorem_Bouloc}, the fiber $L(\mathbf{u}_{\mathrm{SO}(3)})$ is a Lagrangian special orthogonal group $\mathrm{SO}(3)$, which is diffeomorphic to the real projective space $\R P^3$ and the subset ${U}_{\mathrm{SO}(3)}$ is symplectically embedded into to the cotangent bundle $(T^* \R P^3, \omega_{\mathrm{std}})$. Fix $\mathbf{w} = (1, w_2, 3) \in \Delta_{\mathrm{SO}(3)}$ with $2.5 < w_2 < 3$. Consider the inverse image
\begin{equation}\label{equ_cutboundary2}
U_{w_2} \coloneqq \Phi_\mathbf{r}^{-1} \left( \{ \mathbf{u}  \in \Delta_{\mathrm{SO}(3)} \mid u_2 \geq -3 + 2 w_2 \} \right) \subset {U}_{\mathrm{SO}(3)} \mbox{ and } U^t_{w_2} \coloneqq \psi_{1,t}(U_{w_2}).
\end{equation}
Take the symplectic cut of $U^t_{w_2}$ along its boundary with respect to the $\mathbb{S}^1$-action generated by $\Phi_2^t$. The resulting cut space is denoted by $\overline{U}^t_{w_2}$ and is a symplectic manifold equipped with the induced symplectic from ${U}^t_{\mathrm{SO}(3)}$. For a point $\mathbf{u} \in \{ \mathbf{u}  \in \Delta_{\mathrm{SO}(3)} \mid u_2 > -3 + 2 w_2 \}$, we denote the Lagrangian torus in $\overline{U}^t_{w_2}$ corresponding to $L^t(\mathbf{u})$ by $\overline{L}^t(\mathbf{u})$. 

\begin{lemma}\label{lemma_monotoneutw2}
For each $t$ with  $t \in (0,1]$, the Lagrangian torus $\overline{L}^t(\mathbf{w})$ of $\overline{{U}}^t_{w_2}$ is monotone. In particular, the cut space $\overline{{U}}^t_{w_2}$ is monotone.
\end{lemma}

The proof of the Lemma~\ref{lemma_monotoneutw2} is omitted because it is similar to that of Lemma~\ref{lemma_monotoneutw1}.

\begin{lemma}\label{lemma_sympleccp3}
For each $t$ with $t \in (0,1]$,  the cut space $\overline{{U}}_{w_2}^{t}$ is symplectomorphic to $\CP^3$ equipped with a multiple of the Fubini--Study form.
\end{lemma}

\begin{proof}[Sketch of Proof]
We may assume that $t = 1$. As we did in Lemma~\ref{lemma_sympleccp1cp1}, we relate our case to a well-known construction of $\CP^3$ from a disk bundle of the tangent bundle of $\mathbb{R}P^3 \simeq \mathbb{S}^3 / (\mathbb{Z}/2)$. Identifying $T^{ } \mathbb{R}P^3$ with $T^*\mathbb{R}P^3$ via the round metric, we regard $T^{ } \mathbb{R}P^3$ as a symplectic manifold. Then $\CP^3$ is obtained by the symplectic cut on $T^{ } \mathbb{R}P^3$ along a level set of a moment map $\mu_{\mathbb{S}^1}$ of the Hamiltonian $\mathbb{S}^1$-action generated by the geodesic flow, see in \cite[Section 3.1]{OakleyUsher}.

For $[\mathbf{v}] = [ (v_1, v_2, \cdots, v_6)] \in U_{w_2} \backslash L(\mathbf{u}_{\mathrm{SO}(3)})$, we denote by $\mathbf{u}_\mathbf{v}$ the position of $\mathbf{v}$, i.e. $\Phi_\mathbf{r}([\mathbf{v}]) = \mathbf{u}_{\mathbf{v}}$. We take the line segment connecting $\mathbf{u}_{\mathrm{SO}(3)}$ and $\mathbf{u}_\mathbf{v}$ as
$$
\mathbf{u}_s \coloneqq (1-s) \mathbf{u}_{\mathrm{SO}(3)} + s \mathbf{u}_\mathbf{v}.
$$
Note that the diagonal vector $d_{2}$ shrinks to the zero vector as $s \to 0$. Each representative $\mathbf{v}$ in $\R^3$ consists of four triangles
$$
(v_1, v_2, -d_1), (d_1, v_3, -d_2), (d_2, v_4, -d_3), \mbox{ and } (d_3, v_4, v_5).
$$
Let $(x,y,z)$ be the coordinate system for the Euclidean space $\R^3$ in which the polygon $\mathbf{v}$ is contained. To represent an equivalence class $[\mathbf{v}]$, we take its slice in a way that the vector $d_2$ is located at the positive $x$-axis and the second triangle $(d_2, v_4, -d_3)$ is contained in $\{ (x,y,z) \mid y \geq 0, z =0 \}$. We take a family 
$$
\{ \mathbf{v}_s = (v_{s,1}, v_{s,2}, \cdots, v_{s,6}) \mid s \in (0,1]\}
$$ 
of points (representatives) in $U_{w_2}$ such that 
\begin{enumerate}
\item $\mathbf{v}_s$ is contained in the fiber over $\mathbf{u}_s$,
\item the given point $\mathbf{v}$ occurs at $s=1$,
\item every element $\mathbf{v}_s$ is contained in $\mathcal{M}_\mathbf{r}$, and
\item the angles between triangles in $\mathbf{v}$ are same as those between triangles in $\mathbf{v}_s$.
\end{enumerate}

To describe $\pi(\mathbf{v})$ under the projection map $\pi \colon U_{w_2} \to \mathrm{SO}(3)$ in~\eqref{equ_cutboundary2} of the disk bundle, we take a limit of the above family
$$
\mathbf{v}_0 \coloneqq (v_{0,1}, v_{0,2}, \cdots, v_{0,6}) \coloneqq \lim_{s \to 0} (v_{s,1}, v_{s,2}, \cdots, v_{s,6}).
$$
Note that the limit configuration consists of two triangles $(v_{0,1}, v_{0,2}, v_{0,3})$ and $(v_{0,4}, v_{0,5}, v_{0,6})$. For $i = 1$ or $4$, let $\rho_{\mathbf{v}, i}$ be an element of $\mathrm{SO}(3)$ acting on $\R^3$ which sends the triangle $(v_{0,i}, v_{0,i+1}, v_{0,i+2})$ to its congruent triangle such that
\begin{enumerate}
\item $\rho_{\mathbf{v},i} (v_{0,i})$ is located at the positive $x$-axis and
\item $\rho_{\mathbf{v},i} ((v_{0,i}, v_{0,i+1}, v_{0,i+2}))$ is contained in $\{ (x,y,z) \mid y \geq 0, z =0 \}$.
\end{enumerate}
Then the projection map $\pi$ is defined by 
$$
\pi (\mathbf{v}) = \rho_{\mathbf{v},4}^{-1} \circ  \rho^{}_{\mathbf{v},1} \in \mathrm{SO}(3).
$$ 
In other words, we first locate $\mathbf{v}_0$ at the ``reference" position and then assigns an orientation preserving isometry from the reference position to the first triangle $\rho^{-1}_{\mathbf{v},4}((v_{0,1}, v_{0,2}, v_{0,3}))$ as the output of the projection $\pi$. 

Consider
$$
{n}_{0,2} \coloneqq  \rho_{\mathbf{v},4}^{-1} \left( \lim_{ \| d_{s,2} \| \to 0} \frac{{d}_{s,2}}{\|{d}_{s,2}\|} \right) \in \mathbb{S}^2.
$$
By bending the first triangle along the axis $n_{0,2}$, we have an $\mathbb{S}^1$-family $\gamma_\mathbf{v}(\theta)$ of elements in $\mathrm{SO}(3)$ such that $\gamma_\mathbf{v}(0) = \pi (\mathbf{v})$. By differentiating the curve, we obtain a tangent vector of $\mathrm{SO}(3)$ at $\pi(\mathbf{v})$. The magnitude will be determined by the difference of the second component of $\mathbf{u}_{\mathrm{SO}(3)} - \mathbf{u}_{\mathbf{v}}$. Taking $r_2 = (3 - w_2)$, the bundle map is given by
$$
\phi \colon U_{w_2} \to D_{r_2}(T^{}\mathrm{SO}(3)), \quad \phi (\mathbf{v}) = \left( \pi(\mathbf{v}), \frac{|d_2|}{2} \cdot \frac{d}{d \theta} \Big{|}_{\theta = 0} \gamma_\mathbf{v}(\theta) \right).
$$

Moreover, the map $\phi$ is $\mathbb{S}^1$-equivariant with respect to the bending $\mathbb{S}^1$-action along $d_{s,2}$ on $U_{w_2} \backslash L(\mathbf{u}_{\mathrm{SO}(3)})$ and the $\mathbb{S}^1$-action generated by the geodesic flow on $D_{r_2}(T^{}\mathrm{SO}(3)) \backslash \{o_{\mathrm{SO}(3)} \} \simeq D_{r_2}(T^{}\R P^3) \backslash \{o_{\R P^3} \} $. To see that this bending action corresponds to the geodesic flow of $T^{} \R P^3$, recall that the geodesic flow on $T^{}{\R P^3}$ is induced from one on $T^{}{\mathbb{S}^3}$ regarding $\R P^3 \simeq \mathbb{S}^3 / (\Z / 2)$. Also, $\R P^3$ is diffeomorphic to the quotient space $\mathbb{S}^3 / (\Z / 2) \simeq \mathbb{D}^3 / \sim$ where $\sim$ is given by the identification of the antipodal points at its boundary. Consider the isomorphism $\mathrm{SO}(3) \to \mathbb{D}^3 / \sim$ defined by $(v, \theta) \to [v \theta / {\pi}]$ where $v$ is the axis of rotations and $\theta$ ($- \pi \leq \theta \leq \pi$) is a rotational angle. Thus, the bending $\mathbb{S}^1$-action induces a rotation with a fixed axis, which generates the flow traveling along a line passing through the origin in the disk $\mathbb{D}^3$. It lifts to the geodesic flow in $\R P^3$. 

The induced symplectic form on the cut space $\overline{{U}}_{w_2}$ is K\"{a}hler since ${U}_{\mathrm{SO}(3)}$ is K\"{a}hler. By Lemma~\ref{lemma_monotoneutw2}, the cut space $\overline{{U}}_{w_2}$ is monotone, the K\"{a}hler form associated to the symplectic cut is in the same de Rham cohomology class $k \cdot [ \omega_{FS} ]$ for $k = 4(3 - w_2)$. By applying the Moser argument to the linear interpolation of the K\"{a}hler forms, the space $\overline{{U}}_{w_2}$ is symplectomorphic to $\CP^3$ equipped with a multiple of the Fubini--Study form. 
\end{proof}

By applying the same arguments in Proposition~\ref{proposition_effectiveclasses} and~\ref{proposition_countinginvarinatone}, one can verify that the following homotopy classes
$$
\overline{\alpha}, \alpha_{2,2}, \alpha_{2,1}, \alpha_{2,0}, \alpha_{3,0}, \alpha \coloneqq  ({\alpha_{3,0} + \alpha_{2,2}})/{2} =  ({\alpha_{2,0} + \alpha_{2,1}})/{2}
$$
can contribute the disk potential of $\overline{L}(\mathbf{w})$ and the counting invariants of $\alpha_{2,2}, \alpha_{2,1}, \alpha_{2,0}, \alpha_{3,0}$, and $\overline{\alpha}$ are all one. Here, $\overline{\alpha}$ is the class in a gradient disk generated by the $\mathbb{S}^1$-action associated with the moment map $(2w_2 - 3) - \Phi_2$. Note that $\partial \overline{\alpha}, \partial \alpha_{2,2}, \partial \alpha_{2,1}, \partial  \alpha_{2,0}, \partial \alpha_{3,0}$, and $\partial \alpha$ correspond to the lattice points $(0, 1, 0), (0, -1, 1), (1, -1, 0), (-1,-1,0), (0, -1, -1)$, and $(0, -1, 0)$ in $N$, respectively. Therefore, the disk potential function of $\overline{L}^1$ is of the form
$$
W_{\overline{L}(\mathbf{w})} (\mathbf{y}) = y_{2} + \frac{y_{3}}{y_{2}} + \frac{y_{1}}{y_{2}} + \frac{1}{y_{1}y_{2}} + \frac{1}{y_{2}y_{3}} +  \frac{\kappa}{y_{2}} 
$$
for some $\kappa \in \Z$. The following lemma computes the number $\kappa$.


\begin{lemma}\label{lemma_diskpotus3}
For $\mathbf{w} = (1, w_0,3) \in \Delta_{\mathrm{SO}(3)}$ with $2.5 < w_0 < 3$, the counting invariant is
\begin{equation}\label{equ_countinginvariantssumell2}
\kappa = n_{\alpha} = 0.
\end{equation}
\end{lemma}

\begin{proof}
Regrading the quantum cohomology as a vector space, it has the eigenvalues $\{\pm 4, \pm 4 \sqrt{-1}\}$ of the linear transformation induced by the quantum cup product of the first Chern class of the tangent bundle of $\CP^3$. Then the object in the monotone Fukaya category can be non-zero only when the critical value of a deformation datum is $\{\pm 4, \pm 4 \sqrt{-1}\}$ again by the structural theorem of the monotone Fukaya category by Sheridan \cite[Corollary 1.5]{SheridanFano}. It determines $\kappa = 0$ as desired, see \cite[Section 7]{KimQuad}
\end{proof}

\subsection{Global model}\label{section_Globalmodel}

The goal of this section is to complete the proof of Theorem~\ref{theorem_main}. Let us fix $\mathbf{r}$ with $r_1 = r_2 = \dots = r_{n+3} = 1$. Consider the face $g_k$ of $\Delta_{\mathbf{r}}$ contained in $u_{k} = k+1$. 

\begin{enumerate}
\item When $k = 1$ or $n$, the face $g_k$ is of codimension two in the $n$-dimensional polytope $\Delta_\mathbf{r}$. The fiber over every point in the relative interior of $g_k$ is diffeomorphic to $\bS^2 \times T^{n-2} \eqqcolon \mcal{S}_k$.
\item When $k = 2, 3, \dots, n-2$, or $n-1$, the face $g_k$ is of codimension three in the $n$-dimensional polytope $\Delta_\mathbf{r}$. The fiber over each point in the relative interior of $g_k$ is diffeomorphic to $\mathrm{SO}(3) \times T^{n-3} \eqqcolon \mcal{S}_k$.
\end{enumerate}

Recall that the following facets intersect at the stratum $g_k$.
$$
\begin{cases}
\ell_{k,0}(\mathbf{u}) = 2k - u_{k-1} - u_{k} = 0, &\ell_{k+1,0}(\mathbf{u}) =  2(k+1) - u_{k} - u_{k+1} = 0, \\
\ell_{k,1}(\mathbf{u}) = 2 - u_{k} + u_{k-1} = 0, &\ell_{k,2}(\mathbf{u}) = u_{k+1} - u_{k} = 0.
\end{cases}
$$
Here $u_{0}$ and $u_{n+1}$ are set to the constants $0$ and $n+1$, respectively.  For sufficiently small $\varepsilon > 0$, let ${U}_k$ be the the following inverse image$\colon$ 
$$
U_k \coloneqq \Phi_\mathbf{r}^{-1} \left( \left\{\mathbf{u} \in \Delta_\mathbf{r} \mid u_{k} > k + \varepsilon \right\}  \, \cap \, \bigcap_{j \neq k-1, k, k+1} \left\{\mathbf{u} \in \Delta_\mathbf{r} \mid |u_j -j | < \varepsilon \right\} \right).
$$ 

\begin{proposition}\label{proposition_producttype}
For a choice of $\mathbf{r}$ with $r_1 = r_2 = \dots = r_{n+3} = 1$, consider the face $g_k$ contained in $u_{k} = k+1$. Then the open set ${U}_{k}$ is diffeomorphic to $T^* \mcal{S}_k$. Namely,
\begin{enumerate}
\item $T^* \mcal{S}_k \coloneqq T^*\bS^2 \times (T^*\bS^1)^{n-2}$ when $k = 1$ or $n$,
\item $T^* \mcal{S}_k \coloneqq T^*SO(3) \times (T^*\bS^1)^{n-3}$ when $k = 2, 3, \cdots, n-2$, or $n-1$.
\end{enumerate}
Moreover, the Hamiltonian torus action generated by the components in $\mcal{M}_\mathbf{r}$ other than $\Phi_{k-1}, \Phi_{k}$, and $\Phi_{k+1}$ acts on the cotangent bundle of the torus factors of $\mcal{S}_k$. 
\end{proposition}

\begin{proof}
Since the method proving the statement is similar for all $k$, we exhibit the proof of the statement only when $k = 1$. Set $k \coloneqq 1$.
Consider a toric degeneration 
$$
\psi_{1,0} \colon {U}_{k} = {U}^1_{k} \to {U}^0_{k}
$$ 
induced from the toric degeneration of $\mathcal{M}_\mathbf{r}$. Let $\Phi^t_{\geq \ell} \coloneqq  (\Phi^t_{\ell}, \Phi^t_{\ell+1}, \dots, \Phi^t_{n}).$ We take the level set of $\Phi^t_{\geq 3}$ at which the monotone Lagrangian torus fiber is located. Namely, we take
$$
{V}_k^t \coloneqq (\Phi^t_{\geq 3})^{-1}(3,4, \dots, n).
$$

For each $t$, the torus action generated by $\Phi^t_{\geq 3}$ on ${V}_{k=1}^t$ is free. We then have the principal $T^{n-2}$-bundle over the quotient space ${V}_k^t / T^{n-2}$. By the equivariance of $\psi_{1,0}$ under the $T^{n-2}$-action, we obtain the following bundle morphism of principal $T^{n-2}$-bundles
\begin{equation}\label{equ_commutdiapsi10}
\xymatrix{
{V}_k^1  \ar[rr]^{\psi_{1,0}} \ar[d] & & {V}_k^0 \ar[d] \\
\underline{{U}}_k^1\coloneqq {{U}}_k^1 / (\C^*)^{n-2} ={V}_k^1 / T^{n-2} \ar[rr]^{\underline{\psi}_{1,0}} & & {V}_k^0/T^{n-2} =  {{U}}_k^0 / (\C^*)^{n-2} \eqqcolon \underline{{U}}_k^0.
}
\end{equation}

The fan associated to the toric variety ${U}^0_{k=1}$ consists of two $1$-cones so that the cones generate two-dimensional subspace and do \emph{not} span $N_\R$. It has a $(\C^*)^{n-2}$-torus factor and ${U}^0_{k}$ can be trivialized as ${U}^0_{k} \simeq \underline{{U}}^0_k \times (\C^*)^{n-2}$. This trivialization induces a trivialization of ${V}_k^0 \simeq \underline{{U}}^0_k \times T^{n-2}$ and hence ${V}^0_{k}$ is a trivial principal $T^{n-2}$-bundle over $\underline{{U}}^0_k$. Therefore, the pullback bundle by $\underline{\psi}_{1,0}$ is also trivial. We have a bundle map from the pullback bundle $\underline{\psi}_{1,0}^* {{V}}^0_k$ to ${{V}}^1_k$ over the same base $\underline{{U}}_k^1$. As this bundle map is an isomorphism, the bundle ${{V}}^1_k \to \underline{{U}}^1_k$ is then also trivial. In particular, we obtain a trivialization 
$$
{{V}}^1_k \simeq \underline{{U}}^1_k \times T^{n-2}.
$$
A priori, this trivialization is a homeomorphism as the horizontal maps in~\eqref{equ_commutdiapsi10} are continuous. We may assume that this trivialization is a diffeomorphism because of the equivalence of smooth and continuous equivalences of principal bundles, see \cite{MullerWockel} for instance.

Recall the center of $\Delta_\mathbf{r}$ is located at $(u_1 = 1, u_2 = 2, \cdots, u_n =n)$ in~\eqref{equ_centerofdeltamathbfr} under our choice of $\mathbf{r} = (1, 1, \cdots, 1)$.
Consider the sequence of following maps, which project to an action variable$\colon$
\begin{itemize}
\item $\Phi_3 \colon {U}^1_k \to \R$
\item $\Phi_4 \colon \Phi_3^{-1}(3) \to \R$
\item $\Phi_5 \colon (\Phi_3, \Phi_4)^{-1}(3, 4) \to \R$
\item $\cdots$
\item $\Phi_n \colon  (\Phi_3, \Phi_4, \cdots, \Phi_{n-1})^{-1}(3, 4, \cdots, n-1) \to \R$
\end{itemize}
The image of each map is an open interval, which is contractible. Then the map $\Phi_3$ is a \emph{trivial} fiber bundle over $\Phi_3({U}^1_k)$ whose typical fiber is $\Phi_{3}^{-1}(3)$. In particular, ${U}^1_k$ and $\Phi_{3}^{-1}(3) \times \R$ are diffeomorphic. Proceeding inductively, we obtain a diffeomorphism 
$$
{U}^1_k \simeq {V}^1_k \times \R^{n-2}.
$$
In summary, we have 
\begin{equation}\label{equ_trivializationcho}
{U}^1_k \simeq  {V}^1_k \times \R^{n-2} \simeq \underline{{U}}^1_k \times T^{n-2} \times \R^{n-2}.
\end{equation}

The map $(\Phi_1, \Phi_2)$ factors through $ \underline{{U}}^1_k$ and we then have
$$
\xymatrix{
{V}^1_k \ar[rr]^{(\Phi_1, \Phi_2)} \ar[rd] & & \R^2  \\
& \underline{{U}}^1_k  \ar[ru]_{(\underline{\Phi}_1, \underline{\Phi}_2)} &  } 
$$
Then $\underline{{U}}^1_k$ can be identified with the first local model ${U}_{\bS^2}$ as follows. 
Since ${V}_k^1 = (\Phi^1_{\geq 3})^{-1}(3,4, \dots, n)$ is a subset of $\mathcal{M}_\mathbf{r}$, each point in this set is an equivalence class of an $(n+3)$ tuple of vectors such that
\begin{enumerate}
\item $v_1 + v_2 + \dots + v_{n+3} = 0$
\item $|d_3| = |d_4| = \cdots = |d_n| = 1$.
\end{enumerate}
Note that the lengths of diagonals $d_3, d_4, \cdots, d_n$ are determined because the level of $\Phi_{\geq 3}$ is chosen as $(3, 4, \cdots, n)$. As in Figure~\ref{fig_identification}, consider the \emph{forgetful} map given by
$$
\Upsilon \colon {{V}}^1_k \to {U}_{\bS^2}, \quad 
[(v_1, v_2, \dots, v_{n+3})] \mapsto [(v_1, v_2, \dots, v_4, -d_3)].
$$
This map is well-defined because $|d_3| = 1$ and the length of each edge of ${U}_{\bS^2}$ is one. Moreover, the map factors through $\underline{{U}}_k^1$ and induces 
\begin{equation}\label{equ_underlineupsilon}
\underline{\Upsilon} \colon \underline{{U}}^1_k \to {U}_{\bS^2}, \quad 
\left[ [(v_1, v_2, \dots, v_{n+3})] \right] \mapsto [(v_1, v_2, \dots, v_4, -d_3)].
\end{equation}
Explicitly, ${{U}}_k^1$ and ${U}_{\bS^2}$ are identified via the map $\underline{\Upsilon}$. Hence, combining with~\eqref{equ_trivializationcho}, we have
$$
{U}^1_k \simeq \underline{{U}}^1_k \times T^{n-2} \times \R^{n-2} \simeq T^*\bS^2 \times T^*(\bS^1)^{n-2}.
$$
\end{proof}

\begin{figure}[h]
	\scalebox{0.95}{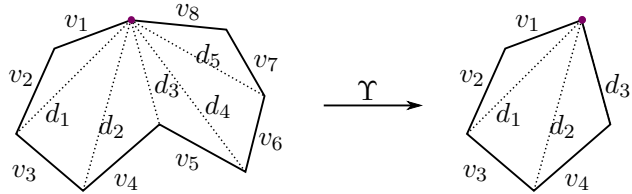}
	\caption{\label{fig_identification} The map $\Upsilon$.}	
\end{figure}

To compute the remaining counting invariants, we need to carry the Lagrangian torus at the center to one near the Lagrangian fiber of the form $\bS^2 \times T^{n-2}$ or $\mathrm{SO}(3) \times T^{n-3}$ and pass it into the local model. When passing it into the local model, we need to make sure that the counting invariant of a  targeted class does not change through the process. For this reason, we take the line in the center $\mathbf{u}_0$ and the center $\mathbf{u}_k$ of the face $g_k$. We denote by the line segment $[\mathbf{u}_0, \mathbf{u}_k]$. The fiber at each point of the half open segment $[\mathbf{u}_0, \mathbf{u}_k) \coloneqq [\mathbf{u}_0, \mathbf{u}_k] - \{ \mathbf{u}_k \}$ is a Lagrangian torus. We have the fiber $\bS^2 \times T^{n-2}$ or $\mathrm{SO}(3) \times T^{n-3}$ at the point $\mathbf{u}_k$. A reason to take the line segment is that the carried Lagrangian torus in the local model becomes a monotone Lagrangian torus, which will facilitate counting pseudo-holomorphic disks.

For a choice of a point $\mathbf{w}_0 \in [\mathbf{u}_0, \mathbf{u}_k)$ and each $t$, the choice of the segment $[\mathbf{u}_0, \mathbf{w}_0]$ leads to a Lagrangian isotopy $\iota^t \colon [0,1] \times T^{n} \to {X}^t$ such that $\iota^t (0, T^n) = L^t(\mathbf{u}_0)$ and  $\iota^t(1, T^n) = L^t(\mathbf{w}_0)$. To show that the counting invariants are preserved through the isotopy, we construct a cobordism between the moduli space of pseudo-holomorphic disks bounded by $L(\textbf{u}_0)$ and that by $L(\textbf{w}_0)$. As most of them are non-monotone in ${X}^t$, a Lagrangian torus over a certain point in the segment may bound a pseudo-holomorphic disks of Maslov index less than or equal to zero. Indeed, the existence of such a disk is sensitive to a choice of compatible almost complex structure. Nevertheless, we may choose a position $\textbf{w}_0$ ``close" to the center $\mathbf{u}_0$ so that the holomorphic disk of Maslov index $\leq 0$ does not occur along the isotopy.

\begin{lemma}\label{lemma_genericacs}
For a generic and compatible almost complex structure $J$, there exists a position $\mathbf{w}_0$ in $[\mathbf{u}_0, \mathbf{u}_k)$ such that the fiber of each point $\mathbf{u} \in [\mathbf{u}_0, \mathbf{w}_0]$ does not bound any  $J$-holomorphic disk of Maslov index less than or equal to zero.
\end{lemma}

\begin{proof}
First, the generic choice of compatible almost complex structure $J$ ensures that any simple $J$-holomorphic disks of Maslov index $- \ell$ for sufficiently large $\ell$ do not occur. Second, appearance of simple $J$-holomorphic disks of Maslov index $k$ for any given fixed integer $k \leq 0$ is a closed condition. Lastly, $L(\mathbf{u}_0)$ does not bound any non-constant $J$-holomorphic disks of Maslov index $\leq 0$ since $L(\mathbf{u}_0)$ is monotone. Therefore, combining the above three facts, we see that there exists a position $\mathbf{w}_0$ (depending on $J$) sufficiently close to $\mathbf{u}_0$ such that the torus located at each point $\mathbf{u} \in [\mathbf{u}_0, \mathbf{w}_0]$ does not bound any $J$-holomorphic disk of Maslov index $\leq 0$.
\end{proof}

When $k=1$, we are now aiming at computing the (sum of) counting invariants for the classes
\begin{equation}\label{equ_alphaell}
\alpha{(c)} \coloneqq c \cdot \alpha_1 + (\alpha_{1,0} + \alpha_{1,1}) / 2 = c \cdot \alpha_1 + (\alpha_{2,0} + \alpha_{1,2}) / 2 \quad (c \in \Z)
\end{equation}
in Proposition~\ref{proposition_effectiveclasses}.
The sum of counting invariants of those classes is equal to $\kappa_1$. The remaining invariants can be similarly calculated and we shall omit the process.

\begin{proposition}\label{proposition_monotoneuklu}
The counting invariant for $\alpha{(c)}$ in $\pi_2( \mcal{M}_\mathbf{r}, L(\mathbf{u}_0))$ is equal to that for $\alpha{(c)}$ in $\pi_2( \mcal{M}_\mathbf{r}, L(\mathbf{u}))$ for $\mathbf{u} \in  [\mathbf{u}_0, \mathbf{u}_k)$ such that $\mathbf{u}$ is sufficiently close to $\mathbf{u}_k$.
\end{proposition}

\begin{proof}
We proceed the proof of Proposition~\ref{proposition_monotoneuklu} in five steps. 
\begin{enumerate}
\item Proposition~\ref{lemma_monotonecbs} asserts that the Lagrangian torus $L^1(\mathbf{u}_0)$ at the center is monotone. For each nonzero $t$, the symplectomorphism $\psi_{1,t}$ gives rise to an oriented cobordism so that
$$
n(\alpha{(c)}; \mcal{M}_\mathbf{r} \simeq {X}^1, L^1(\mathbf{u}_0)) = n(\psi_{1,t,* } (\alpha{(c)}) ; {X}^t, L^t(\mathbf{u}_0))
$$
\item Note that each $L^t(\mathbf{u}_0)$ does not bound any stable map of Maslov index zero because of the monotonicity. For each $t$ sufficiently close to $0$, there exists a position $\mathbf{w}_0$ close to $\mathbf{u}_0$ in Lemma~\ref{lemma_genericacs}, the family of Lagrangian tori over $[\mathbf{u}_0, \mathbf{w}_0]$ does not bound any holomorphic disks of Maslov index less than or equal to zero. We then have an oriented cobordism between the moduli spaces so that the counting invariants do not change. (By abuse of notation, we shall employ the same notation for the relative homotopy classes in $\pi_2({X}^1, L^1(\mathbf{u}_0))$ to denote those in $\pi_2({X}^1, L^1(\mathbf{w}_0))$ via the isomorphism of the homotopy groups induced by the isotopy $\iota^1$.) 
Namely, we have
$$
n(\psi_{1,t,* }({\alpha{(c)}}); {X}^t, L^t(\mathbf{u}_0)) = n(\psi_{1,t,* }({\alpha{(c)}}); {X}^t, L^t(\mathbf{w}_0)).
$$

\item We claim that each holomorphic disk $\varphi$ bounded by the toric fiber $L^0(\mathbf{w}_0)$ in the class $\psi_{1,0,*} (\alpha{(c)}) =  (\beta_{2,0} + \beta_{1,2})/2$ is contained in ${U}^0_k$. To see this, let $\Pi$ be the toric morphism from a simplicialization $\widetilde{{X}}^0$ of ${X}^0$ to the toric variety ${X}^0$ associated with $\Delta_\mathbf{r}$. We denote by $\widetilde{L}^0(\mathbf{w}_0)$ the Lagrangian toric fiber such that $\Pi(\widetilde{L}^0(\mathbf{w}_0)) = L^0(\mathbf{w}_0)$. Consider the strict transformation $\widetilde{\varphi}$ of the holomorphic disk $\varphi$, which is a holomorphic disk bounded by $\tilde{L}^0(\mathbf{w}_0)$ in the class $(\widetilde{\beta}_{2,0} + \widetilde{\beta}_{1,2})/2$. The disk $\widetilde{\varphi}$ can only intersect the union of the toric divisors corresponding to $\ell_{2,0}(\mathbf{u})=0$ and $\ell_{1,2}(\mathbf{u})=0$. By the classification result \cite[Theorem 6.2]{ChoPoddar} and the modulus principle applying to the coordinate function $u_k$, $\widetilde{\varphi}$ is contained in $\Pi^{-1}(U_k^0)$ and hence $\varphi$ is contained in $U_k^0$.

Note that a sequence of $J_t$-holomorphic disks in $\psi_{1,t,*} (\alpha{(c)})$ converges to the holomorphic disk $\varphi$ in the class $\psi_{1,0,*} (\alpha{(c)})$. Since the moduli space of $J_1$-holomorphic disks in $\alpha{(c)}$ is compact, for $t$ sufficiently close to $0$, every $J_t$-holomorphic disk in the class $\psi_{1,t,*} (\alpha{(c)})$ is fully contained in ${U}^t_k$. Thus,
$$
n(\psi_{1,t,* }({\alpha{(c)}}); {X}^t, L^t(\mathbf{w}_0)) =  n(\psi_{1,t,* }({\alpha{(c)}}); {U}_k^t, L^t(\mathbf{w}_0)).
$$
\item The set ${U}_k^t$ is symplectically embedded into the cotangent bundle $T^* \mcal{S}_k$.
Since each Lagrangian torus at $\mathbf{u} \in  [\mathbf{u}_0, \mathbf{u}_k)$ is monotone, it does not bound any holomorphic disks of Maslov index $\leq 0$. Thus, we have
$$
n(\psi_{1,t,* }({\alpha{(c)}});{U}_k^t, L^t(\mathbf{w}_0)) = n(\psi_{1,t,* }({\alpha{(c)}}); {U}_k^t, L^t(\mathbf{u})).
$$
\item If $\mathbf{u}$ is close to $\mathbf{u}_k$, then every $J_t$-holomorphic disk in the class $\psi_{1,t,* }({\alpha{(c)}})$ is contained in ${U}_k^t$ for each nonzero $t$. Since $L^t(\mathbf{u})$ is monotone for each $t$, we have
$$
n(\psi_{1,t,* }({\alpha{(c)}});{U}_k^t, L^t(\mathbf{w}_0)) = n({\alpha{(c)}}; {U}^1_k, L^1(\mathbf{u})).
$$
\end{enumerate}
Combining the five steps, we establish the desired statement.
\end{proof}

\begin{remark}
Proposition~\ref{proposition_monotoneuklu} tells us  that the holomorphic disks in the class Proposition~\ref{proposition_effectiveclasses} can be captured by a local model. If an effective disk class were degenerated into a combination of the disk classes which can intersect many lower dimensional strata, then any local model cannot capture such a contribution to the disk potential function.
\end{remark}

We are ready to compute $\kappa_\bullet$'s.

\begin{proof}[Proof of Theorem~\ref{theorem_main}]
By Proposition~\ref{proposition_monotoneuklu}, we can carry the Lagrangian torus at the center to a Lagrangian torus $L$ near the Lagrangian fiber of the form $\bS^2 \times T^{n-2}$ or $\mathrm{SO}(3) \times T^{n-3}$ in $U_k$ without changing the counting invariants. We now take the symplectic reduction of the Hamiltonian torus action generated by the components other than $\Phi_{k-1}, \Phi_{k}$, and $\Phi_{k+1}$. By Proposition~\ref{proposition_producttype}, the reduced space is diffeomorphic to either the first local model ${U}_{\bS^2}$ or the second local model $U_{\mathrm{SO}(3)}$.

Concerning the symplectic forms on ${U}_{\bS^2}$ (or $U_{\mathrm{SO}(3)}$), there are two natural symplectic forms that one can adorn the space with. The first one is the reduced symplectic form and the other is one from the local model. When it comes to pseudo-holomorphic curves, the choice of symplectic forms does not cause any differences on counting invariants because the residual torus action on the reduced space agrees with the torus action from the local model. Thus, the counting invariants in the reduced space can be computed by Lemma~\ref{lemma_diskpotus2} and~\ref{lemma_diskpotus3}.

We claim that there are orientation preserving diffeomorphisms
$$
\begin{cases}
\mcal{M}( {U}_k, {L}; \alpha{(c)}) \to \mcal{M}( {U}_{\bS^2}, L / T^{n-2}; \underline{\alpha}{(c)}) \times T^{n-2}, \\
\mcal{M}( {U}_k, {L}; \alpha{(c)}) \to \mcal{M}( {U}_{\mathrm{SO}(3)}, L / T^{n-3}; \underline{\alpha}{(c)}) \times T^{n-3},
\end{cases}
$$
which is constructed below. We only discuss the isomorphism when $k = 1$ as the other cases are similar. We start to choose with a generic almost complex structure $J_{\bS^2}$ on the reduced space $\underline{U}_k \simeq U_{\bS^2}$ in~\eqref{equ_underlineupsilon}. Recall that $U_k \simeq U_{\bS^2} \times (T^* \bS^1)^{n-2}$ are diffeomorphic by Proposition~\ref{proposition_producttype}

 Take a generic almost complex structure $J_{\mathrm{prod}}$ of product type on the product $U_k \simeq U_{\bS^2} \times (T^* \bS^1)^{n-2}$ by taking the product of $J_{\bS^2}$ and an almost complex structure on $(T^* \bS^1)^{n-2}$. In the second component, there is no non-trivial holomorphic disk bounded by $\bS^1$ in $(T^* \bS^1)^{n-2}$. If $\varphi \colon (\mathbb{D}^2, \partial \mathbb{D}^2) \to (U_k, L)$ is a $J_{\mathrm{prod}}$-holomorphic curve, then $\mathrm{pr}_2 \circ \varphi$ is constant. Because of our choice of $J_{\mathrm{prod}}$, one can lift the holomorphic disk bounded by $L/T^{n-2}$ in $U_{\bS^2}$ to one bounded by $L$ in $U_{k}$. Thus, we have a diffeomorphism
$$
\mcal{M}( {U}_k, {L}, J_{\mathrm{prod}}; \alpha{(c)}) \to \mcal{M}( {U}_{\bS^2}, L / T^{n-2}; \underline{\alpha}{(c)}) \times T^{n-2}
$$
given $\varphi \mapsto (\mathrm{pr}_1 \circ \varphi, \mathrm{pr}_2 \circ \varphi)$. Giving the orientations and spin structures arising from the torus actions and their residual torus actions, the diffeomorphism is indeed orientation preserving. Finally, since $L$ is monotone, the counting is invariant under a generic choice of almost complex structure. Hence,~\eqref{equ_computationk1k2} is derived. 
\end{proof}

\appendix
\section{Polygon space of pentagons}\label{AppendixA}

Suppose that a tuple $\mathbf{r}$ is generic. The polygon space $\mathcal{M}_\mathbf{r}$ can admit a monotone torus fiber (and hence $\mathcal{M}_\mathbf{r}$ is monotone) even if $\mathbf{r}$ is \emph{not} equilateral. 
In this section, we discuss the disk potential function of a monotone torus fiber of the caterpillar bending system $\Phi_\mathbf{r}$ of some non-equilateral polygon space in the four dimensional case.

Depending on a choice of $\mathbf{r}$, we can produce different polygon spaces and complete integrable systems. Some classification result on polygon spaces (as algebraic varieties) was obtained in \cite{HausmannKnutson, AlessiaMilena}. We recall one fact on polygon spaces, see \cite[Section 5]{AlessiaMilena} for instance. 

\begin{lemma}\label{lemma_5tupletoric}
For a generic $5$-tuple $\mathbf{r}$, let $\mcal{M}_\mathbf{r}$ be the space of pentagons associated to $\mathbf{r}$. Then the caterpillar bending torus action on $\mcal{M}_\mathbf{r}$ is toric if and only if $r_1 \neq r_2$ and $r_4 \neq r_5$.
\end{lemma}

If the polygon space $\mcal{M}_\mathbf{r}$ is toric and the bending torus action is toric, then the disk potential function of a toric fiber can be computed by the result of the the second named author with Chan \cite{ChanLau}. This appendix only deals with non-toric examples. We can compute the disk potential by applying the same strategy in Section~\ref{sec_classificationofeffec} and~\ref{Section_computationofopenGWinv}.  That is, we first classify all possible effective disks using a toric degeneration. We then compute the counting invariants by considering on the local model. In the case of dimension four, $T^*\mathbb{S}^2$ is the only interesting local model. We only state the computational result without giving the proof.

\begin{itemize}
\item  For $\mathbf{r} = (1,1,3,3,3)$, we have $\mcal{M}_\mathbf{r} \simeq \CP^1 \times \CP^1$. The caterpillar bending polytope is depicted in the first figure of Figure~\ref{fig_momentpolytope}. The disk potential function of the Lagrangian monotone torus fiber $\Phi_\mathbf{r}^{-1}(\mathbf{u}_0)$ at $\mathbf{u}_0 = (1,2)$ is 
\begin{equation}\label{equ_catertorusfiber11322}
W_{\mathbf{r}}(\mathbf{y}) =y_1 + \frac{y_2}{y_1} + \frac{2}{y_1} + \frac{1}{y_1y_2}. 
\end{equation}

\item For $\mathbf{r} = (1,1,3,2,2)$, we have $\mcal{M}_\mathbf{r} \simeq \CP^2 \# 2\overline{\CP^2}$.  The caterpillar bending polytope is depicted in the second figure of Figure~\ref{fig_momentpolytope}. The disk potential function of the Lagrangian monotone torus fiber $\Phi_\mathbf{r}^{-1}(\mathbf{u}_0)$ at $\mathbf{u}_0 = (1,2)$ is 
\begin{equation}\label{equ_catertorusfiber11322}
W_{\mathbf{r}}(\mathbf{y}) =y_1 + \frac{y_2}{y_1} + \frac{2}{y_1} + \frac{1}{y_1y_2} + {y_2}. 
\end{equation}

%

\item For $\mathbf{r} = (2,2,2,4,4)$, we have $\mcal{M}_\mathbf{r} \simeq \CP^2 \# 3\overline{\CP^2}$.  The caterpillar bending polytope is depicted in the third figure of Figure~\ref{fig_momentpolytope}.  The disk potential function of the Lagrangian monotone torus fiber $\Phi_\mathbf{r}^{-1}(\mathbf{u}_0)$ at $\mathbf{u}_0 = (2,4)$ is 
\begin{equation}\label{equ_catertorusfiber22244}
W_{\mathbf{r}}(\mathbf{y}) =y_1 + \frac{y_2}{y_1} + \frac{2}{y_1} + \frac{1}{y_1y_2} + \frac{2}{y_2} + \frac{y_1}{y_2}. 
\end{equation}
\end{itemize}

\begin{figure}[h]
	\scalebox{1}{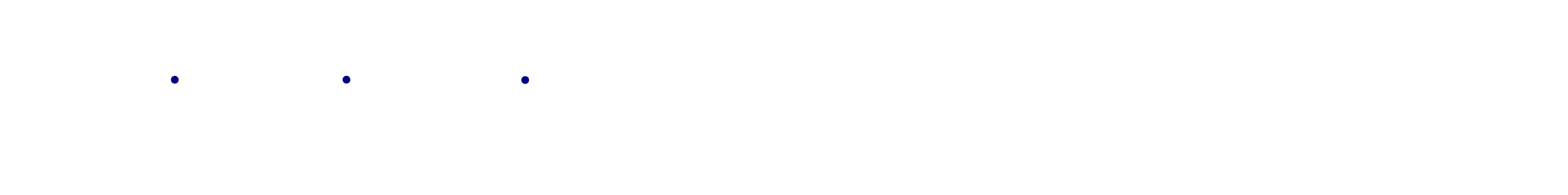}
	\caption{\label{fig_momentpolytope} Caterpillar bending polytopes with different choice of $\mathbf{r}$.}	
\end{figure}

%

\providecommand{\bysame}{\leavevmode\hbox to3em{\hrulefill}\thinspace}
\providecommand{\MR}{\relax\ifhmode\unskip\space\fi MR }
\providecommand{\MRhref}[2]{%
  \href{http://www.ams.org/mathscinet-getitem?mr=#1}{#2}
}
\providecommand{\href}[2]{#2}

\end{document}

%% file: fig_bendpotential.pdf_tex
\begingroup%
  \makeatletter%
  \providecommand\color[2][]{%
    \errmessage{(Inkscape) Color is used for the text in Inkscape, but the package 'color.sty' is not loaded}%
    \renewcommand\color[2][]{}%
  }%
  \providecommand\transparent[1]{%
    \errmessage{(Inkscape) Transparency is used (non-zero) for the text in Inkscape, but the package 'transparent.sty' is not loaded}%
    \renewcommand\transparent[1]{}%
  }%
  \providecommand\rotatebox[2]{#2}%
  \newcommand*\fsize{\dimexpr\f@size pt\relax}%
  \newcommand*\lineheight[1]{\fontsize{\fsize}{#1\fsize}\selectfont}%
  \ifx\svgwidth\undefined%
    \setlength{\unitlength}{245.50343198bp}%
    \ifx\svgscale\undefined%
      \relax%
    \else%
      \setlength{\unitlength}{\unitlength * \real{\svgscale}}%
    \fi%
  \else%
    \setlength{\unitlength}{\svgwidth}%
  \fi%
  \global\let\svgwidth\undefined%
  \global\let\svgscale\undefined%
  \makeatother%
  \begin{picture}(1,0.40015113)%
    \lineheight{1}%
    \setlength\tabcolsep{0pt}%
    \put(0.79076021,0.40247492){\color[rgb]{0,0,0}\makebox(0,0)[lt]{\begin{minipage}{0.19551661\unitlength}\raggedright $1$\end{minipage}}}%
    \put(0.62624171,0.31714346){\color[rgb]{0,0,0}\makebox(0,0)[lt]{\begin{minipage}{0.19551661\unitlength}\raggedright $2$\end{minipage}}}%
    \put(0.58531673,0.14654957){\color[rgb]{0,0,0}\makebox(0,0)[lt]{\begin{minipage}{0.19551661\unitlength}\raggedright $3$\end{minipage}}}%
    \put(0.69324346,0.01773378){\color[rgb]{0,0,0}\makebox(0,0)[lt]{\begin{minipage}{0.19551661\unitlength}\raggedright $4$\end{minipage}}}%
    \put(0.86644847,0.01773378){\color[rgb]{0,0,0}\makebox(0,0)[lt]{\begin{minipage}{0.19551661\unitlength}\raggedright $5$\end{minipage}}}%
    \put(0.97002333,0.14480882){\color[rgb]{0,0,0}\makebox(0,0)[lt]{\begin{minipage}{0.19551661\unitlength}\raggedright $6$\end{minipage}}}%
    \put(0.94391202,0.31888424){\color[rgb]{0,0,0}\makebox(0,0)[lt]{\begin{minipage}{0.19551661\unitlength}\raggedright $7$\end{minipage}}}%
    \put(0.20421037,0.40247492){\color[rgb]{0,0,0}\makebox(0,0)[lt]{\begin{minipage}{0.19551661\unitlength}\raggedright $1$\end{minipage}}}%
    \put(0,0){\includegraphics[width=\unitlength,page=1]{fig_bendpotential.pdf}}%
    \put(0.03969187,0.31714346){\color[rgb]{0,0,0}\makebox(0,0)[lt]{\begin{minipage}{0.19551661\unitlength}\raggedright $2$\end{minipage}}}%
    \put(-0.00123312,0.14654957){\color[rgb]{0,0,0}\makebox(0,0)[lt]{\begin{minipage}{0.19551661\unitlength}\raggedright $3$\end{minipage}}}%
    \put(0.10669362,0.01773378){\color[rgb]{0,0,0}\makebox(0,0)[lt]{\begin{minipage}{0.19551661\unitlength}\raggedright $4$\end{minipage}}}%
    \put(0.27989863,0.01773378){\color[rgb]{0,0,0}\makebox(0,0)[lt]{\begin{minipage}{0.19551661\unitlength}\raggedright $5$\end{minipage}}}%
    \put(0.38347349,0.14480882){\color[rgb]{0,0,0}\makebox(0,0)[lt]{\begin{minipage}{0.19551661\unitlength}\raggedright $6$\end{minipage}}}%
    \put(0.35736218,0.31888424){\color[rgb]{0,0,0}\makebox(0,0)[lt]{\begin{minipage}{0.19551661\unitlength}\raggedright $7$\end{minipage}}}%
    \put(0,0){\includegraphics[width=\unitlength,page=2]{fig_bendpotential.pdf}}%
  \end{picture}%
\endgroup%

%% file: fig_caterpillar.pdf_tex
\begingroup%
  \makeatletter%
  \providecommand\color[2][]{%
    \errmessage{(Inkscape) Color is used for the text in Inkscape, but the package 'color.sty' is not loaded}%
    \renewcommand\color[2][]{}%
  }%
  \providecommand\transparent[1]{%
    \errmessage{(Inkscape) Transparency is used (non-zero) for the text in Inkscape, but the package 'transparent.sty' is not loaded}%
    \renewcommand\transparent[1]{}%
  }%
  \providecommand\rotatebox[2]{#2}%
  \ifx\svgwidth\undefined%
    \setlength{\unitlength}{148.94180008bp}%
    \ifx\svgscale\undefined%
      \relax%
    \else%
      \setlength{\unitlength}{\unitlength * \real{\svgscale}}%
    \fi%
  \else%
    \setlength{\unitlength}{\svgwidth}%
  \fi%
  \global\let\svgwidth\undefined%
  \global\let\svgscale\undefined%
  \makeatother%
  \begin{picture}(1,0.81643601)%
    \put(0,0){\includegraphics[width=\unitlength,page=1]{fig_caterpillar.pdf}}%
    \put(0.23409175,0.80176611){\color[rgb]{0,0,0}\makebox(0,0)[lt]{\begin{minipage}{0.32227353\unitlength}\raggedright $v_1$\end{minipage}}}%
    \put(-0.00222927,0.53799312){\color[rgb]{0,0,0}\makebox(0,0)[lt]{\begin{minipage}{0.32227353\unitlength}\raggedright $v_2$\end{minipage}}}%
    \put(0.08371034,0.21571956){\color[rgb]{0,0,0}\makebox(0,0)[lt]{\begin{minipage}{0.32227353\unitlength}\raggedright $v_3$\end{minipage}}}%
    \put(0.50266595,0.0330979){\color[rgb]{0,0,0}\makebox(0,0)[lt]{\begin{minipage}{0.32227353\unitlength}\raggedright $v_4$\end{minipage}}}%
    \put(0.88939419,0.21571956){\color[rgb]{0,0,0}\makebox(0,0)[lt]{\begin{minipage}{0.32227353\unitlength}\raggedright $v_5$\end{minipage}}}%
    \put(0.92162155,0.59170535){\color[rgb]{0,0,0}\makebox(0,0)[lt]{\begin{minipage}{0.32227353\unitlength}\raggedright $v_6$\end{minipage}}}%
    \put(0.70677252,0.80655437){\color[rgb]{0,0,0}\makebox(0,0)[lt]{\begin{minipage}{0.32227353\unitlength}\raggedright $v_7$\end{minipage}}}%
    \put(0.21877423,0.48459894){\color[rgb]{0,0,0}\makebox(0,0)[lt]{\begin{minipage}{0.32227353\unitlength}\raggedright $d_1$\end{minipage}}}%
    \put(0.40921119,0.38335397){\color[rgb]{0,0,0}\makebox(0,0)[lt]{\begin{minipage}{0.32227353\unitlength}\raggedright $d_2$\end{minipage}}}%
    \put(0.57182113,0.39650701){\color[rgb]{0,0,0}\makebox(0,0)[lt]{\begin{minipage}{0.32227353\unitlength}\raggedright $d_3$\end{minipage}}}%
    \put(0.72982029,0.53763203){\color[rgb]{0,0,0}\makebox(0,0)[lt]{\begin{minipage}{0.32227353\unitlength}\raggedright $d_4$\end{minipage}}}%
    \put(0,0){\includegraphics[width=\unitlength,page=2]{fig_caterpillar.pdf}}%
  \end{picture}%
\endgroup%

%% file: fig_singularfiber.pdf_tex
\begingroup%
  \makeatletter%
  \providecommand\color[2][]{%
    \errmessage{(Inkscape) Color is used for the text in Inkscape, but the package 'color.sty' is not loaded}%
    \renewcommand\color[2][]{}%
  }%
  \providecommand\transparent[1]{%
    \errmessage{(Inkscape) Transparency is used (non-zero) for the text in Inkscape, but the package 'transparent.sty' is not loaded}%
    \renewcommand\transparent[1]{}%
  }%
  \providecommand\rotatebox[2]{#2}%
  \ifx\svgwidth\undefined%
    \setlength{\unitlength}{268.15360282bp}%
    \ifx\svgscale\undefined%
      \relax%
    \else%
      \setlength{\unitlength}{\unitlength * \real{\svgscale}}%
    \fi%
  \else%
    \setlength{\unitlength}{\svgwidth}%
  \fi%
  \global\let\svgwidth\undefined%
  \global\let\svgscale\undefined%
  \makeatother%
  \begin{picture}(1,0.45672172)%
    \put(0.28778962,0.25928528){\color[rgb]{0,0,0}\makebox(0,0)[lt]{\begin{minipage}{0.17900188\unitlength}\raggedright $v_1$\end{minipage}}}%
    \put(0.16348871,0.1849168){\color[rgb]{0,0,0}\makebox(0,0)[lt]{\begin{minipage}{0.17900188\unitlength}\raggedright $v_2$\end{minipage}}}%
    \put(0.20492299,0.05410165){\color[rgb]{0,0,0}\makebox(0,0)[lt]{\begin{minipage}{0.17900188\unitlength}\raggedright $v_3$\end{minipage}}}%
    \put(0.38888275,0.09435072){\color[rgb]{0,0,0}\makebox(0,0)[lt]{\begin{minipage}{0.17900188\unitlength}\raggedright $v_4$\end{minipage}}}%
    \put(0.36453369,0.29908019){\color[rgb]{0,0,0}\makebox(0,0)[lt]{\begin{minipage}{0.17900188\unitlength}\raggedright $v_5$\end{minipage}}}%
    \put(0.42593148,0.34925801){\color[rgb]{0,0,0}\makebox(0,0)[lt]{\begin{minipage}{0.17900188\unitlength}\raggedright $v_6$\end{minipage}}}%
    \put(0.6814815,0.04933332){\color[rgb]{0,0,0}\makebox(0,0)[lt]{\begin{minipage}{0.17900188\unitlength}\raggedright $v_8$\end{minipage}}}%
    \put(0,0){\includegraphics[width=\unitlength,page=1]{fig_singularfiber.pdf}}%
    \put(0.46941206,0.11506505){\color[rgb]{0,0,0}\makebox(0,0)[lt]{\begin{minipage}{0.17900188\unitlength}\raggedright $v_7$\end{minipage}}}%
    \put(0.49200141,0.32969961){\color[rgb]{0,0,0}\makebox(0,0)[lt]{\begin{minipage}{0.17900188\unitlength}\raggedright $v_{11}$\end{minipage}}}%
    \put(0.68789922,0.35353196){\color[rgb]{0,0,0}\makebox(0,0)[lt]{\begin{minipage}{0.17900188\unitlength}\raggedright $v_{10}$\end{minipage}}}%
    \put(0.79195093,0.22249472){\color[rgb]{0,0,0}\makebox(0,0)[lt]{\begin{minipage}{0.17900188\unitlength}\raggedright $v_9$\end{minipage}}}%
    \put(0,0){\includegraphics[width=\unitlength,page=2]{fig_singularfiber.pdf}}%
    \put(0.34403888,0.46025412){\color[rgb]{0,0,0}\makebox(0,0)[lt]{\begin{minipage}{0.08950094\unitlength}\raggedright $S^2$\end{minipage}}}%
    \put(0,0){\includegraphics[width=\unitlength,page=3]{fig_singularfiber.pdf}}%
    \put(-0.00123821,0.10631503){\color[rgb]{0,0,0}\makebox(0,0)[lt]{\begin{minipage}{0.17900188\unitlength}\raggedright $\mathrm{SO}(3)$\end{minipage}}}%
    \put(0,0){\includegraphics[width=\unitlength,page=4]{fig_singularfiber.pdf}}%
    \put(0.85644742,0.3386307){\color[rgb]{0,0,0}\makebox(0,0)[lt]{\begin{minipage}{0.17900188\unitlength}\raggedright $\mathrm{SO}(3)$\end{minipage}}}%
    \put(0,0){\includegraphics[width=\unitlength,page=5]{fig_singularfiber.pdf}}%
    \put(0.29653789,0.14881219){\color[rgb]{0,0,0}\makebox(0,0)[lt]{\begin{minipage}{0.13425141\unitlength}\raggedright $S^1$\end{minipage}}}%
    \put(0.58470481,0.16402893){\color[rgb]{0,0,0}\makebox(0,0)[lt]{\begin{minipage}{0.13425141\unitlength}\raggedright $S^1$\end{minipage}}}%
    \put(0.61038305,0.2524762){\color[rgb]{0,0,0}\makebox(0,0)[lt]{\begin{minipage}{0.13425141\unitlength}\raggedright $S^1$\end{minipage}}}%
  \end{picture}%
\endgroup%

%% file: fig_whiteheadmove.pdf_tex
\begingroup%
  \makeatletter%
  \providecommand\color[2][]{%
    \errmessage{(Inkscape) Color is used for the text in Inkscape, but the package 'color.sty' is not loaded}%
    \renewcommand\color[2][]{}%
  }%
  \providecommand\transparent[1]{%
    \errmessage{(Inkscape) Transparency is used (non-zero) for the text in Inkscape, but the package 'transparent.sty' is not loaded}%
    \renewcommand\transparent[1]{}%
  }%
  \providecommand\rotatebox[2]{#2}%
  \newcommand*\fsize{\dimexpr\f@size pt\relax}%
  \newcommand*\lineheight[1]{\fontsize{\fsize}{#1\fsize}\selectfont}%
  \ifx\svgwidth\undefined%
    \setlength{\unitlength}{212.03113141bp}%
    \ifx\svgscale\undefined%
      \relax%
    \else%
      \setlength{\unitlength}{\unitlength * \real{\svgscale}}%
    \fi%
  \else%
    \setlength{\unitlength}{\svgwidth}%
  \fi%
  \global\let\svgwidth\undefined%
  \global\let\svgscale\undefined%
  \makeatother%
  \begin{picture}(1,0.463321)%
    \lineheight{1}%
    \setlength\tabcolsep{0pt}%
    \put(0.23644804,0.46601164){\color[rgb]{0,0,0}\makebox(0,0)[lt]{\begin{minipage}{0.22638185\unitlength}\raggedright $a$\end{minipage}}}%
    \put(-0.00142778,0.16968463){\color[rgb]{0,0,0}\makebox(0,0)[lt]{\begin{minipage}{0.22638185\unitlength}\raggedright $b$\end{minipage}}}%
    \put(0.12353681,0.02053332){\color[rgb]{0,0,0}\makebox(0,0)[lt]{\begin{minipage}{0.22638185\unitlength}\raggedright $c$\end{minipage}}}%
    \put(0.32408483,0.02053332){\color[rgb]{0,0,0}\makebox(0,0)[lt]{\begin{minipage}{0.22638185\unitlength}\raggedright $d$\end{minipage}}}%
    \put(0,0){\includegraphics[width=\unitlength,page=1]{fig_whiteheadmove.pdf}}%
    \put(0.79532823,0.46601164){\color[rgb]{0,0,0}\makebox(0,0)[lt]{\begin{minipage}{0.22638185\unitlength}\raggedright $a$\end{minipage}}}%
    \put(0.55745241,0.16968463){\color[rgb]{0,0,0}\makebox(0,0)[lt]{\begin{minipage}{0.22638185\unitlength}\raggedright $b$\end{minipage}}}%
    \put(0.682417,0.02053332){\color[rgb]{0,0,0}\makebox(0,0)[lt]{\begin{minipage}{0.22638185\unitlength}\raggedright $c$\end{minipage}}}%
    \put(0.88296502,0.02053332){\color[rgb]{0,0,0}\makebox(0,0)[lt]{\begin{minipage}{0.22638185\unitlength}\raggedright $d$\end{minipage}}}%
    \put(0,0){\includegraphics[width=\unitlength,page=2]{fig_whiteheadmove.pdf}}%
  \end{picture}%
\endgroup%

%% file: fig_ladder.pdf_tex
\begingroup%
  \makeatletter%
  \providecommand\color[2][]{%
    \errmessage{(Inkscape) Color is used for the text in Inkscape, but the package 'color.sty' is not loaded}%
    \renewcommand\color[2][]{}%
  }%
  \providecommand\transparent[1]{%
    \errmessage{(Inkscape) Transparency is used (non-zero) for the text in Inkscape, but the package 'transparent.sty' is not loaded}%
    \renewcommand\transparent[1]{}%
  }%
  \providecommand\rotatebox[2]{#2}%
  \newcommand*\fsize{\dimexpr\f@size pt\relax}%
  \newcommand*\lineheight[1]{\fontsize{\fsize}{#1\fsize}\selectfont}%
  \ifx\svgwidth\undefined%
    \setlength{\unitlength}{360.96288968bp}%
    \ifx\svgscale\undefined%
      \relax%
    \else%
      \setlength{\unitlength}{\unitlength * \real{\svgscale}}%
    \fi%
  \else%
    \setlength{\unitlength}{\svgwidth}%
  \fi%
  \global\let\svgwidth\undefined%
  \global\let\svgscale\undefined%
  \makeatother%
  \begin{picture}(1,0.21900421)%
    \lineheight{1}%
    \setlength\tabcolsep{0pt}%
    \put(0.08227235,0.16598452){\color[rgb]{0,0,0}\makebox(0,0)[lt]{\begin{minipage}{0.13297766\unitlength}\raggedright $v_{1,2}$\end{minipage}}}%
    \put(0,0){\includegraphics[width=\unitlength,page=1]{fig_ladder.pdf}}%
    \put(0.16538339,0.16598452){\color[rgb]{0,0,0}\makebox(0,0)[lt]{\begin{minipage}{0.13297766\unitlength}\raggedright $v_{2,2}$\end{minipage}}}%
    \put(0.24849442,0.16598452){\color[rgb]{0,0,0}\makebox(0,0)[lt]{\begin{minipage}{0.13297766\unitlength}\raggedright $v_{3,2}$\end{minipage}}}%
    \put(0.08227235,0.08287348){\color[rgb]{0,0,0}\makebox(0,0)[lt]{\begin{minipage}{0.13297766\unitlength}\raggedright $v_{1,1}$\end{minipage}}}%
    \put(0.16538339,0.08287348){\color[rgb]{0,0,0}\makebox(0,0)[lt]{\begin{minipage}{0.13297766\unitlength}\raggedright $v_{2,1}$\end{minipage}}}%
    \put(0.24849442,0.08287348){\color[rgb]{0,0,0}\makebox(0,0)[lt]{\begin{minipage}{0.13297766\unitlength}\raggedright $v_{3,1}$\end{minipage}}}%
    \put(-0.00083869,0.16598452){\color[rgb]{0,0,0}\makebox(0,0)[lt]{\begin{minipage}{0.13297766\unitlength}\raggedright $v_{0,2}$\end{minipage}}}%
    \put(0.33576101,0.08287348){\color[rgb]{0,0,0}\makebox(0,0)[lt]{\begin{minipage}{0.13297766\unitlength}\raggedright $v_{4,1}$\end{minipage}}}%
    \put(0,0){\includegraphics[width=\unitlength,page=2]{fig_ladder.pdf}}%
    \put(0.11551677,0.0122291){\color[rgb]{0,0,0}\makebox(0,0)[lt]{\begin{minipage}{0.40317578\unitlength}\raggedright $\Gamma_{\lambda}$ and $\Gamma^\vee_{\lambda}$ \end{minipage}}}%
    \put(0.60171633,0.16598452){\color[rgb]{0,0,0}\makebox(0,0)[lt]{\begin{minipage}{0.13297766\unitlength}\raggedright $v_{1,2}$\end{minipage}}}%
    \put(0,0){\includegraphics[width=\unitlength,page=3]{fig_ladder.pdf}}%
    \put(0.68482737,0.16598452){\color[rgb]{0,0,0}\makebox(0,0)[lt]{\begin{minipage}{0.13297766\unitlength}\raggedright $v_{2,2}$\end{minipage}}}%
    \put(0.68482737,0.08287348){\color[rgb]{0,0,0}\makebox(0,0)[lt]{\begin{minipage}{0.13297766\unitlength}\raggedright $v_{2,1}$\end{minipage}}}%
    \put(0.7679384,0.08287348){\color[rgb]{0,0,0}\makebox(0,0)[lt]{\begin{minipage}{0.13297766\unitlength}\raggedright $v_{3,1}$\end{minipage}}}%
    \put(0.51860529,0.16598452){\color[rgb]{0,0,0}\makebox(0,0)[lt]{\begin{minipage}{0.13297766\unitlength}\raggedright $v_{0,2}$\end{minipage}}}%
    \put(0.85520499,0.08287348){\color[rgb]{0,0,0}\makebox(0,0)[lt]{\begin{minipage}{0.13297766\unitlength}\raggedright $v_{4,1}$\end{minipage}}}%
    \put(0,0){\includegraphics[width=\unitlength,page=4]{fig_ladder.pdf}}%
    \put(0.64327185,0.0122291){\color[rgb]{0,0,0}\makebox(0,0)[lt]{\begin{minipage}{0.40317578\unitlength}\raggedright $\Gamma_{\mathbf{r}}$ and $\Gamma^\vee_{\mathbf{r}}$ \end{minipage}}}%
  \end{picture}%
\endgroup%

%% file: fig_deltak.pdf_tex
\begingroup%
  \makeatletter%
  \providecommand\color[2][]{%
    \errmessage{(Inkscape) Color is used for the text in Inkscape, but the package 'color.sty' is not loaded}%
    \renewcommand\color[2][]{}%
  }%
  \providecommand\transparent[1]{%
    \errmessage{(Inkscape) Transparency is used (non-zero) for the text in Inkscape, but the package 'transparent.sty' is not loaded}%
    \renewcommand\transparent[1]{}%
  }%
  \providecommand\rotatebox[2]{#2}%
  \newcommand*\fsize{\dimexpr\f@size pt\relax}%
  \newcommand*\lineheight[1]{\fontsize{\fsize}{#1\fsize}\selectfont}%
  \ifx\svgwidth\undefined%
    \setlength{\unitlength}{218.45112206bp}%
    \ifx\svgscale\undefined%
      \relax%
    \else%
      \setlength{\unitlength}{\unitlength * \real{\svgscale}}%
    \fi%
  \else%
    \setlength{\unitlength}{\svgwidth}%
  \fi%
  \global\let\svgwidth\undefined%
  \global\let\svgscale\undefined%
  \makeatother%
  \begin{picture}(1,0.39257382)%
    \lineheight{1}%
    \setlength\tabcolsep{0pt}%
    \put(0.35826606,0.01992988){\color[rgb]{0,0,0}\makebox(0,0)[lt]{\begin{minipage}{0.21972879\unitlength}\raggedright $\delta_1$\end{minipage}}}%
    \put(0,0){\includegraphics[width=\unitlength,page=1]{fig_deltak.pdf}}%
    \put(0.49559655,0.01992988){\color[rgb]{0,0,0}\makebox(0,0)[lt]{\begin{minipage}{0.21972879\unitlength}\raggedright $\delta_2$\end{minipage}}}%
    \put(0.63292704,0.01992988){\color[rgb]{0,0,0}\makebox(0,0)[lt]{\begin{minipage}{0.21972879\unitlength}\raggedright $\delta_3$\end{minipage}}}%
    \put(0.9007215,0.01992988){\color[rgb]{0,0,0}\makebox(0,0)[lt]{\begin{minipage}{0.21972879\unitlength}\raggedright $\delta_5$\end{minipage}}}%
    \put(0,0){\includegraphics[width=\unitlength,page=2]{fig_deltak.pdf}}%
    \put(0.77712406,0.01992988){\color[rgb]{0,0,0}\makebox(0,0)[lt]{\begin{minipage}{0.21972879\unitlength}\raggedright $\delta_4$\end{minipage}}}%
    \put(0,0){\includegraphics[width=\unitlength,page=3]{fig_deltak.pdf}}%
  \end{picture}%
\endgroup%

%% file: fig_local.pdf_tex
\begingroup%
  \makeatletter%
  \providecommand\color[2][]{%
    \errmessage{(Inkscape) Color is used for the text in Inkscape, but the package 'color.sty' is not loaded}%
    \renewcommand\color[2][]{}%
  }%
  \providecommand\transparent[1]{%
    \errmessage{(Inkscape) Transparency is used (non-zero) for the text in Inkscape, but the package 'transparent.sty' is not loaded}%
    \renewcommand\transparent[1]{}%
  }%
  \providecommand\rotatebox[2]{#2}%
  \newcommand*\fsize{\dimexpr\f@size pt\relax}%
  \newcommand*\lineheight[1]{\fontsize{\fsize}{#1\fsize}\selectfont}%
  \ifx\svgwidth\undefined%
    \setlength{\unitlength}{644.78989089bp}%
    \ifx\svgscale\undefined%
      \relax%
    \else%
      \setlength{\unitlength}{\unitlength * \real{\svgscale}}%
    \fi%
  \else%
    \setlength{\unitlength}{\svgwidth}%
  \fi%
  \global\let\svgwidth\undefined%
  \global\let\svgscale\undefined%
  \makeatother%
  \begin{picture}(1,0.15220488)%
    \lineheight{1}%
    \setlength\tabcolsep{0pt}%
    \put(0,0){\includegraphics[width=\unitlength,page=1]{fig_local.pdf}}%
    \put(0.21366702,0.08723715){\color[rgb]{0,0,0}\makebox(0,0)[lt]{\lineheight{1.25}\smash{\begin{tabular}[t]{l}$\mathbb{S}^2$\end{tabular}}}}%
    \put(0.16714024,0.05932108){\color[rgb]{0,0,0}\makebox(0,0)[lt]{\lineheight{1.25}\smash{\begin{tabular}[t]{l}$\mathbf{w}$\end{tabular}}}}%
    \put(0.61211982,0.08111556){\color[rgb]{0,0,0}\makebox(0,0)[lt]{\lineheight{1.25}\smash{\begin{tabular}[t]{l}$\mathrm{SO}(3)$\end{tabular}}}}%
    \put(0,0){\includegraphics[width=\unitlength,page=2]{fig_local.pdf}}%
    \put(-0.00352131,0.07327968){\color[rgb]{0,0,0}\makebox(0,0)[lt]{\lineheight{1.25}\smash{\begin{tabular}[t]{l}$\,$\end{tabular}}}}%
    \put(0,0){\includegraphics[width=\unitlength,page=3]{fig_local.pdf}}%
    \put(0.21134072,0.05932108){\color[rgb]{0,0,0}\makebox(0,0)[lt]{\lineheight{1.25}\smash{\begin{tabular}[t]{l}$\mathbf{u}_{\mathbb{S}^2}$\end{tabular}}}}%
    \put(0.6137974,0.0546684){\color[rgb]{0,0,0}\makebox(0,0)[lt]{\lineheight{1.25}\smash{\begin{tabular}[t]{l}$\mathbf{u}_{\mathrm{SO}(3)}$\end{tabular}}}}%
    \put(0,0){\includegraphics[width=\unitlength,page=4]{fig_local.pdf}}%
    \put(0.56261794,0.0546684){\color[rgb]{0,0,0}\makebox(0,0)[lt]{\lineheight{1.25}\smash{\begin{tabular}[t]{l}$\mathbf{w}$\end{tabular}}}}%
  \end{picture}%
\endgroup%

%% file: fig_map.pdf_tex
\begingroup%
  \makeatletter%
  \providecommand\color[2][]{%
    \errmessage{(Inkscape) Color is used for the text in Inkscape, but the package 'color.sty' is not loaded}%
    \renewcommand\color[2][]{}%
  }%
  \providecommand\transparent[1]{%
    \errmessage{(Inkscape) Transparency is used (non-zero) for the text in Inkscape, but the package 'transparent.sty' is not loaded}%
    \renewcommand\transparent[1]{}%
  }%
  \providecommand\rotatebox[2]{#2}%
  \newcommand*\fsize{\dimexpr\f@size pt\relax}%
  \newcommand*\lineheight[1]{\fontsize{\fsize}{#1\fsize}\selectfont}%
  \ifx\svgwidth\undefined%
    \setlength{\unitlength}{246.21483848bp}%
    \ifx\svgscale\undefined%
      \relax%
    \else%
      \setlength{\unitlength}{\unitlength * \real{\svgscale}}%
    \fi%
  \else%
    \setlength{\unitlength}{\svgwidth}%
  \fi%
  \global\let\svgwidth\undefined%
  \global\let\svgscale\undefined%
  \makeatother%
  \begin{picture}(1,0.29941776)%
    \lineheight{1}%
    \setlength\tabcolsep{0pt}%
    \put(0.08736735,0.28452173){\color[rgb]{0,0,0}\makebox(0,0)[lt]{\begin{minipage}{0.19495169\unitlength}\raggedright $v_1$\end{minipage}}}%
    \put(-0.00122955,0.1959758){\color[rgb]{0,0,0}\makebox(0,0)[lt]{\begin{minipage}{0.19495169\unitlength}\raggedright $v_2$\end{minipage}}}%
    \put(0.00399138,0.04347617){\color[rgb]{0,0,0}\makebox(0,0)[lt]{\begin{minipage}{0.19495169\unitlength}\raggedright $v_3$\end{minipage}}}%
    \put(0.16635983,0.02834115){\color[rgb]{0,0,0}\makebox(0,0)[lt]{\begin{minipage}{0.19495169\unitlength}\raggedright $v_4$\end{minipage}}}%
    \put(0.26400975,0.06055157){\color[rgb]{0,0,0}\makebox(0,0)[lt]{\begin{minipage}{0.19495169\unitlength}\raggedright $v_5$\end{minipage}}}%
    \put(0.39694044,0.09389482){\color[rgb]{0,0,0}\makebox(0,0)[lt]{\begin{minipage}{0.19495169\unitlength}\raggedright $v_6$\end{minipage}}}%
    \put(0.056841,0.14275944){\color[rgb]{0,0,0}\makebox(0,0)[lt]{\begin{minipage}{0.19495169\unitlength}\raggedright $d_1$\end{minipage}}}%
    \put(0.14210442,0.12442636){\color[rgb]{0,0,0}\makebox(0,0)[lt]{\begin{minipage}{0.19495169\unitlength}\raggedright $d_2$\end{minipage}}}%
    \put(0.23075944,0.19048655){\color[rgb]{0,0,0}\makebox(0,0)[lt]{\begin{minipage}{0.19495169\unitlength}\raggedright $d_3$\end{minipage}}}%
    \put(0.31168029,0.15629685){\color[rgb]{0,0,0}\makebox(0,0)[lt]{\begin{minipage}{0.19495169\unitlength}\raggedright $d_4$\end{minipage}}}%
    \put(0,0){\includegraphics[width=\unitlength,page=1]{fig_map.pdf}}%
    \put(0.38942525,0.22053732){\color[rgb]{0,0,0}\makebox(0,0)[lt]{\begin{minipage}{0.19495169\unitlength}\raggedright $v_7$\end{minipage}}}%
    \put(0.26482776,0.30173485){\color[rgb]{0,0,0}\makebox(0,0)[lt]{\begin{minipage}{0.19495169\unitlength}\raggedright $v_8$\end{minipage}}}%
    \put(0.29666736,0.23806708){\color[rgb]{0,0,0}\makebox(0,0)[lt]{\begin{minipage}{0.19495169\unitlength}\raggedright $d_5$\end{minipage}}}%
    \put(0.80625172,0.28452173){\color[rgb]{0,0,0}\makebox(0,0)[lt]{\begin{minipage}{0.19495169\unitlength}\raggedright $v_1$\end{minipage}}}%
    \put(0.71765482,0.1959758){\color[rgb]{0,0,0}\makebox(0,0)[lt]{\begin{minipage}{0.19495169\unitlength}\raggedright $v_2$\end{minipage}}}%
    \put(0.72287575,0.04347617){\color[rgb]{0,0,0}\makebox(0,0)[lt]{\begin{minipage}{0.19495169\unitlength}\raggedright $v_3$\end{minipage}}}%
    \put(0.8852442,0.02834115){\color[rgb]{0,0,0}\makebox(0,0)[lt]{\begin{minipage}{0.19495169\unitlength}\raggedright $v_4$\end{minipage}}}%
    \put(0.77572537,0.14275944){\color[rgb]{0,0,0}\makebox(0,0)[lt]{\begin{minipage}{0.19495169\unitlength}\raggedright $d_1$\end{minipage}}}%
    \put(0.86098879,0.12442636){\color[rgb]{0,0,0}\makebox(0,0)[lt]{\begin{minipage}{0.19495169\unitlength}\raggedright $d_2$\end{minipage}}}%
    \put(0.94964383,0.19048655){\color[rgb]{0,0,0}\makebox(0,0)[lt]{\begin{minipage}{0.19495169\unitlength}\raggedright $d_3$\end{minipage}}}%
    \put(0,0){\includegraphics[width=\unitlength,page=2]{fig_map.pdf}}%
    \put(0.55391574,0.17789163){\color[rgb]{0,0,0}\makebox(0,0)[lt]{\begin{minipage}{0.19495169\unitlength}\raggedright $\Upsilon$\end{minipage}}}%
    \put(0,0){\includegraphics[width=\unitlength,page=3]{fig_map.pdf}}%
  \end{picture}%
\endgroup%

%% file: fig_momentpolytope.pdf_tex
\begingroup%
  \makeatletter%
  \providecommand\color[2][]{%
    \errmessage{(Inkscape) Color is used for the text in Inkscape, but the package 'color.sty' is not loaded}%
    \renewcommand\color[2][]{}%
  }%
  \providecommand\transparent[1]{%
    \errmessage{(Inkscape) Transparency is used (non-zero) for the text in Inkscape, but the package 'transparent.sty' is not loaded}%
    \renewcommand\transparent[1]{}%
  }%
  \providecommand\rotatebox[2]{#2}%
  \newcommand*\fsize{\dimexpr\f@size pt\relax}%
  \newcommand*\lineheight[1]{\fontsize{\fsize}{#1\fsize}\selectfont}%
  \ifx\svgwidth\undefined%
    \setlength{\unitlength}{753.78808397bp}%
    \ifx\svgscale\undefined%
      \relax%
    \else%
      \setlength{\unitlength}{\unitlength * \real{\svgscale}}%
    \fi%
  \else%
    \setlength{\unitlength}{\svgwidth}%
  \fi%
  \global\let\svgwidth\undefined%
  \global\let\svgscale\undefined%
  \makeatother%
  \begin{picture}(1,0.12421987)%
    \lineheight{1}%
    \setlength\tabcolsep{0pt}%
    \put(0,0){\includegraphics[width=\unitlength,page=1]{fig_momentpolytope.pdf}}%
    \put(0.09853295,0.00701535){\color[rgb]{0,0,0}\makebox(0,0)[lt]{\lineheight{1.25}\smash{\begin{tabular}[t]{l}$A.1$\end{tabular}}}}%
    \put(0.2099701,0.00701535){\color[rgb]{0,0,0}\makebox(0,0)[lt]{\lineheight{1.25}\smash{\begin{tabular}[t]{l}$A.2$\end{tabular}}}}%
    \put(0.32140726,0.00701535){\color[rgb]{0,0,0}\makebox(0,0)[lt]{\lineheight{1.25}\smash{\begin{tabular}[t]{l}$A.3$\end{tabular}}}}%
    \put(0,0){\includegraphics[width=\unitlength,page=2]{fig_momentpolytope.pdf}}%
    \put(0.43875659,0.0692902){\color[rgb]{0,0,0}\makebox(0,0)[lt]{\lineheight{1.25}\smash{\begin{tabular}[t]{l}$\mathbb{S}^2$-fiber\end{tabular}}}}%
    \put(0.43776162,0.04541081){\color[rgb]{0,0,0}\makebox(0,0)[lt]{\lineheight{1.25}\smash{\begin{tabular}[t]{l}The center $\mathbf{u}_0$\end{tabular}}}}%
    \put(-0.00301213,0.05834548){\color[rgb]{0,0,0}\makebox(0,0)[lt]{\lineheight{1.25}\smash{\begin{tabular}[t]{l}$\,$\end{tabular}}}}%
  \end{picture}%
\endgroup%